\documentclass[11pt]{article}
\usepackage[T1]{fontenc}
\usepackage[utf8]{inputenc}
\usepackage{amsmath,amssymb,euscript}
\usepackage{bbm}
\usepackage{graphicx}
\usepackage{epic}
\usepackage{epsfig}
\usepackage{pstricks}
\usepackage{psfrag}
 \usepackage{curves}
 \usepackage{picture}
\usepackage{mathrsfs}
\usepackage{rotating}
\usepackage{color}
\usepackage{hyperref}

\newtheorem{thm}{Theorem}[section]
\newtheorem{lem}[thm]{Lemma}
\newtheorem{cor}[thm]{Corollary}

\newtheorem{rem}[thm]{Remark}

\newenvironment{proof}{\noindent {\bf Proof \phantom{9}}}
{\hfill $\square$ \vspace{0.25cm}}

\def\be{\begin{eqnarray}}
\def\ee{\end{eqnarray}}
\def\ben{\begin{eqnarray*}}
\def\een{\end{eqnarray*}}

\numberwithin{equation}{section}
\numberwithin{figure}{section}

\def\be{\begin{eqnarray}}
\def\ee{\end{eqnarray}}

\newcommand{\RR}{\mathbb{R}}

\newcommand{\PP}{\mathbb{P}}

\def\me{\medskip\noindent}

\def\N{\mathbb{N}}
\def\P{\mathbb{P}}
\def\R{\mathbb{R}}
\def\E{\mathbb{E}}

\def\ind{{\mathchoice {\rm 1\mskip-4mu l} {\rm 1\mskip-4mu l}
{\rm 1\mskip-4.5mu l} {\rm 1\mskip-5mu l}}}

\def\11{\mathbbm{1}}

\title{Stochastic analysis of emergence of evolutionary cyclic behavior in population dynamics with transfer}

\author{Nicolas Champagnat\thanks{Universit\'e de Lorraine, CNRS, Inria, IECL, F-54000 Nancy, France; E-mail: \texttt{nicolas.champagnat@inria.fr}},  Sylvie
    M\'el\'eard\thanks{Ecole Polytechnique, CNRS, IUF, CMAP, route de
    Saclay, 91128 Palaiseau Cedex-France; E-mail: \texttt{sylvie.meleard@polytechnique.edu}}, Viet Chi Tran\thanks{LAMA, Univ Gustave Eiffel, UPEM, Univ Paris Est Creteil, CNRS, F-77447, Marne-la-Vallée, France; E-mail:
    \texttt{chi.tran@u-pem.fr}}}

\date{\today}

\begin{document}

\maketitle

\begin{abstract}
Horizontal gene transfer consists in exchanging genetic materials between microorganisms during their lives. This is a major mechanism of bacterial evolution and is believed to be of main importance in antibiotics resistance. We consider a stochastic model for the evolution of a discrete population structured by a trait taking finitely many values, with density-dependent competition. Traits are vertically inherited unless a mutation occurs, and can also be horizontally transferred by unilateral conjugation with frequency dependent rate. Our goal is to analyze the trade-off between natural evolution to higher birth rates on one side, and transfer which drives the population towards lower birth rates on the other side. Simulations show that evolutionary outcomes include evolutionary suicide or cyclic re-emergence of small populations with well-adapted traits. We focus on a parameter scaling where individual mutations are rare but the global mutation rate tends to infinity. This implies that negligible sub-populations may have a strong contribution to evolution. Our main result quantifies the asymptotic dynamics of subpopulation sizes on a logarithmic scale. We characterize the possible evolutionary outcomes with explicit criteria on the model parameters. An important ingredient for the proofs lies in comparisons of the stochastic population process with linear or logistic birth-death processes with immigration. For the latter processes, we derive several results of independent interest.\end{abstract}

\me Keywords: horizontal gene transfer, bacterial conjugation, stochastic individual-based models, long time behavior, large population approximation, coupling, branching processes with immigration, logistic competition.

\bigskip

% \tableofcontents

\section{Introduction and presentation of the model}
\label{sec:introduction}

Bacterial evolution understanding is fundamental in biology, medicine and industry. The ability of a bacterium to survive and reproduce depends on its genes, and evolution mainly results from the
following basic mechanisms: heredity (also called vertical transmission);
mutation which generates variability of the traits; selection which results from the
interactions between individuals and their environment; exchange of genetic information between non-parental individuals during their
lifetimes (also called horizontal gene transfer (HGT)), see for example \cite{Ochman2000}, \cite{Keeling2008}. In many biological situations, these mechanisms drive the population to different evolutionary outcomes: directional evolution, re-emergence of apparently extinct traits or extinction of the population. For example, in antibiotic resistance, cyclic re-emergence of resistant strains are observed, while evolutionary suicide may correspond to successful antibiotic treatments~\cite{billiardcolletferrieremeleardtran2}. Such cyclic evolutionary dynamics and evolutionary suicide are due to an intricate interplay between HGT and selection and are therefore of a different nature from evolutionary behaviors observed in dynamical systems as prey-predator systems~\cite{dercoleferrieregragnanirinaldi,gyllenbergparvinen}. Usually, genes responsible for pathogens or antibiotic resistances  are carried by small DNA chains called \emph{plasmids} that can be exchanged between bacteria by HGT.
 Plasmids HGT also play a key role in other biological contexts, including transmission of an epidemic, epigenetics or bacterial degradation of novel compounds such as human-created
pesticides (e.g. \cite{Stewart1979}, \cite{Ginty2013}, \cite{Getino2015}).

In
\cite{billiardcolletferrieremeleardtran} and \cite{billiardcolletferrieremeleardtran2}, the authors introduced an
individual-based stochastic process for the trade-off between competition, transfer and advantageous mutations from which they derived some macroscopic approximations. They proved that the whole population can be driven (by transfer) to evolutionary suicide, under the assumption that mutations are very rare. However, simulations show much richer evolutionary behaviors when mutation events are more frequent.

We propose below a toy model to capture these phenomena. Up to our knowledge, this is the first evolutionary model involving specific mutation scales allowing to recover all these phenomena. The simulations of this model are shown in Figure~\ref{fig:simu}. We observe, depending on the
transfer rate, either {dominance} of the {trait with higher
  birth rate},  or a cyclic phenomenon, or evolutionary suicide.
In this model, the biological assumptions are as follows. We consider a large population of bacteria, characterized by a phenotypic value quantifying the macroscopic effect of plasmids on the demographic parameters. For example, the phenotype may describe the pathogenic strength or the antibiotic resistance~\cite{billiardcolletferrieremeleardtran2}.
When a transfer happens, a recipient bacterium is chosen uniformly at random in the population (frequency-dependence), as observed by biologists for large populations~\cite{raul}.
During the transfer, the quantity of plasmids in the recipient bacterium increases. The recipient bacterium receives the donor trait (this is called \emph{conjugation} in the biological
setting~\cite{billiardcolletferrieremeleardtran}). A large quantity of plasmids induces a large reproductive cost which slows down cell division~\cite{baltrus}, so that the division rate is a decreasing function of the trait.
The density-dependence in the competition death rate is uniform over the trait space.

\begin{figure}[!ht]
\begin{center}
\hspace{-1.5cm}
\begin{tabular}{cccc}
\includegraphics[width=4cm,height=4cm]{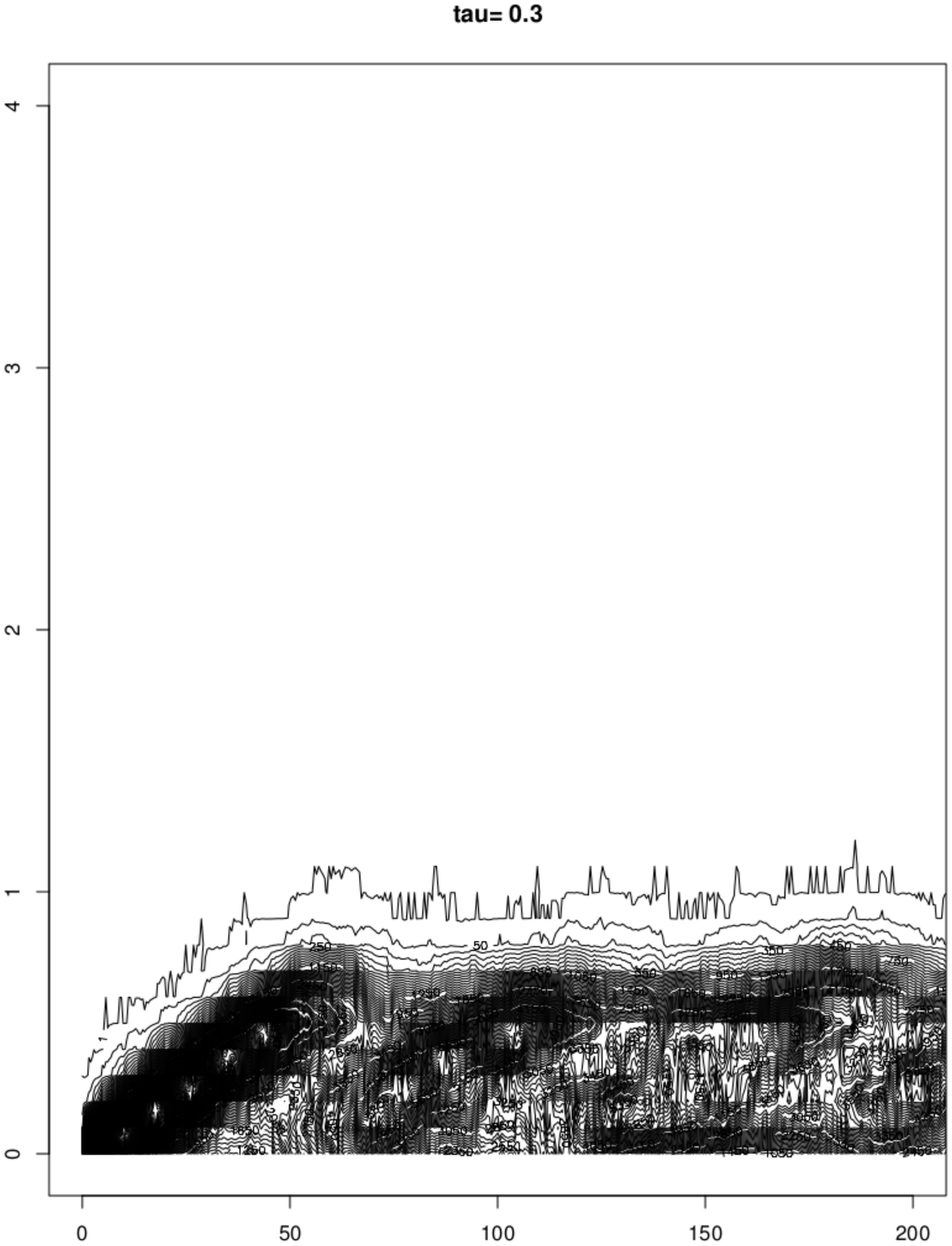} &
\includegraphics[width=4cm,height=4cm]{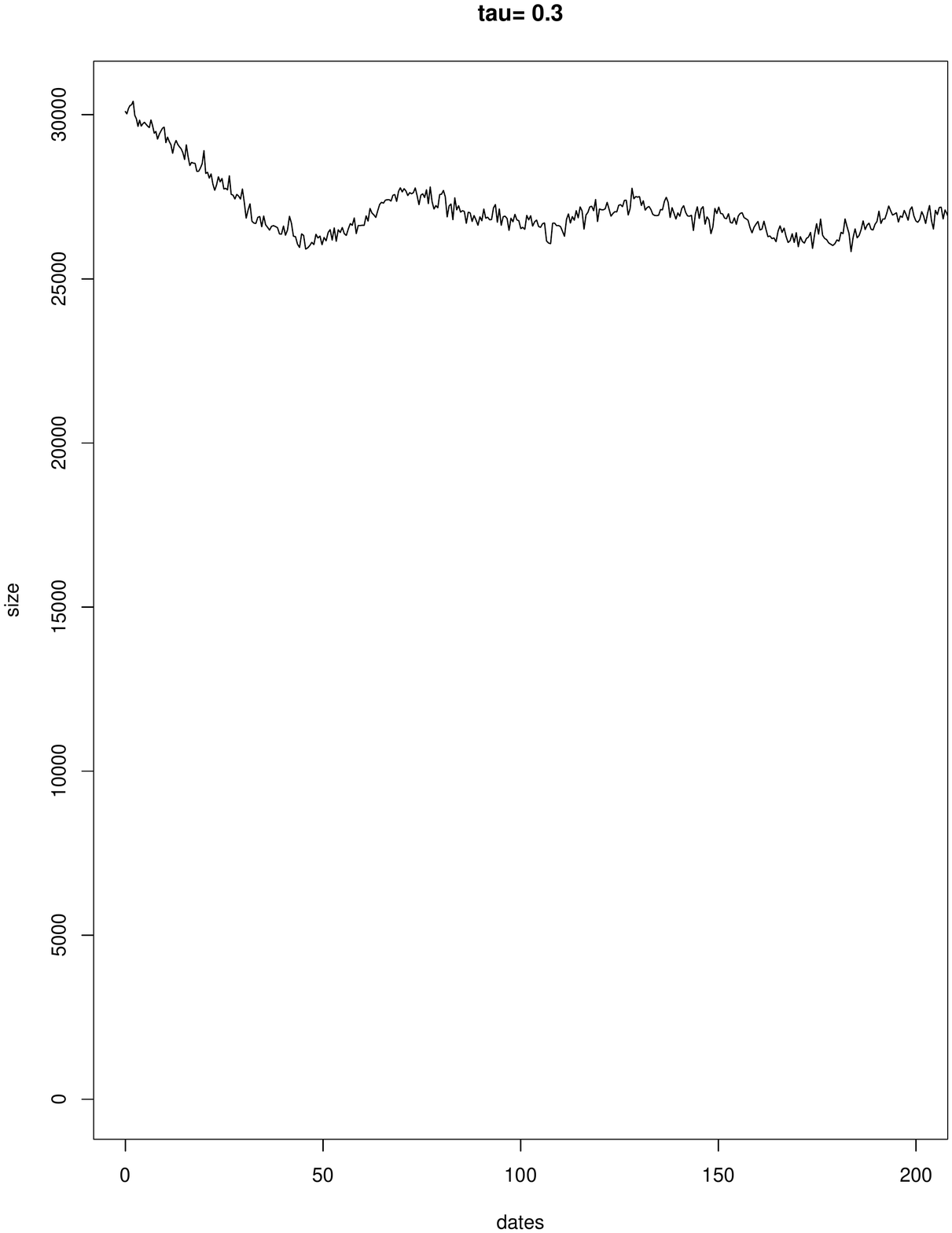} & \includegraphics[width=4cm,height=4cm]{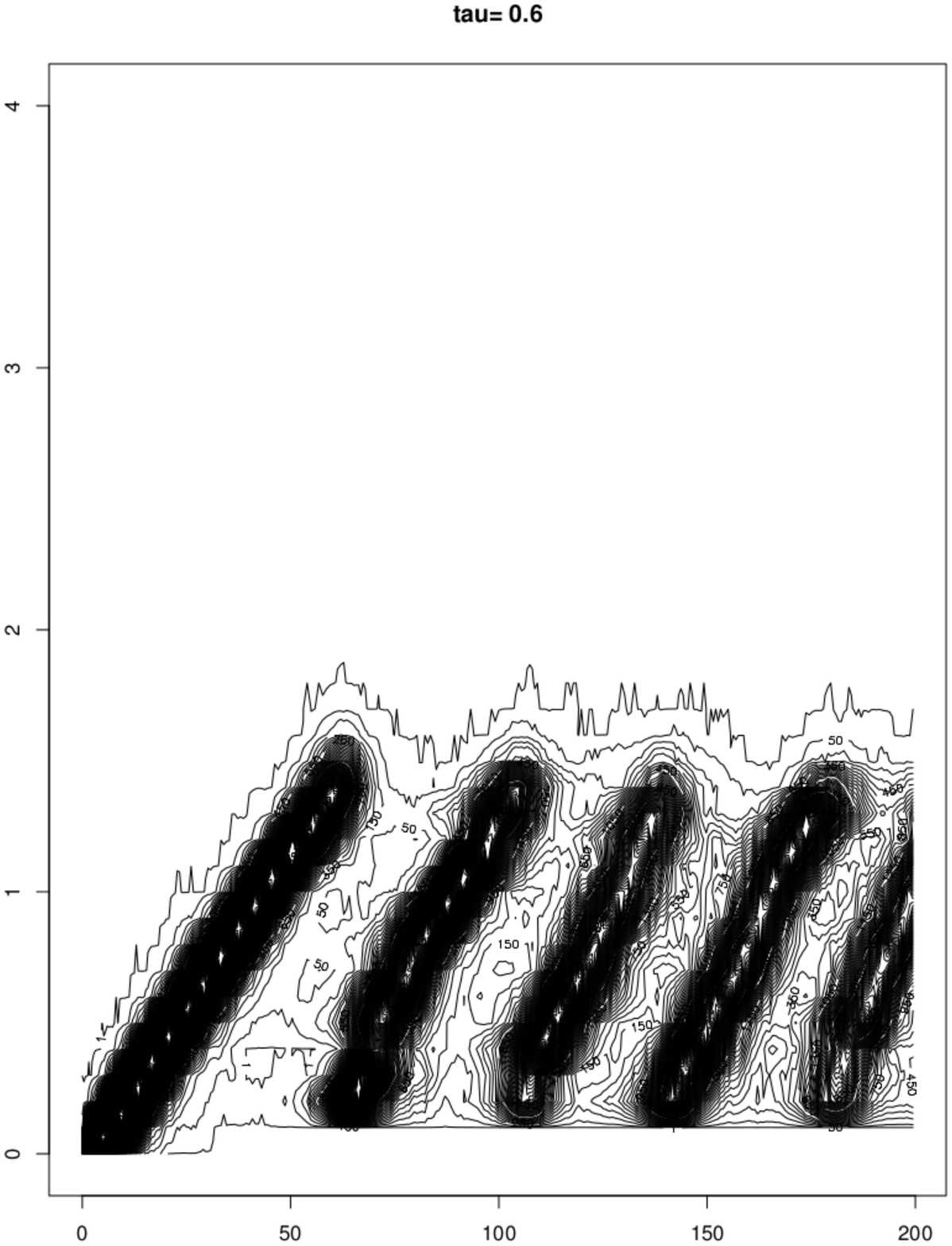} &
\includegraphics[width=4cm,height=4cm]{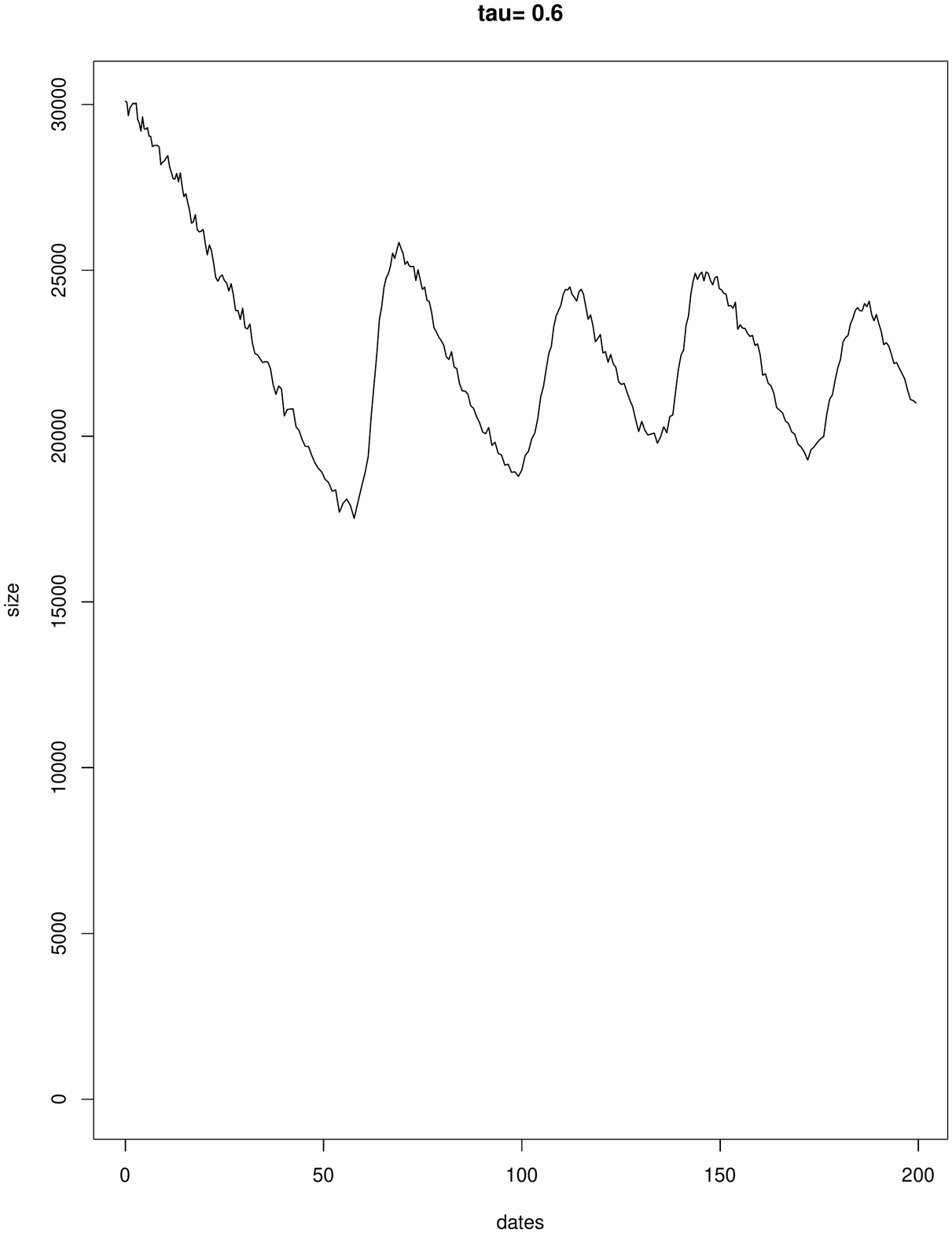}  \\
(a) & (b) & (c) & (d)\\
\includegraphics[width=4cm,height=4cm]{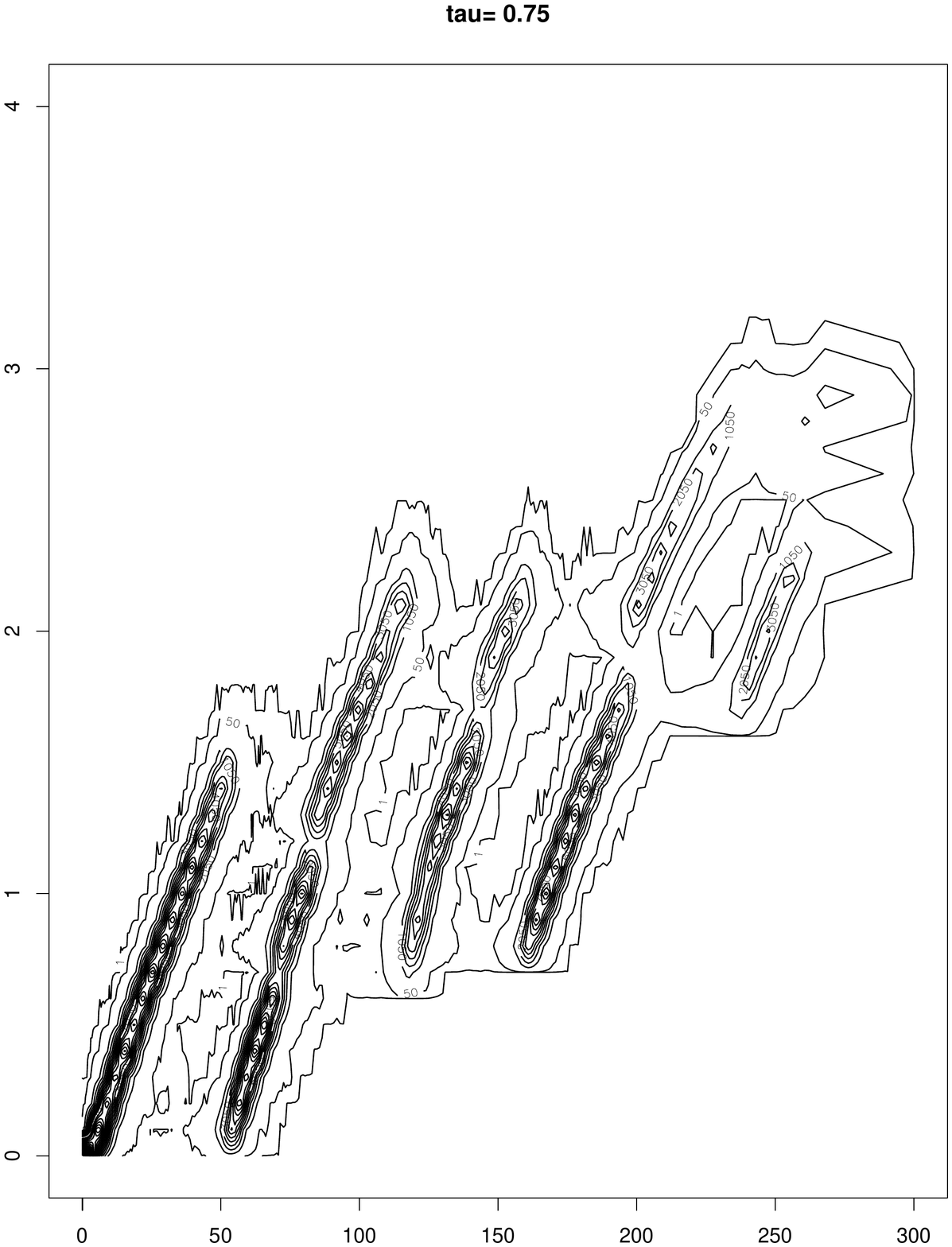} &
\includegraphics[width=4cm,height=4cm]{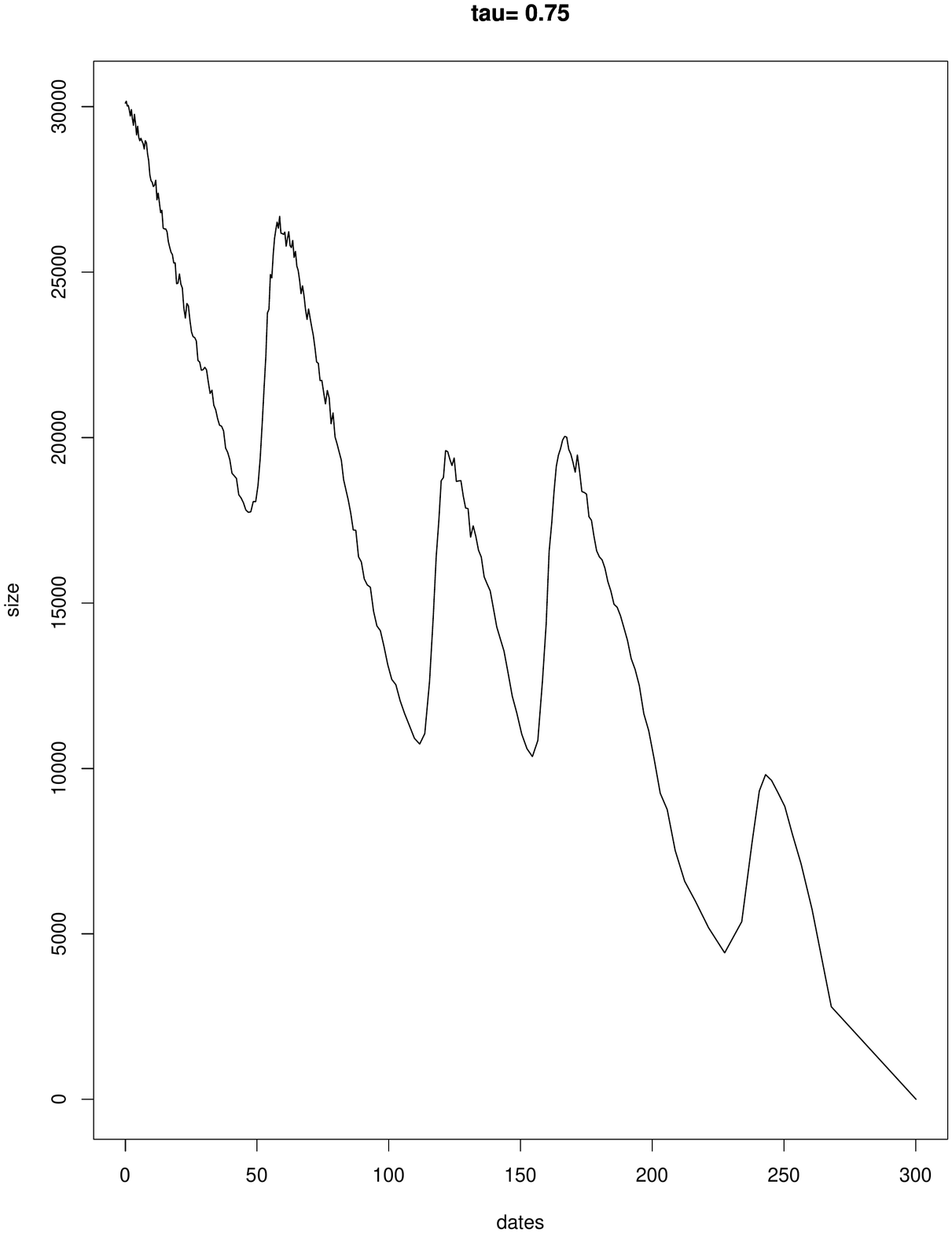} &
\includegraphics[width=4cm,height=4cm]{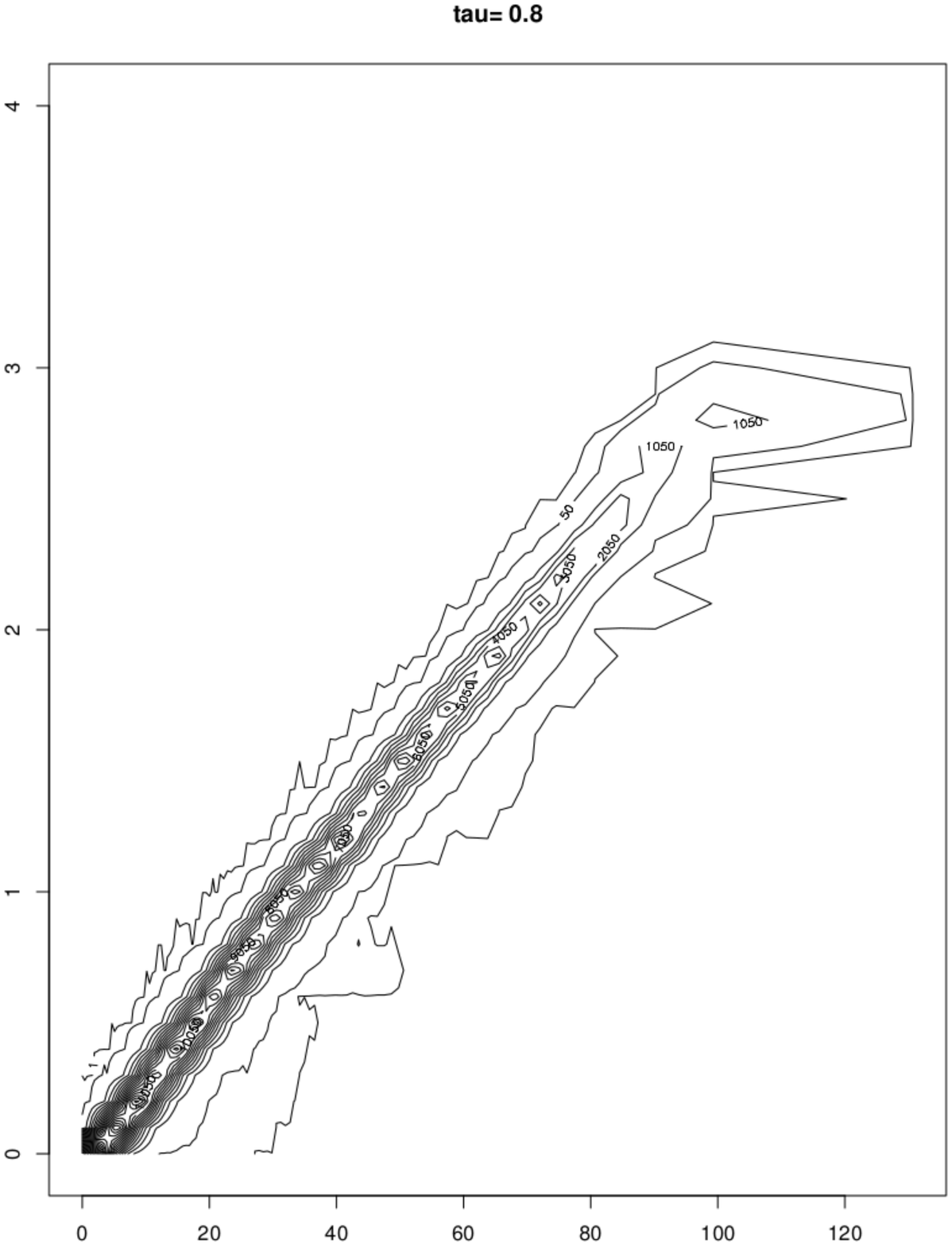} &
\includegraphics[width=4cm,height=4cm]{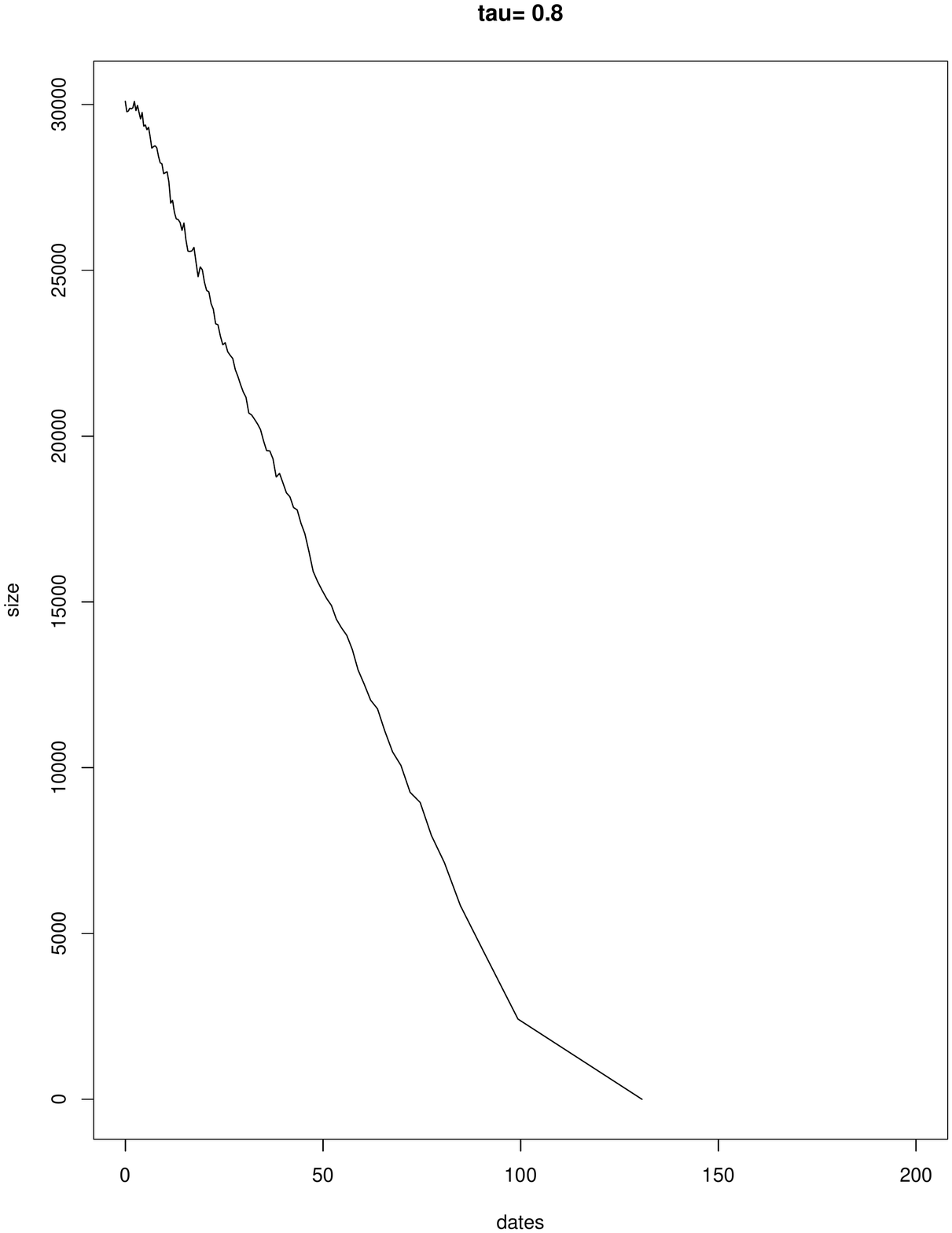} \\
(e) & (f) & (g) & (h)
\end{tabular}
\caption{{\small \textit{Simulations of eco-evolutionary dynamics with unilateral frequency dependent trait transfer. The evolution
      of the trait distribution is pictured on the left columns, the evolution of the population sizes $(N_0^K(t),\ldots,N^K_L(t))$ is
      on the right columns. In all the simulations, $K=10 000$ and $\delta=0.1$ and $\alpha=0.5$. (a)-(b): $\tau=0.3$. Smaller
      transfer rates are considered on longer time windows. (c)-(d): $\tau=0.6$. We see cyclic re-emergences of the fittest traits.
      (e)-(f): $\tau=0.75$. Re-emergence still occurs, but the higher transfer rate drives the trait distribution towards higher and
      less fit trait values. On stochastic simulations, this can lead to extinction or (temporary) cyclic behaviour. (g)-(h):
      $\tau=0.8$. An evolutionary suicide takes place. }}}\label{fig:simu}
\end{center}
\end{figure}

\bigskip

We consider a stochastic discrete population of individuals characterized by some trait. The population evolves in continuous time
through births (with or without mutation), deaths and transfer of traits between pairs of individuals. More precisely, we study a
continuous time Markov pure jump process, scaled by some integer parameter $K$, often called carrying capacity. The
initial population size is of the order of magnitude of $K$ for large $K$. The trait space is the grid of mesh $\delta>0$ of $[0,4]$:
$\mathcal{X}=[0,4]\cap \delta \N=\{0,\delta,\dots, L \delta\}$ where $L=\lfloor 4/\delta\rfloor$. The population is described by the vector
\[
% N^K(t)=
\big(N_0^K(t),\dots, N_\ell^K(t),\dots, N_{L}^K(t)\big)
\]
where $N_\ell^K(t)$ is the number of individuals of trait $x=\ell\delta$ at time $t$. We define the total population size $N^K_t$ as
\[
N^K_t=\sum_{\ell=0}^L N^K_\ell(t).
\]

\me Let us now describe the dynamics of the population process.
\begin{itemize}
\item An individual with trait $x=\ell \delta$ in the population gives birth to another individual with rate $b(x)=4-x$. With
  probability \be
  \label{mutation-prob} p_K=K^{-\alpha} \ \hbox{ with } \ \alpha\in (0,1)\ ,% \ {1/\alpha\notin \mathbb{N}},
  \ee
  a mutation occurs and the new offspring carries the  mutant trait $(\ell+1)\delta$.\\
  With probability $1-p_K=1-K^{-\alpha}$, the new individual inherits the ancestral trait.\\
  The birth rate favors small values of $x$ with optimum at $x=0$.
\item An individual with trait $x$ transfers its trait to a given individual of trait $y$ in a population of total size $N$ at rate
 \[
 \tau(x,y,N)=\frac{\tau}{N}\, \ind_{x>y},
 \]
 for some parameter $\tau>0$.
\item The individuals compete to survive (to share resources or territories). An individual with trait $x=\ell\delta$ in the
  population of total size $N$ dies with natural death rate $d_K(N)=1+CN/K$.
\end{itemize}

Note that because of the factor $1/K$,
competition is governed only by traits with population size of order $K$. Therefore, density-dependence disappears when the total
population size is negligible with respect to $K$.

\me Let us note that under scaling \eqref{mutation-prob}, the total mutation rate in a population with size of order $K$, is equal
to $K^{1-\alpha}$ and then goes to infinity with $K$. We are very far from the situation described in many papers as \cite{champagnat06,champagnatmeleard2011,billiardcolletferrieremeleardtran2} where the authors explore the assumptions of the  adaptive dynamics theory. In that cases,  the total mutation rate  $K p_{K}$ is assumed to satisfy
$\log K \ll 1/(Kp_{K}) \ll e^{CK}$, leading to a time scale separation between demographic and mutational events.
Here, small populations of size order $K^\beta, \beta<1$ can have a non negligible contribution to evolution by mutational events and
we need to take into account all subpopulations with size of order $K^\beta$.

\medskip
We need to consider in the sequel two different situations: either there is a single trait $x$ with population size of order
$K$, called \emph{resident} trait, or the total population size is $o(K)$. In this last case, a trait with the
largest population size is called \emph{dominant} trait.

When the trait $x$ is the unique resident trait, it is well
known~\cite{ethierkurtz} that, when $K$ tends to infinity, the total population size can be approximated by $K\,n(t)$ where $n(.)$ solves the ODE
\[
\dot n(t) = n(t)(3-x -C n(t)),
\]
whose unique positive stable equilibrium is given by
\be
\label{equilibrium}
\overline{n}(x)=\frac{3-x}{C}.
\ee
We define the invasion fitness of a mutant individual of trait $y$ in the population of trait $x$ and size $K\bar{n}(x)$ as its initial growth rate given by
\be
\label{eq:general-fitness}
S(y;x)=b(y)-d_K(K\overline{n}(x))+\tau \ind_{x<y}-\tau \ind_{x>y}=x-y+\tau\,\text{sign}(y-x),
\ee
where $\text{sign}(x)=1$ if $x>0$; $0$ if $x=0$; $-1$ if $x<0$.
Indeed, the total transfer rate from $x$ to $y$ is given by $K \overline{n}(x)\,\frac{\tau}{(K \overline{n}(x)+1)} \ind_{y>x} \sim
\tau\, \ind_{y>x} $ when $K\rightarrow+\infty$ and similarly from $y$ to $x$. Note that $S(x;x)=0$ and, for all traits $x, y$,
$S(y;x)= - S(x;y)$  (see Figure~\ref{fig:fitness}). This implies in particular that there is no long-term coexistence of two resident traits.
We also define the fitness of an individual of trait $y$ in a negligible population (of size $o(K)$) with dominant trait $x$ to be
\begin{equation}\label{eq:fitness-chap}
\widehat{S}(y;x)=3-y+\tau \ind_{x<y}-\tau \ind_{x>y}.
\end{equation}
Indeed, it corresponds to $\lim_{K\to\infty} b(y)-d_K(o(K))$.

\medskip
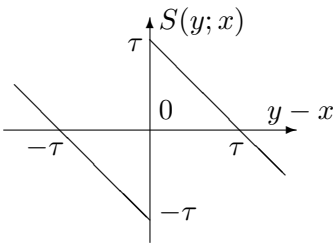
\begin{figure}[ht]
\begin{center}
    \unitlength=0.3cm
    \begin{picture}(13,10.5)
      \put(-1.5,5){\vector(1,0){13}} \put(10.2,5.5){$y-x$}
      \put(5,0){\vector(0,1){10}} \put(5.4,9.5){$S(y;x)$}
      \put(5.4,5.5){0}
      \put(4.0,8.5){$\tau$}
      \put(5.4,1){$-\tau$}
      \put(-0.5,3.9){$-\tau$}
      \put(8.5,3.9){$\tau$}
      \put(-1,7){\line(1,-1){6}}
      \put(5,9){\line(1,-1){6}}
    \end{picture}
 \caption{{\small \textit{Fitness function $S(y;x)$. In the absence of transfer ($\tau=0$), evolution favors traits $y$ smaller than $x$ (their fitness is positive). The introduction of a positive transfer reverts this evolutive trend: $S(y;x)>0$ if
$x<y<x+\tau$. Note that $S(y;x)>0$ also for $y<x-\tau$, which explains possible re-emergence of sufficiently smaller traits.}}}\label{fig:fitness}
\end{center}
\end{figure}

\me

Our study of the evolutionary dynamics of the model is based on a fine analysis of the size order, as power of $K$, of each
subpopulation corresponding to the different trait compartments. These powers of $K$ evolve on the timescale $\log K$, as can be
easily seen in the case of branching processes (see Lemma~\ref{lem:BP}). We thus define
$\beta^K_\ell(t)$ for $0\leq\ell\leq L$ such that
\begin{equation}
\label{def:beta}
N^K_{\ell}(t\log K)= K^{\beta^K_\ell(t)}-1,
\quad \mbox{ i.e. }\quad\beta^K_\ell(t)=\frac{\log(1+N^K_{\ell}(t\log K))}{\log K}.
\end{equation}

We assume that the trait $x=0$ is initially resident, with density $3/C$. A natural initial condition would hence be $N^K(0)=(\lfloor\frac{3K}{C}\rfloor,0,\ldots,0)$. However, on the time scale $\log K$, mutants are immediately created and therefore, we modify the initial condition as
\begin{equation}\label{eq:init}
  N^K(0)=\big(\lfloor\frac{3K}{C}\rfloor,\lfloor K^{1-\alpha}\rfloor,\dots,\lfloor K^{1-\ell\alpha}\rfloor,\dots,\lfloor K^{1-\lfloor
    1/\alpha\rfloor \alpha}\rfloor,  0,\dots, 0\big).
\end{equation}
This can be understood from Lemma~\ref{lem:initial-growth} (with $\beta=0$ and $c=1-\alpha$ for trait $\delta$). With this
initial condition, we have
\begin{equation}\label{eq:transfer-init-cond}
\beta^K_\ell(0)\xrightarrow[K\rightarrow+\infty]{}(1-\ell\alpha) \ind_{0\leq \ell<\frac{1}{\alpha}}.
\end{equation}

Our main result (Theorem~\ref{thm:transfer-main-new}) gives the asymptotic dynamics of
$\beta^K(t)=(\beta^K_0(t),\ldots,\beta^K_L(t))$ for $t\geq 0$ when $K\rightarrow+\infty$.
We show that the limit is a piecewise affine continuous function, which can be described along successive phases determined by their resident or dominant traits. When the latter trait changes, the fitnesses governing the slopes are modified. Moreover, inside each phase, other changes of slopes are possible due to a delicate balance between mutations, transfer and growth of subpopulations. Our ambition is to cover all
the possible cases: local extinctions, re-emergence of subpopulations, changes of slopes due to
mutation and selection, dynamics when the total population size is $o(K)$, total extinction of the population...
We deduce from the asymptotic dynamics of $\beta^K(t)$ explicit criteria for the occurrence of the different evolutionary outcomes observed in Figure~\ref{fig:simu} (Theorem~\ref{thm:criterion-evol-suicide}). We provide a detailed study of the case of three traits in Section~\ref{sec:3-traits}.

Such approach has already been used in~\cite{durrettmayberry,boviercoquillesmadi}. Durrett and Mayberry~\cite{durrettmayberry}
consider constant population size or pure birth (Yule process) models, with directional mutations and increasing fitness parameter.
They obtain travelling waves of selective sweeps. Bovier, Coquille and
Smadi~\cite{boviercoquillesmadi} consider a model with density-dependence but without transfer, with a single trait with positive
fitness separated from the initial trait by unfit traits. They obtain bounds on the time needed to cross the fitness valley.

In our case, the dynamics is far more complex due to the trade-off between larger birth rates for small trait values and transfer to
higher traits, leading to diverse evolutionary outcomes, including cyclic dynamics or evolutionary suicide. As a consequence, we need
to consider cases where the dynamics of a given trait is completely driven by immigrations due to mutations from the resident trait,
with time inhomogeneous immigration rates (see Theorem~\ref{thm:BPI-general}). This complexifies a lot the analysis.

We first state our main results in
Section~\ref{sec:model}. First, we give our general result  on the convergence of the exponents $\beta_\ell^K$ (Theorem~\ref{thm:transfer-main-new}) in Section~\ref{sec:2-1}. Then, we give general criteria on the parameters
$\tau$, $\delta$ and $\alpha$ for re-emergence of trait 0 and evolutionary suicide in Section~\ref{sec:2-2} (Theorem~\ref{thm:criterion-evol-suicide}). We then study in details the limit process in the case of three traits ($L=2$) in
Section~\ref{sec:3-traits}. The proofs of
Theorems~\ref{thm:transfer-main-new} and~\ref{thm:criterion-evol-suicide} are given in Sections~\ref{sec:proof-main} and~\ref{sec:general}. Useful lemmas on branching processes and branching
processes with immigration are given respectively in Appendices~\ref{sec:linear-BDP} and~\ref{sec:BPI}. Technical lemmas on birth and
death processes with logistic competition and transfer are given in Appendix~\ref{sec:logistic-BDP}. We conclude in
Appendix~\ref{sec:algo} with the algorithmic construction of the limit of the exponents $\beta_\ell^K$, used to perform simulations.

\section{Main results}
\label{sec:model}
We state the principal results of the paper. Their proofs are given in Sections \ref{sec:proof-main} and \ref{sec:general}.

\subsection{Asymptotic dynamics}\label{sec:2-1}

The next result characterizes the asymptotic dynamics of
$\beta^K(t)=(\beta^K_0(t),\ldots,\beta^K_L(t))$ (when $K\rightarrow+\infty$) by a succession of deterministic time intervals $[s_{k-1},s_{k}], k\geq 1$, called phases and delimited by changes of resident or dominant traits. The latter are unique except at times $s_k$ and are denoted by $\ell^*_{k}\delta, k\geq 1$. This asymptotic result holds until a time $T_0$, which guarantees that there is neither ambiguity on these traits (Point~(a) below) nor on the extinct subpopulations at the phase transitions (Point~(c) below).

\begin{thm}
  \label{thm:transfer-main-new}
  Assume that $\alpha\in(0,1)$, that $\delta\in(0,4)$ with $3/\delta\not\in\mathbb{N}$ and $\frac{\tau\pm 3}{\delta}\not\in\mathbb{N}$ and that
  \eqref{eq:transfer-init-cond} hold true.
 \begin{description}
  \item[\textmd{(i)}]
 There exists $T_0>0$ such that, for all $T\in(0,T_0)$, the sequence $(\beta^K(t), t\in[0,T])$ converges in probability in $\mathbb{D}([0,T],[0,1]^L)$ to a deterministic piecewise affine
  continuous function $(\beta(t)=(\beta_1(t),\ldots,\beta_L(t)),t\in[0,T])$, such that $\beta_\ell(0)=(1-\ell\alpha)
  \ind_{0\leq \ell<\frac{1}{\alpha}}$. The functions $\beta$ and $T_0$ are parameterized by $\alpha$, $\delta$ and $\tau$ defined as follows.
  \item[\textmd{(ii)}]
There exists an increasing nonnegative sequence $(s_k)_{k\geq 0}$  and a sequence
  $(\ell^*_k)_{k\geq 1}$ in $\{0,\ldots,L\}$ defined inductively as follows: $s_0=0$, $\ell^*_1=0$,
  and, for all $k\geq 1$, assuming that $s_{k-1}<T_0$ and $\ell^*_k$ have been constructed {and that $\beta(s_{k-1})\neq 0$}, we can construct $s_k>s_{k-1}$ as follows
  \begin{equation}
    \label{eq:rec-sk}
    s_{k}=\inf\{t> s_{k-1}:\exists \ell\neq\ell^*_k,\,\beta_\ell(t)=\beta_{\ell^*_k}(t)\}.
  \end{equation}
%  \begin{description}
%\item[\textmd{(a)}] \textcolor{red}{virer} if $\beta(s_{k-1})=0$, then $s_{k}=T_0=+\infty$ and $\beta(s)=0$ for all $s\in[s_{k-1},+\infty)$;
%\item[\textmd{(b)}] \textcolor{red}{seul cas} otherwise,
%  \end{description}
  We can then decide whether {we continue the induction after time $s_k$ (i.e.\ $T_0>s_k$) or not} as follows:
  \begin{description}
      \item[\textmd{(a)}] if $\beta_{\ell^*_k}(s_k)>0$, we set
  \begin{equation}
    \label{eq:rec-ell*}
    \ell^*_{k+1}=\arg\max_{\ell\neq \ell^*_{k}}\beta_\ell(s_{k})
  \end{equation}
  if the argmax is unique, or otherwise we set $T_0=s_k$ and we stop the induction;
  \item[\textmd{(b)}] {if $\beta_{\ell^*_k}(s_k)=0$, we set $s_{k+1}=T_0=+\infty$ and $\beta(t)=0$ for all $t\geq s_k$};
\item[\textmd{(c)}] {if in one of the previous cases, we have for some $\ell\neq\ell^*_k$, $\beta_\ell(s_{k})=0$ and $\beta_\ell(s_{k}-\varepsilon)>0$ for all $\varepsilon>0$ small enough, then we also set $T_0=s_k$ and stop the induction; otherwise, the induction proceeds to the next step.}
  \end{description}
  In the case where the induction never stops, we define $T_0=\sup_{k\geq 0} s_k$.
  \item[\textmd{(iii)}] In (ii), the functions $\beta_\ell$ are defined, for all $t\in [s_{k-1},s_k]$, by
  \begin{equation}
    \label{eq:evol-beta_0}
    \beta_0(t) =\left[\mathbbm{1}_{\beta_0(s_{k-1})>0}\,\left(\beta_0(s_{k-1})+\int_{s_{k-1}}^t \widetilde{S}_{s,k}(0;\ell^*_k\delta)\,ds\right)\right]\vee 0
  \end{equation}
  and, for all $\ell\in\{1,\ldots,L\}$,
  \begin{equation}
    \label{eq:evol-beta_j}
    \beta_\ell(t)
    =\left(\beta_\ell(s_{k-1})+\int_{t_{\ell-1,k}\wedge t}^t \widetilde{S}_{s,k}(\ell\delta;\ell^*_k\delta)\,ds\right)
    \vee(\beta_{\ell-1}(t)-\alpha)\vee 0,
  \end{equation}
  where, for all traits $x,y$,
  \begin{equation}
    \label{eq:def-tilde-S}
    \widetilde{S}_{t,k}(y;x)=\mathbbm{1}_{\beta_{\ell^*_k}(t)=1}\,S(y;x)+\mathbbm{1}_{\beta_{\ell^*_k}(t)<1}\,\widehat{S}(y;x)
  \end{equation}
  and where
  \begin{equation}
    \label{eq:def-t-ell-k}
    t_{\ell-1,k}=\begin{cases}
      \inf\{t\geq s_{k-1},\ \beta_{\ell-1}(t)=\alpha\},& \quad \mbox{ if }\beta_\ell(s_{k-1})=0,\\
      s_{k-1}, &  \quad \mbox{ otherwise.}
    \end{cases}
  \end{equation}
\end{description}
In addition, for all $\ell$ and all $a<b<T_0$ such that  the  time interval $[a,b]$ is included in the interior of the zero-set of $\beta_{\ell}$, the event $\{N_{\ell}^K(t\log K) = 0, \forall t\in[a,b]\}$ has a  probability converging to one as $K$ tends to infinity.
\end{thm}

\begin{rem}
\begin{enumerate}
\item It follows from the definition of $s_k$ and $\ell^*_{k+1}$ that $\max_\ell \beta_\ell(t)=\beta_{\ell^*_k}(t)$ for all $t\in[s_{k-1},s_k)$.
    \item In \eqref{eq:def-tilde-S}, when $\beta_{\ell^*_k}(t)=1$ for some $t\in(s_{k-1},s_k)$, there is a single resident trait $\ell^*_k \delta$ with
population of the order of $K$ and the function $S$ defined in \eqref{eq:general-fitness} is used.
In the case where $\beta_{\ell^*_k}(t)<1$, there is a single dominant trait and the total population size is of order $o(K)$ and the
fitness function is $\widehat{S}$ defined in \eqref{eq:fitness-chap}. % This explains the definition of $\widetilde{S}_{t,k}$ in~\eqref{eq:def-tilde-S}.
During each phase, the function $\widetilde{S}_{t,k}$ is actually constant, equal to $S$ or $\widehat{S}$ as above, except when a dominant population becomes resident in the same phase. In the first case, for all $t\in[s_{k-1},s_k)$, Eq.~\eqref{eq:evol-beta_0} and~\eqref{eq:evol-beta_j} take the simpler form
  \begin{equation*}
%    \label{eq:evol-beta_0}
    \beta_0(t) =\begin{cases}
      \big[\mathbbm{1}_{\beta_0(s_{k-1})>0}\,\big(\beta_0(s_{k-1})+S(\ell\delta;\ell^*_k\delta)(t-s_{k-1})\big)\big]\vee 0 & \text{if }\beta_{\ell^*_k}(s_{k-1})=1, \\
        \big[\mathbbm{1}_{\beta_0(s_{k-1})>0}\,\big(\beta_0(s_{k-1})+\widehat{S}(\ell\delta;\ell^*_k\delta)(t-s_{k-1})\big)\big]\vee 0 & \text{if }\beta_{\ell^*_k}(s_{k-1})<1
    \end{cases}
  \end{equation*}
  and, for all $\ell\in\{1,\ldots,L\}$,
  \begin{equation}
   \label{eq:evol-beta_j-simple}
    \beta_\ell(t)
    =\begin{cases}
      \big(\beta_\ell(s_{k-1})+S(\ell\delta;\ell^*_k\delta)(t-t_{\ell-1,k})_+\big)
    \vee(\beta_{\ell-1}(t)-\alpha)\vee 0 & \text{if }\beta_{\ell^*_k}(s_{k-1})=1, \\
       \big(\beta_\ell(s_{k-1})+\widehat{S}(\ell\delta;\ell^*_k\delta)(t-t_{\ell-1,k})_+\big)
    \vee(\beta_{\ell-1}(t)-\alpha)\vee 0 & \text{if }\beta_{\ell^*_k}(s_{k-1})<1.
    \end{cases}
  \end{equation}
  Otherwise, $\widetilde{S}_{t,k}$ switches from $\widehat{S}$ to $S$ at the first time where $\max_\ell \beta_\ell(t)=\beta_{\ell^*_k}(t)=1$. Therefore, since $S(\ell^*_k\delta,\ell^*_k\delta)=0$, we obtain in all cases
  \begin{equation*}
\beta_{\ell^*_k}(t)
=\begin{cases}
  1 & \text{if }\beta_{\ell^*_k}(s_{k-1})=1, \\
  \left[\left(\beta_{\ell^*_k}(s_{k-1})+\widehat{S}(\ell^*_k\delta;\ell^*_k\delta)\,(t-s_{k-1})\right)\wedge 1\right]\vee 0 & \text{if }\beta_{\ell^*_k}(s_{k-1})<1.
\end{cases}
\end{equation*}
      \item It follows from the last formula that $\max_\ell \beta_\ell(t)\leq 1$ for all $t\in[0,T_0]$.
      \item When $\beta_{\ell}(s_{k-1})=0$, the time $t_{\ell-1,k}$ corresponds to the first time where the incoming mutation rate in subpopulation $\ell\delta$ becomes significant.
      \end{enumerate}
\end{rem}
\begin{rem}
If the initial condition~\eqref{eq:init} was replaced by $N^K(0)=(\lfloor\frac{3K}{C}\rfloor,0,\ldots,0)$, the convergence in (i) would hold on $[\log K,T\wedge T_0]$ instead of $[0,T\wedge T_0]$.
\end{rem}

%It is clear from the construction that the function $\beta(t)$ is continuous and piecewise affine and that we always have $\beta_\ell(t)\geq \beta_{\ell-1}(t)-\alpha$ for all $\ell\geq 1$ and $t\geq 0$.

%The time $T_0$ is introduced to ensure that the argmax in~\eqref{eq:rec-ell*} is unique \textcolor{magenta}{and to ensure that
%  non-dominant or non-resident traits with zero exponent are actually extinct at times $s_k$. When the first condition is violated,
%  it is not clear which trait is going to be the next resident one, and when the second condition is violated, a change of resident
%  or dominant population occurs at a time where a subpopulation is close to extinction. Whether this subpopulation goes extinct or is
%  able to grow again relies on a finer analysis.}

\begin{figure}[!ht]
\begin{center}
\begin{tabular}{cc}
\includegraphics[width=7.5cm,height=4cm]{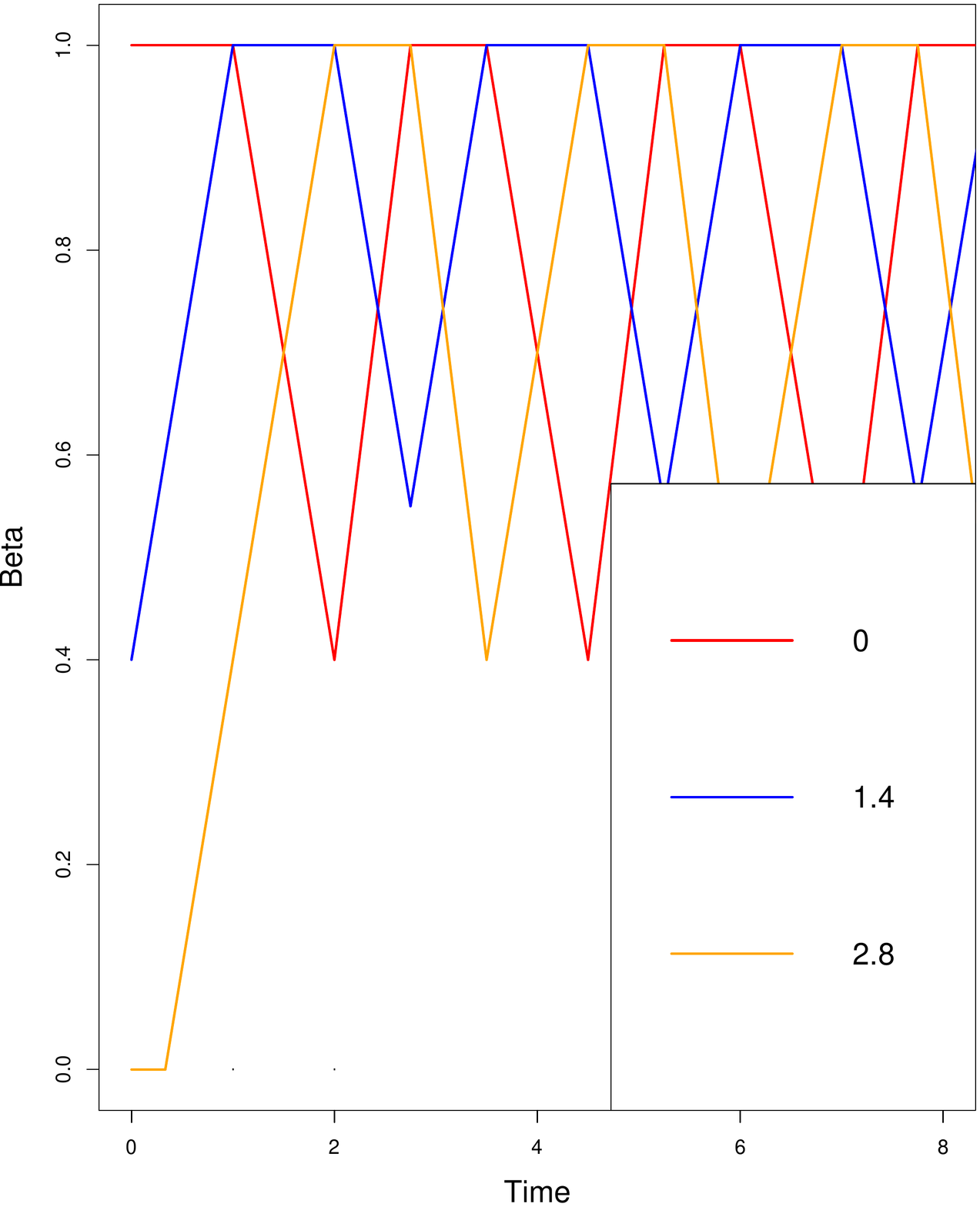}  &  \includegraphics[width=7.5cm,height=4cm]{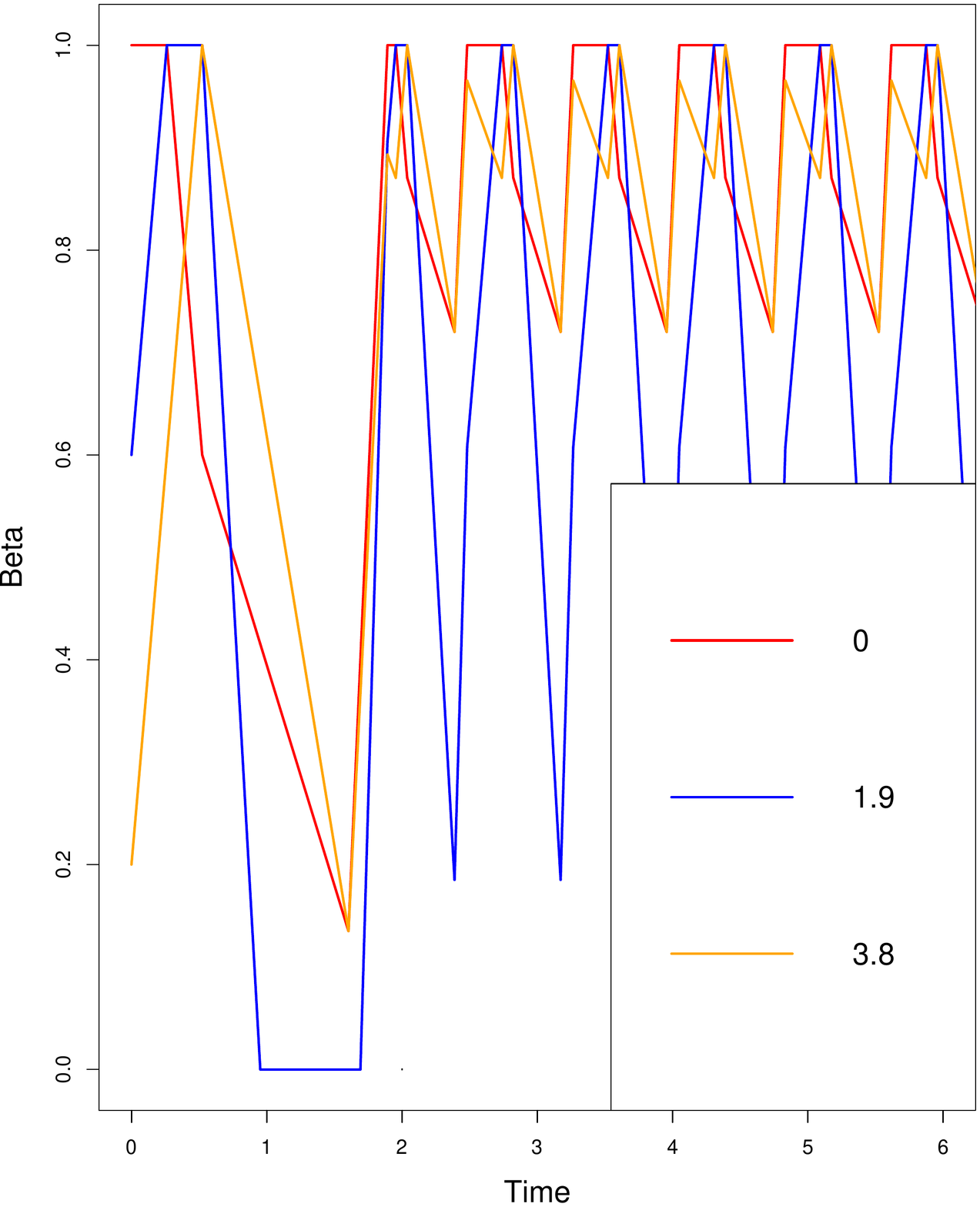}  \\
(a) & (b)  \\ \includegraphics[width=7.5cm,height=4cm]{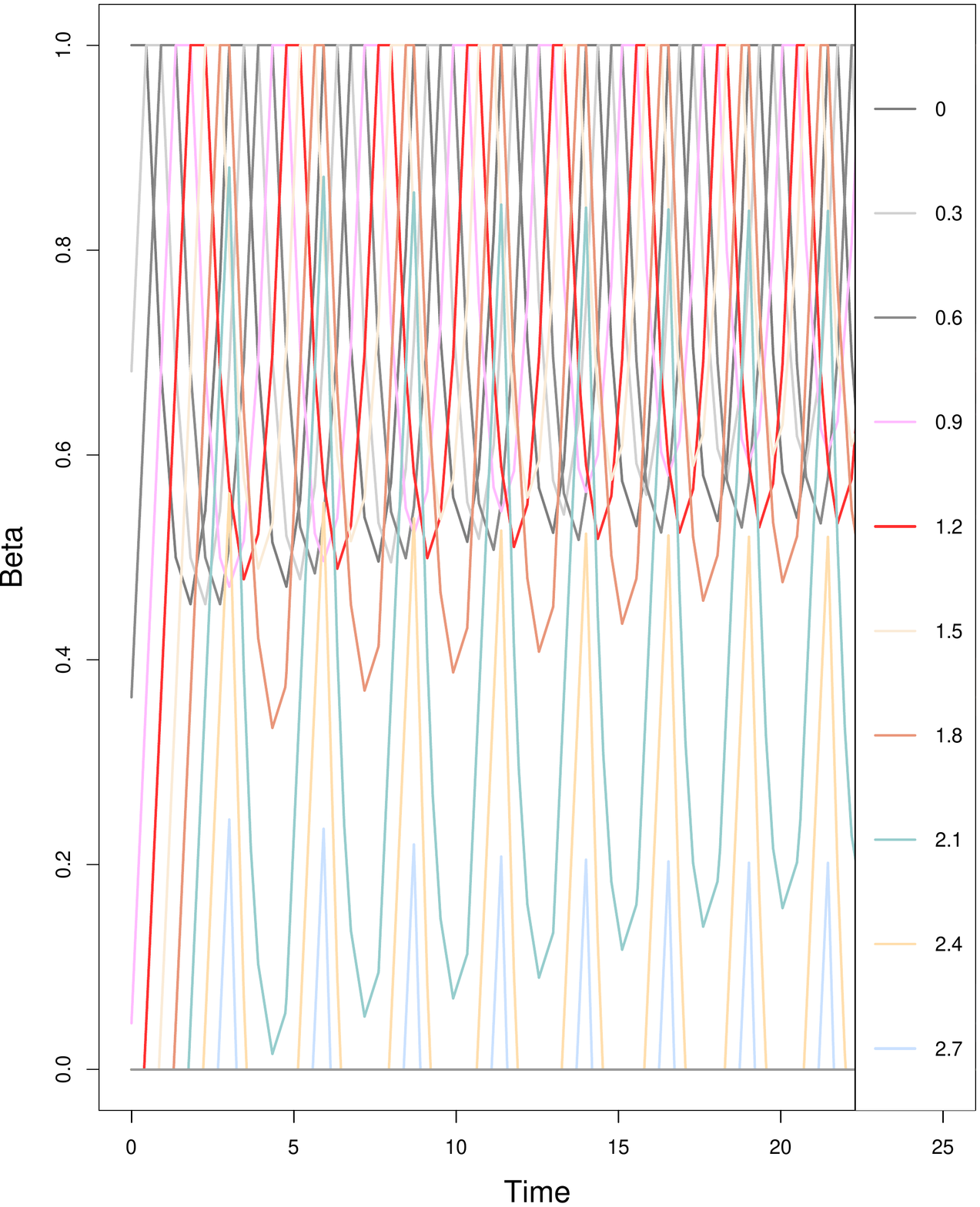}  &  \includegraphics[width=7.5cm,height=4cm]{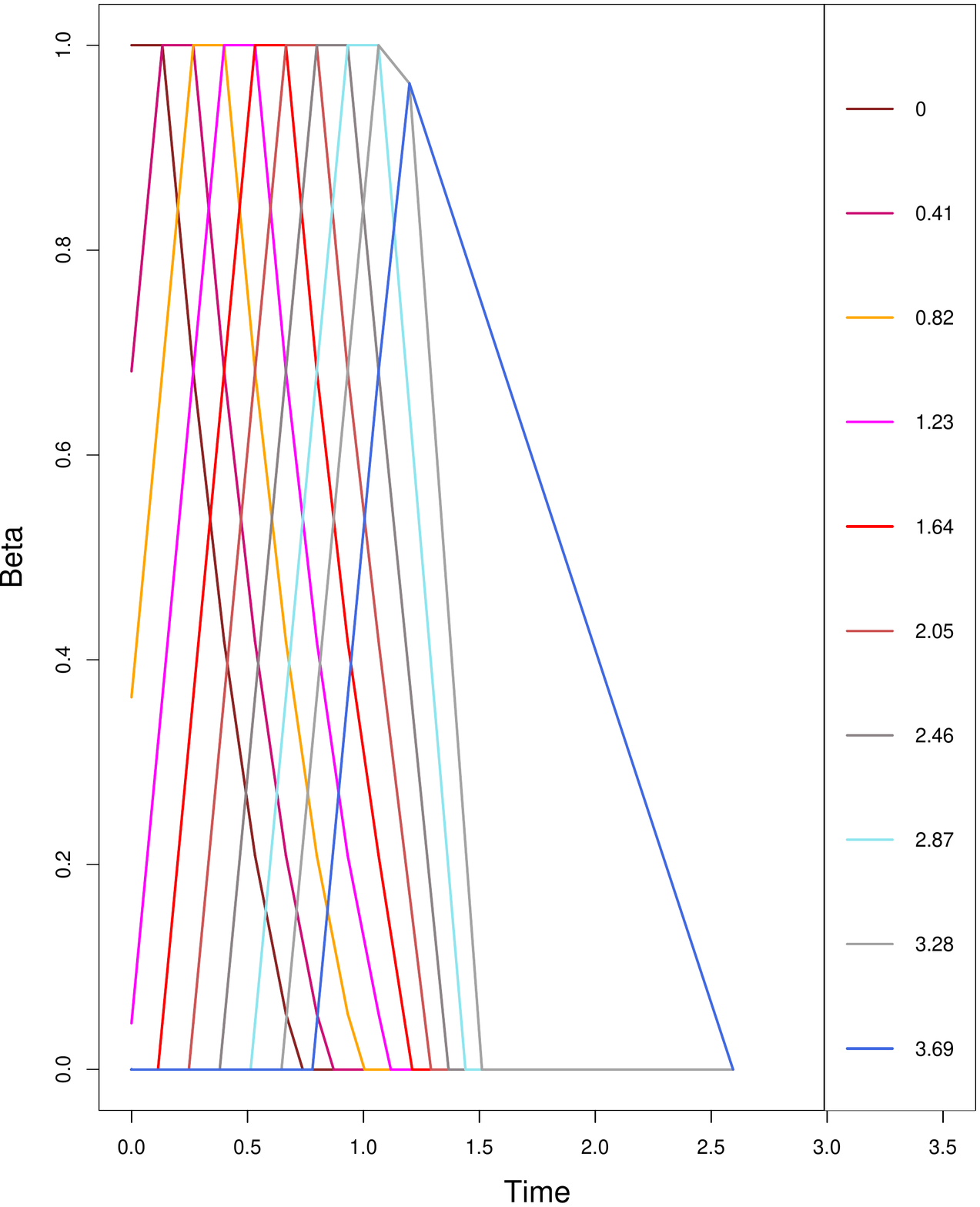}  \\
(c) & (d)
 \end{tabular}
 \caption{{\small \textit{Exponents $\beta_\ell(t)$ as functions of time. (a): $\delta=1.4$, $\alpha=0.6$, $\tau=2$. We see a
       {periodic} behavior showing re-emergences of the fittest traits. (b): $\delta=1.9$,
       $\alpha=0.4$, $\tau=3.43$. When the trait $2\delta$ becomes dominant, the population size is of order $o(K)$. We see a
       re-emergence of trait 0 after a phase of apparent macroscopic extinction (i.e.\ a total population size $o(K)$). Although the
       trait $\delta$ goes extinct while $2\delta$ is dominant, it is recreated by mutations from trait 0. (c): $\delta=0.3$,
       $\alpha=1/\pi$, $\tau=1$. A cyclic but non-periodic behaviour is observed. (d): $\delta=0.41$, $\alpha=1/\pi$, $\tau=2.8$. The
       population is directly driven to evolutionary suicide. }}}\label{fig:simu2}
\end{center}
\end{figure}

We cannot ensure that $T_0=+\infty$ for almost all parameters $\alpha$, $\delta$ and $\tau$. However we have not encountered any case
where $T_0<+\infty$ in the simulations. In the sequel, we exhibit large sets of parameters where $T_0=+\infty$ in the case of three
traits (Section~\ref{sec:3-traits}). We also prove in Theorem~\ref{thm:criterion-evol-suicide} that, for any sets of parameters,
$T_0$ is larger than the time of extinction or the time of first re-emergence.

Note that the previous result keeps track  of populations of size $K^\beta$ for $0<\beta\leq 1$, but not of  populations of smaller order, which go fast to extinction on the time scale
  $\log K$.

  \medskip The next theorem gives a characterization of $\beta$ as solution of a dynamical system.
\begin{cor}
  \label{thm:transfer-main}
  Under the assumptions of Theorem \ref{thm:transfer-main-new}, we set
   $$\,\ell^*(t)\,= \sum_{k\geq 1} \ell^*_{k}\mathbbm{1}_{[s_{k-1},s_{k}[}(t) \hbox{ and }   \widetilde{S}_{t}(y;x) =  \mathbbm{1}_{\beta_{\ell^*(t)}(t)=1}\,S(y;x)+\mathbbm{1}_{\beta_{\ell^*(t)}(t)<1}\,\widehat{S}(y;x). $$
 The function $\beta(t)$ is {right-}differentiable {on $[0,T_0)$} and
   satisfies
  \begin{equation}
    \label{eq:probleme-limite}
      \dot{\beta}_\ell(t) =
      \Sigma_\ell(t)\ind_{\beta_\ell(t)>0\text{ or }(\beta_\ell(t)=0\text{ and }\beta_{\ell-1}(t)=\alpha)} \end{equation}
 where $\Sigma_\ell$ is defined recursively by
 $\ \Sigma_0(t)= \widetilde{S}_{t}(0, \delta \ell^*(t))\ $ and $ \forall \ell\geq 1$
   \begin{eqnarray}
   \label{rec-sigma}
    \Sigma_\ell(t)&=&\begin{cases} \widetilde{S}_{t}(\ell\delta;\ell^*(t)\delta) \vee  \Sigma_{\ell-1}(t)&\hbox{ if } \beta_\ell(t) =  \beta_{\ell-1}(t)-\alpha\\
    \widetilde{S}_{t}(\ell\delta;\ell^*(t)\delta) &\hbox{ if } \beta_\ell(t) >  \beta_{\ell-1}(t)-\alpha.
    \end{cases}
    \end{eqnarray}
\end{cor}

\begin{rem}
  One may wonder if the ODE \eqref{eq:probleme-limite} characterizes the function $\beta$. For this we first need to characterize
  $\ell^*(t)$ as an explicit function of $\beta(t)$. One would like to define it as $\ell^*(t)= \arg\max_{0\leq \ell\leq L}
  \beta_{\ell}(t)$ and take it right-continuous. This is correct if there is a single argmax. Otherwise, there are by
  {definition of $T_0$} only two choices $\ell$ and $\ell'$ and there is a single admissible
  choice in the sense that the corresponding {affine} solution to \eqref{eq:probleme-limite} on
  $[t,t+\varepsilon]$ satisfies $\ell^*(s)= \arg\max_{0\leq \ell\leq L} \beta_{\ell}(s)$ locally for $s\in(t,t+\varepsilon)$ for
  $\varepsilon>0$ small enough. Indeed if $\max_{0\leq \ell\leq L} \beta_{\ell}(t)=1$ and since
  $S(\ell'\delta,\ell\delta)=-S(\ell\delta,\ell'\delta)$, one of the two fitnesses is positive, for example
  $S(\ell\delta,\ell'\delta)$. If one takes the wrong choice $\ell^*(t)= \ell'$, then $\Sigma_{\ell}(t) =
  S(\ell\delta,\ell'\delta)>0$, hence the solution of \eqref{eq:probleme-limite} gives $\beta_{\ell} (s)>1$ for $s>t$ locally, which is forbidden. If
  $\max_{0\leq \ell\leq L} \beta_{\ell}(t)<1$, a similar argument with $\widehat S$ consists in choosing the trait with higher
  invasion fitness.

  Therefore,~\eqref{eq:probleme-limite} can be expressed as an autonomous ODE and there is a unique admissible solution.
  Generalizations of our result to models with different birth, death and transfer rates, can be obtained by changing accordingly the fitness
  function in this ODE.
\end{rem}

Simulations are shown in Figure~\ref{fig:simu2} for various parameter values. The times $s_k$ correspond to changes of resident or
dominant populations. However, we observe several changes of slopes between these times. The computation of these successive times
called $t_k$ is given in Theorem~\ref{cor:transfer-main} in Appendix~\ref{sec:algo}.

%\clearpage

\subsection{Re-emergence of trait 0}\label{sec:2-2}

In Figure~\ref{fig:simu2}, we have exhibited different evolutionary dynamics (re-emergence of a trait, cyclic behavior, local extinction, evolutionary suicide). By re-emergence of a trait $\ell\delta$, we mean that $\beta_\ell(s)=1$ on some non-empty time interval $[t_1,t_2]$, then
$\beta_\ell(s)<1$ on some non-empty interval $(t_2,t_3)$ and then $\beta_\ell(s)=1$ again on some non-empty interval $[t_3,t_4]$. We would like to predict the evolutionary outcome as a function of parameters $\alpha,\delta,\tau$.
As detailed for three traits ($L=2$) in the next section, there are so many situations that  we are not able to fully
characterize the outcomes. Therefore, we focus on the beginning of the dynamics until either global extinction or re-emergence of one trait occurs.
The resurgence of trait $0$ is a prerequisite for a cyclic dynamics as those observed in
Figures~\ref{fig:simu}~(c).

We assume that $\delta<4/3$ (so that $L\geq 3$) and only consider the case $\delta<\tau<3$. Let
\begin{equation}
\label{eq:def-k}
\widetilde{k}:=\lceil \frac{\tau}{\delta}\rceil \quad\hbox{and}\quad \bar{k}=\lfloor 2\frac{\tau}{\delta}
\rfloor. %\hbox{ and } \widehat{k}=\lceil \frac{3}{\delta}\rceil
\end{equation}
We will see in the proof of the next result that, for the first phases,
\begin{equation*}
%    \label{def:sk}
    s_k:=\frac{k\alpha}{\tau-\delta},
\end{equation*}
the trait $k \delta$ is resident on $[s_{k},s_{k+1})$ ($ \beta_k(s)=1$) and  for all $s\in
[s_{k},s_{k+1})$,
\begin{equation*}
  \beta_\ell(s)=
  \begin{cases}
    \left[1-(\ell-k)\alpha+(\tau-\delta)(s-s_k)\right]\vee 0 & \text{if }k<\ell\leq L, \\
    1-\frac{\alpha(k-\ell-1)}{\tau-\delta}\left(\tau-\frac{k-\ell}{2}\delta\right)-(\tau-(k-\ell)\delta)(s-s_k) &
    \text{if }0\leq \ell<k.
  \end{cases}
\end{equation*}
These formulas stay valid until either $\beta_{0}(s)=0$ (loss of $0$), or $\beta_{0}(s)=1$ for some $s>s_{1}$ (re-emergence of $0$), or $\ell^*_{k} \delta >3$, where $\ell^*_k$ has been defined in~\eqref{eq:rec-ell*} (the population size becomes $o(K)$). The function $\beta_{0}(s)$ in the previous equation is piecewise affine and its slope becomes positive at time $s_{\widetilde{k}}$. Hence its minimal value   is equal to
\begin{equation}
\label{eq:def-m0}
m_{0} = \beta_{0}(s_{\widetilde{k}}) = 1- \frac{\alpha(\widetilde{k}-1)}{\tau-\delta}\Big(\tau-\frac{\widetilde{k}}{2}\delta\Big).
\end{equation}
Provided the latter is positive, $\beta_{0}$ reaches $1$ again in phase $[s_{\bar{k}},s_{\bar{k}+1})$ at time
\begin{equation}
\label{eq:def-tau-bar}
  \bar{\tau}:=s_{\bar{k}}+\frac{\alpha(\bar{k}-1)}{\tau-\delta}\,\frac{\tau-\frac{\bar{k}}{2}\delta}{\bar{k}\delta-\tau}=s_{\lfloor 2\frac{\tau}{\delta}\rfloor}+\frac{\alpha(\lfloor 2\frac{\tau}{\delta}\rfloor-1)}{\tau-\delta}\,\frac{\tau-\frac{\lfloor 2\frac{\tau}{\delta}\rfloor}{2}\delta}{\lfloor 2\frac{\tau}{\delta}\rfloor\delta-\tau}.
\end{equation}

\begin{thm}
  \label{thm:criterion-evol-suicide}
  Assume $\delta<\tau<3$, $\delta<4/3$ and under the assumptions of Theorem \ref{thm:transfer-main-new},
  \begin{description}
  \item[\textmd{(a)}] If $m_{0}>0$ and
    $\bar{k}\delta<3$, then the first re-emerging trait  is  $0$ and the maximal
    exponent is always 1 until this re-emergence time.
  \item[\textmd{(b)}] If $m_{0}<0$, the trait $0$ gets lost before its re-emergence and there is
    global extinction of the population before the re-emergence of any trait.
  \item[\textmd{(c)}] If $m_0>0$ and $\bar{k}\delta>3$, there is re-emergence of some trait $\ell\delta<3$ and, for some time $t$
    before the time of first re-emergence, $\max_{1\leq \ell\leq L}\beta_\ell(t)<1$.
  \end{description}
\end{thm}

Biologically, Case (b) corresponds to evolutionary suicide. In Cases (a) and (c), very few individuals with small traits remain,
which are able to re-initiate a population of size of order $K$ (re-emergence) after the resident or dominant trait becomes too large. In these cases,
one can expect successive re-emergences. However, we don't know if there exists a limit cycle for the dynamics. Case (c) means that
the total population is $o(K)$ on some time interval, before re-emergence occurs after
populations with too large traits become small enough.\\

Heuristically, using the approximation that $\widetilde{k}\approx\tau/\delta$, we obtain that $m_0\approx 1-\frac{\alpha
  \tau}{2\delta}$. Hence, we have $m_0>0$ (re-emergence) provided $\tau \lesssim 2\delta/\alpha$ and extinction otherwise. Transfer
rates higher than $2\delta/\alpha$ favor extinction because the population is pushed to higher trait values. Small values of
$\delta$ or high values of $\alpha$ give more time for extinction of the small subpopulations. Note that, for $m_0>0$, the condition
$\bar{k}\delta<3$ is roughly $\tau<3/2$. Hence, for transfer rates smaller than $3/2$, 0 re-emerges first, while other traits can
re-emerge before 0 otherwise.
%\textcolor{magenta}{Faire le lien avec les figures 1.1.}

\section{Case of three traits}
\label{sec:3-traits}

Before proving our main results, let us illustrate the limit exponents $\beta(t)$ in the case of three traits. Let
us consider $\delta>0$ such that $2\delta<4 <3\delta$, so that the possible traits are $0$, $\delta$ and $2\delta$. A simulation is
shown step by step in Figure \ref{fig:exemple3traits}, which we will now explain.

\begin{figure}[!ht]
\begin{center}
\begin{tabular}{cc}
\includegraphics[width=7.5cm, height=4cm,angle=0]{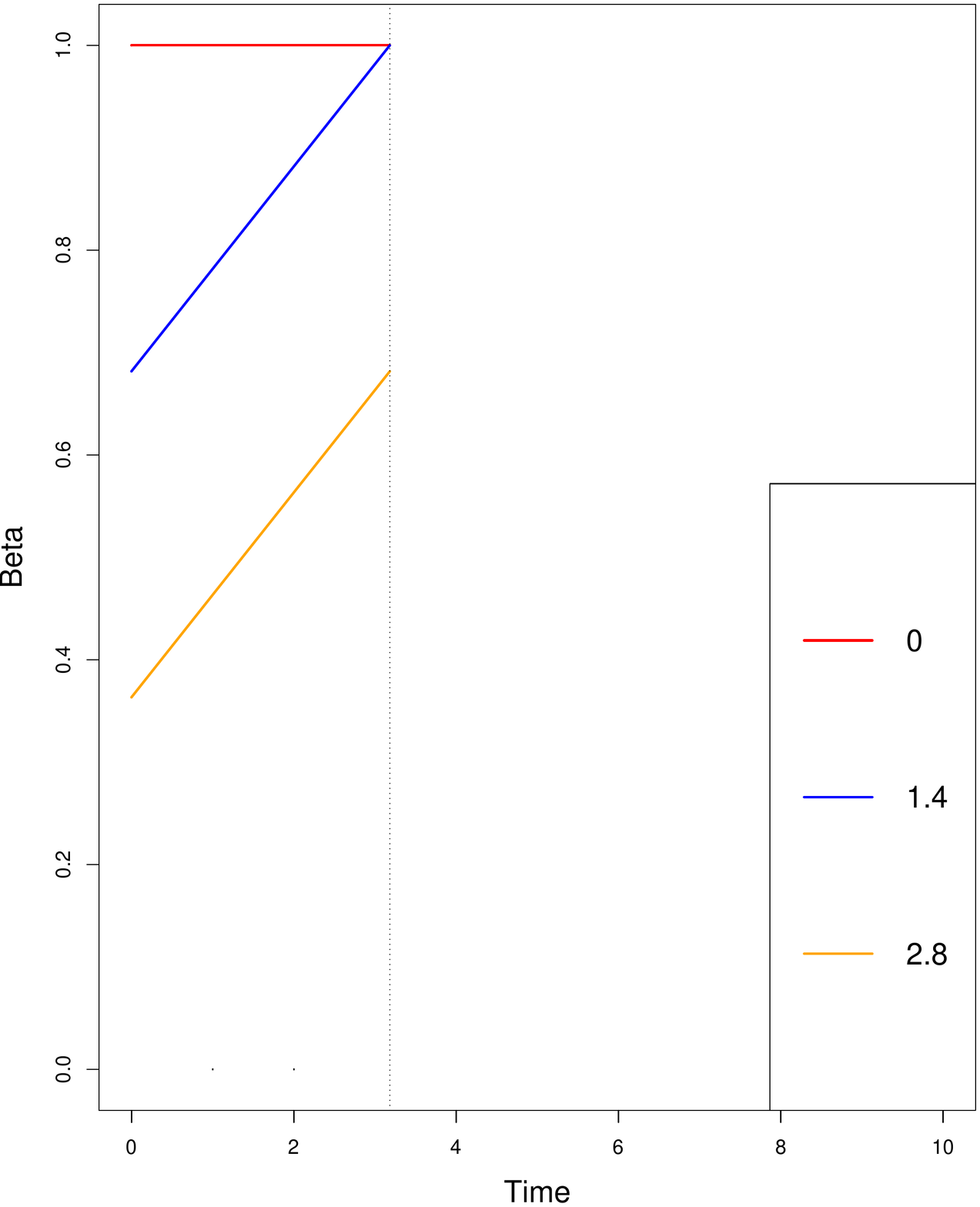} &
\includegraphics[width=7.5cm, height=4cm,angle=0]{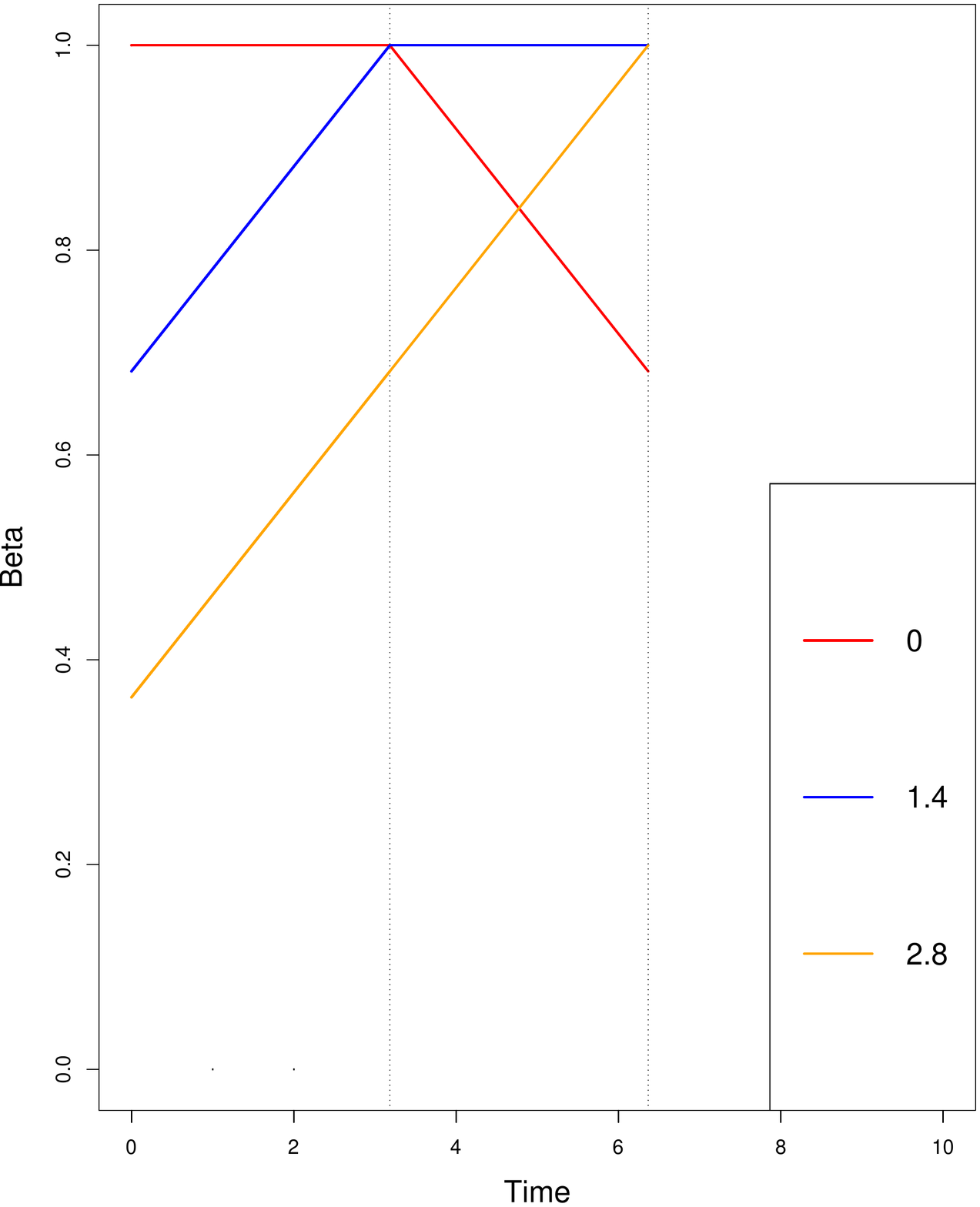} \\
(1) $t\in [0,s_1]$ & (2) $t\in[0,s_2]$ \\
\includegraphics[width=7.5cm, height=4cm,angle=0]{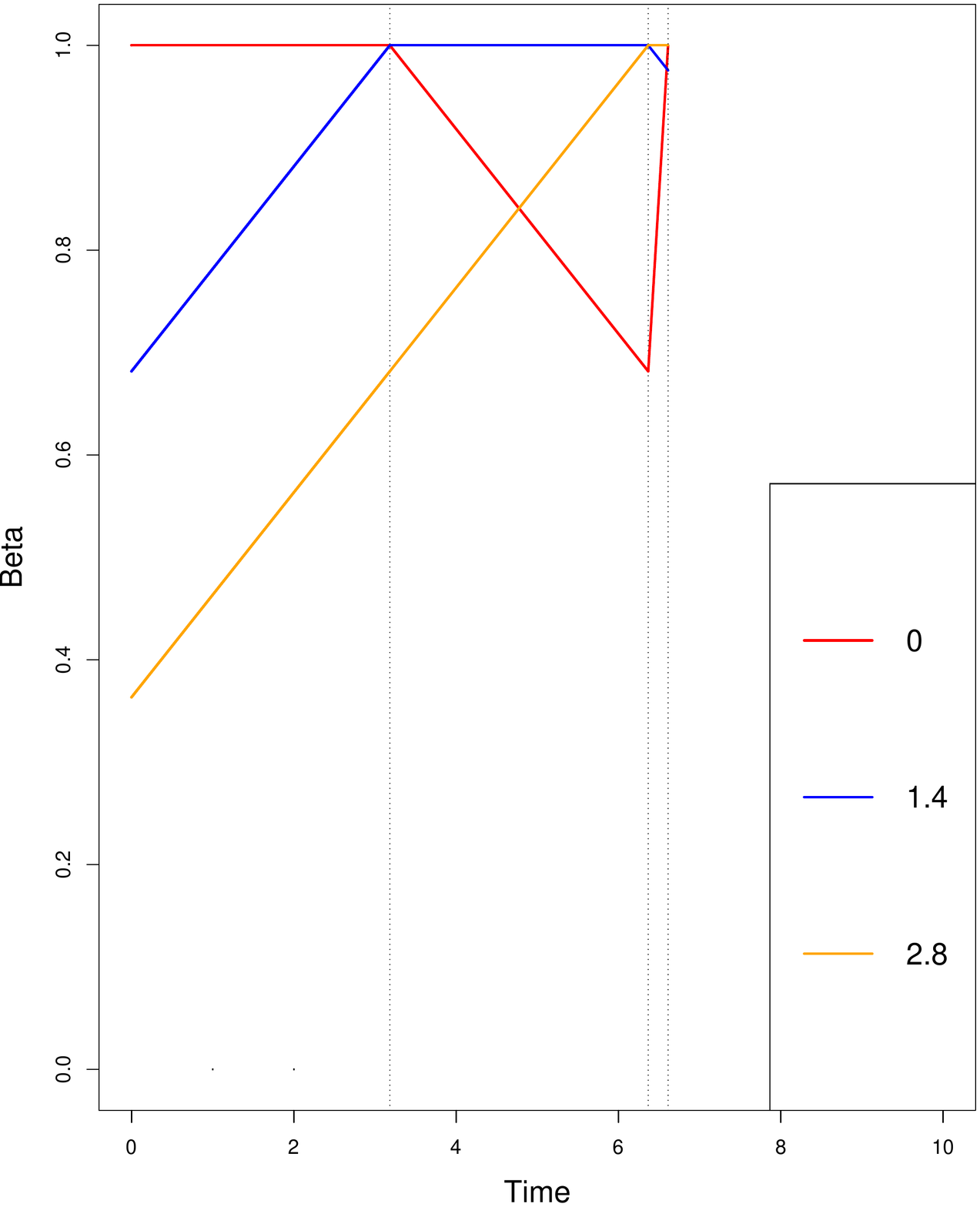} &
\includegraphics[width=7.5cm, height=4cm,angle=0]{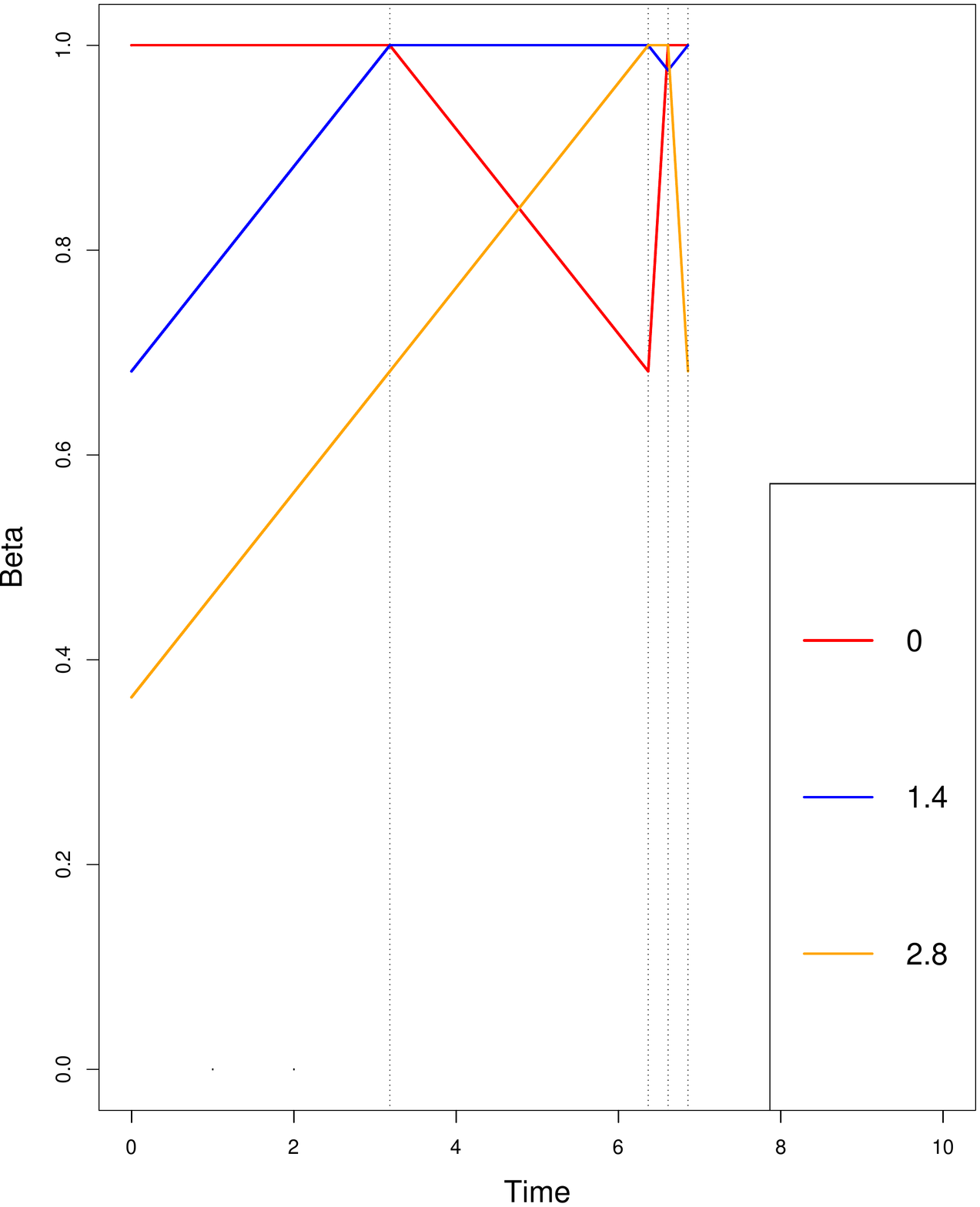} \\
(3) $t\in[0,s_3]$ & (4) $t\in[0,s_4]$ \\
\multicolumn{2}{c}{\includegraphics[width=13cm, height=6.5cm,angle=0]{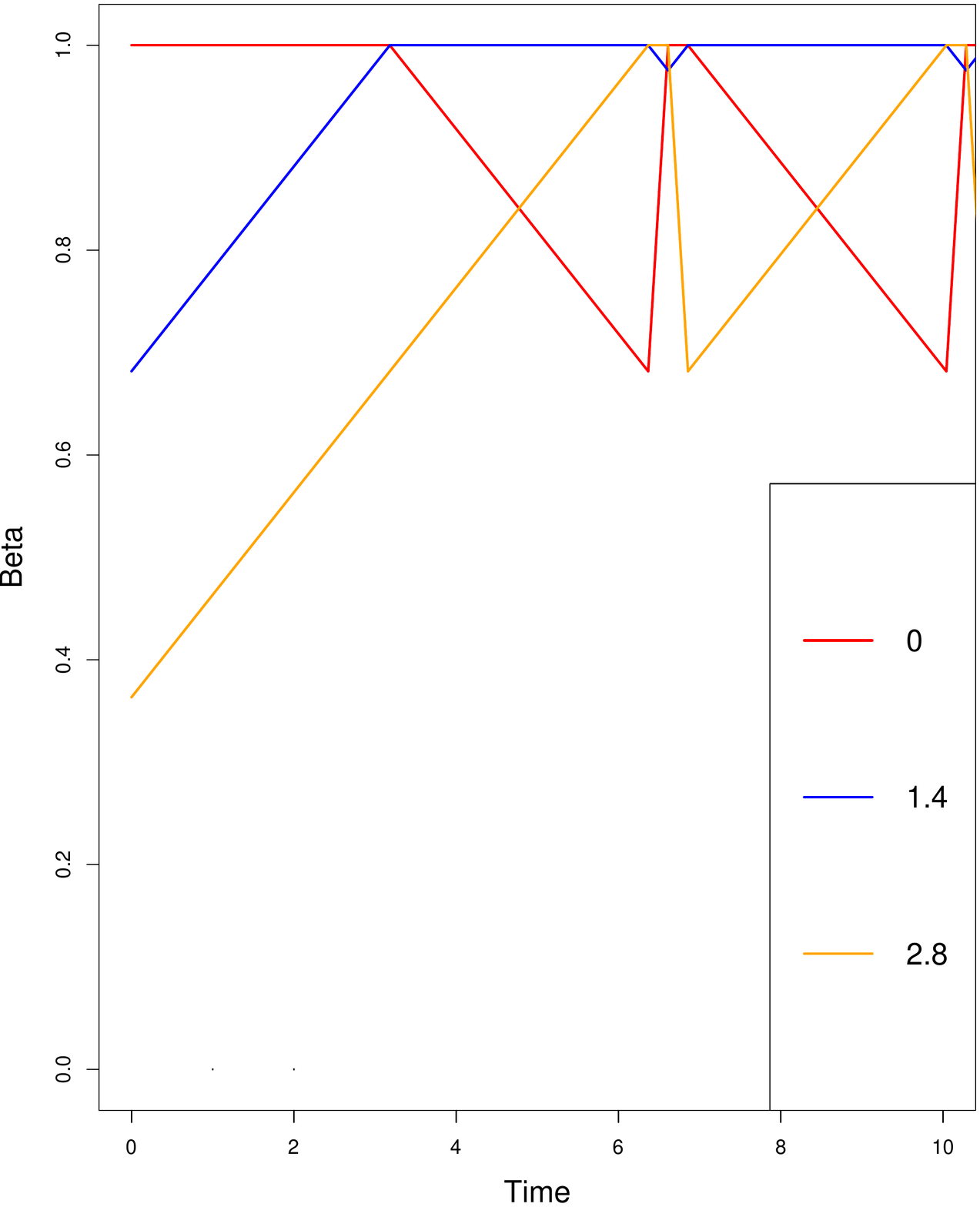}}
\\
\multicolumn{2}{c}{(5) $t\in[0,s_5]$}
\end{tabular}
\caption{{\small \textit{Construction step by step of the exponents $\beta_0(t)$, $\beta_1(t)$ and $\beta_2(t)$, as function of time
      in the case where $\delta < \tau <2\delta <3 <4 <3\delta$. Here, $\delta=1.4$, $\alpha=\frac{1}{\pi}$, $\tau=1.5$. The traits
      are $0$, $1.4$ and $2.8$. }}}\label{fig:exemple3traits}
\end{center}
\end{figure}

The initial condition is
\begin{equation}
\label{eq:beta-0}
\beta(0)=(\beta_0(0),\,\beta_1(0),\,\beta_2(0))=(1,1-\alpha,(1-2\alpha)\vee 0).
\end{equation}
The fitnesses given in~\eqref{eq:general-fitness} are
\[S(0; 0)=0,\qquad S(\delta ; 0)= \tau-\delta,\qquad S(2\delta ; 0)=\tau-2\delta.\]

\subsection{Case 1: $\tau<\delta$}
\label{sec:case-1}

In this case, neither the traits $\delta$ nor the trait $2\delta$ are advantageous and these populations survive only thanks to the
mutations from the trait 0 to $\delta$ and from the trait $\delta$ to $2\delta$. The exponents remain constant and $\forall t\geq 0,\
\beta(t)=\beta(0)$.

\subsection{Case 2: $\delta<\tau<2\delta$}
\label{sec:case-2}

Following Theorem~\ref{thm:transfer-main-new}, we shall decompose the dynamics of $\beta(t)$ into successive phases corresponding to
the time intervals $[s_{k-1},s_k]$.

\bigskip

\noindent \textbf{Phase 1: time interval $[0,s_1]$.} In this Case 2, $S(0;0)=0$, $S(\delta; 0)>0$ and $S(2\delta ; 0)<0$. While the
resident population remains the population with trait $0$, the population with trait $\delta$ has positive fitness and its growth is
described for $t\in[0,s_1]$ by the exponent:
\[\beta_1(t)=\big[(1-\alpha)+(\tau-\delta)t\big] \vee (1-\alpha)\vee 0=(1-\alpha)+(\tau-\delta)t.\]
The bracket corresponds to the intrinsic growth associated with the fitness $S(\delta ; 0)$, the term $1-\alpha$ is the contribution
of mutations from the population of trait $0$ and is here smaller than the term with the bracket.\\
The population with trait $2\delta$ has negative fitness and
\[\beta_2(t)=\big[(1-2\alpha)-(2\delta-\tau)t\big]\vee \big[(1-2\alpha)+(\tau-\delta)t\big]\vee 0=[(1-2\alpha)+(\tau-\delta)t]\vee 0.\]
As for the trait $\delta$, the first bracket corresponds to the intrinsic growth with a negative slope $\tau-2\delta<0$, while the second bracket corresponds to the contribution of mutations from the population with trait $\delta$.\\
It is clear that $\beta_2(t)<\beta_1(t)\leq 1$. Hence the first phase stops when $\beta_1(t)=1$, for
\[s_1=\frac{\alpha}{\tau-\delta}.\]
The first phase is illustrated in Fig. \ref{fig:exemple3traits}(1).
\medskip

\noindent \textbf{Phase 2: time interval $[s_1,s_2]$.} At time $s_1$, the populations with traits $0$ and $\delta$ both have sizes of order $K$: more precisely,
the exponents are
\[\beta_0(s_1)=1,\qquad \beta_1(s_1)=1,\qquad \beta_2(s_1)=1-\alpha.\]
Because $S(0 ; \delta)<0$ and $S(\delta ; 0)>0$, the new resident population with trait $\delta$ replaces the population with trait 0
whose exponent decreases after time $s_1$. The size of the population with trait $\delta$ remains close to $(3-\delta)K/C$, i.e.\ $\beta_1(t)=1$, during the
whole Phase 2, and using~\eqref{eq:general-fitness}:
\[S(0 ; \delta)=\delta-\tau<0,\qquad S(\delta ; \delta)=0,\qquad S(2\delta;\delta)=\tau-\delta>0.\]
Thus, the decrease of the population with trait $0$ is described by
$\ \beta_0(s_1+t)=[1-(\tau-\delta)\,t]\vee 0\ $
(recall that no mutant can have trait $0$).
The population with trait $2\delta$ has a positive fitness (first bracket in the following equation) and benefits from mutations
coming from the trait $\delta$ (second bracket):
\[\beta_2(s_1+t)=\big[(1-\alpha)+(\tau-\delta)\,t\big]\vee \big[1-\alpha\big]\vee 0=(1-\alpha)+(\tau-\delta)\,t.\]
This second phase stops when $\beta_2(t)=1$, at time
\[s_2=s_1+\frac{\alpha}{\tau-\delta}=\frac{2\alpha}{\tau-\delta}.\]
We check that
$\ \beta_0(s_1+t)=1-(\tau-\delta)\,t,\quad \forall t\in[s_2-s_1]$.
This phase is illustrated in Fig. \ref{fig:exemple3traits}(2).

\subsubsection{Case 2(a): $2\delta<3$}
\label{sec:case-2a}

In this case, trait $2\delta$ can survive on its own, i.e.\ its equilibrium population size $\frac{3-2\delta}{C}$ is positive.
\medskip

\noindent \textbf{Phase 3: time interval $[s_2,s_3]$.} Because $S(\delta ; 2\delta)<0$ and $S(2\delta ; \delta)>0$, we have at time
$s_2$ a replacement of the resident population with trait $\delta$ by the population with trait $2\delta$ which becomes the new
resident population, i.e.\ $\beta_2(t)=1$. At time $s_2$, the exponents are:
\begin{equation}
    \label{eq:phase3-fin-fin}
    \beta_0(s_2)=1-\alpha,\qquad \beta_1(s_2)=1,\qquad \beta_2(s_2)=1.
\end{equation}
The population size of trait $2\delta$ is close to $(3-2\delta)K/C$ so that the fitnesses are:
\[S(0 ; 2\delta)=2\delta-\tau>0,\qquad S(\delta ; 2\delta)=\delta-\tau<0,\qquad S(2\delta ; 2\delta)=0.\]
The population with trait $0$ increases with the exponent
$\ \beta_0(s_2+t)=(1-\alpha)+(2\delta-\tau)\,t$.

The trait $\delta$ has negative fitness but benefits from mutations coming from the trait $0$:
\begin{equation}
\beta_1(s_2+t)=\big[1-(\tau-\delta)\,t\big]\vee \big[(1-2\alpha)+(2\delta-\tau)\,t\big]\vee 0.\label{etape-3traits-1}
\end{equation}
This third phase is illustrated in Fig. \ref{fig:exemple3traits}(3). This phase stops when $\beta_0(t)=1$, i.e.\ at time
\[s_3=s_2+\frac{\alpha}{2\delta-\tau}=\frac{2\alpha}{\tau-\delta}+\frac{\alpha}{2\delta-\tau}.\]
We also have
$\
\beta_1(s_3)=\left(1-\alpha\,\frac{\tau-\delta}{2\delta-\tau}\right)\vee(1-\alpha)$.

\medskip

We have to distinguish two cases for Phase 4, depending on the value of $\beta_1(s_3)$.
\medskip

\noindent \textbf{Phase 4, case 2(a)(i): time interval $[s_3,s_4]$ under the assumption $\tau-\delta<2\delta-\tau$.} Then,
\[\beta_0(s_3)=1,\quad \beta_1(s_3)=1-\alpha\,\frac{\tau-\delta}{2\delta-\tau},\quad \beta_2(s_3)=1.\]
The new resident population is the one with trait $0$ and the fitnesses are the same as in Phase 1, but the initial conditions are
different. We obtain as above the exponents $\beta_0(s_3+t)=1$,
\begin{align*}\beta_1(s_3+t)= & \big[1-\frac{\tau-\delta}{2\delta-\tau}\alpha + (\tau-\delta)\,t\big] \vee \big[1-\alpha\big]\vee 0
=
1-\frac{\tau-\delta}{2\delta-\tau}\alpha + (\tau-\delta)\,t.\end{align*}
and
\[\beta_2(s_3+t)= \big[1-(2\delta-\tau)\,t\big]\vee \big[1-\frac{\delta\alpha}{2\delta-\tau} + (\tau-\delta)\,t\big]\vee 0 ,\]
as illustrated in Fig. \ref{fig:exemple3traits} (4). The phase stops when $\beta_1(t)=1$ at time
\[s_4=s_3+\frac{\alpha}{2\delta-\tau}=\frac{2\alpha}{\tau-\delta}+\frac{2\alpha}{2\delta-\tau}.\]
We check that
$\
\beta_2(s_3+t)= 1-(2\delta-\tau)\,t,\quad\forall t\leq s_4-s_3
\ $
and hence
\[\beta_0(s_4)=1,\qquad \beta_1(s_4)=1,\qquad \beta_2(s_4)=1-\alpha.\]
We recognize the initial condition of Phase 2. Therefore, the system behaves periodically (as in Figure~\ref{fig:simu2}(a))
starting from time $s_1$, with period
\[
s_4-s_1=\frac{\alpha}{\tau-\delta}+\frac{2\alpha}{2\delta-\tau}.
\]
\medskip

\noindent \textbf{Phase 4, case 2(a)(ii): time interval $[s_3,s_4]$ under the assumption $\tau-\delta>2\delta-\tau$.} In this case,
\[\beta_0(s_3)=1,\quad \beta_1(s_3)=1-\alpha,\quad \beta_2(s_3)=1.\]
In this case, we obtain $\beta_0(s_3+t)=1$, $\,\beta_1(s_3+t)=1-\alpha + (\tau-\delta)\,t$,
\[\beta_2(s_3+t)= 1-(2\delta-\tau)\,t\quad\text{and}\quad
s_4=s_3+\frac{\alpha}{\tau-\delta}=\frac{3\alpha}{\tau-\delta}+\frac{\alpha}{2\delta-\tau}.\]
\medskip

\noindent \textbf{Phase 5, case 2(a)(ii): time interval $[s_4,s_5]$.} We have
\[\beta_0(s_4)=1,\qquad \beta_1(s_4)=1,\qquad \beta_2(s_4)=1-\alpha\,\frac{2\delta-\tau}{\tau-\delta}.\]
Proceeding as above, we obtain
\[
s_5=s_4+\alpha\,\frac{2\delta-\tau}{(\tau-\delta)^2}=\frac{3\alpha}{\tau-\delta}+\frac{\alpha}{2\delta-\tau}+\alpha\,\frac{2\delta-\tau}{(\tau-\delta)^2}
\]
and for all $t\in[0,s_5-s_4]$, $\beta_1(s_4+t)=1$,
\[\beta_0(s_4+t)=1-(\tau-\delta)t\quad\text{and}\quad
\beta_2(s_4+t)=1-\alpha\,\frac{2\delta-\tau}{\tau-\delta}+(\tau-\delta)\,t.\]
\medskip

\noindent \textbf{Phase 6, case 2(a)(ii): time interval $[s_5,s_6]$.} We have
\[\beta_0(s_5)=1-\alpha\,\frac{2\delta-\tau}{\tau-\delta},\qquad \beta_1(s_5)=1,\qquad \beta_2(s_5)=1.\]
We obtain
\[
s_6=s_5+\frac{\alpha}{\tau-\delta}=\frac{4\alpha}{\tau-\delta}+\frac{\alpha}{2\delta-\tau}+\alpha\,\frac{2\delta-\tau}{(\tau-\delta)^2}
\]
and for all $t\in[0,s_6-s_5]$, $\beta_2(s_5+t)=1$,
\[\beta_0(s_5+t)=1-\alpha\,\frac{2\delta-\tau}{\tau-\delta}+(2\delta-\tau)\,t\quad\text{and}\quad
\beta_1(s_5+t)=1-(\tau-\delta)\,t.\]
Hence
$\ \beta_0(s_6)=1,\ \beta_1(s_6)=1-\alpha,\ \beta_2(s_6)=1$.
We recognize the initial condition as in Phase 4, case 2(a)(ii). Therefore, the system behaves periodically starting from time $s_3$,
with period
\[
s_6-s_3=\frac{2\alpha}{\tau-\delta}+\alpha\,\frac{2\delta-\tau}{(\tau-\delta)^2}.
\]

\subsubsection{Case 2(b): $2\delta>3$}
\label{sec:case-2b}

\noindent \textbf{Phase 3: time interval $[s_2,s_3]$.} In this case, the trait $2\delta$, which replaces the former resident trait
$\delta$ at time $s_2$, cannot survive alone and becomes dominant. So $\beta_2(s_2+t)$ does not remain equal to 1, but decreases with
slope $\widehat{S}(2\delta;2\delta)$. Recall that, at time $s_2$, the exponents are given by~\eqref{eq:phase3-fin-fin}.
The fitnesses now become:
\[\widehat{S}(0 ; 2\delta)=3-\tau,\qquad \widehat{S}(\delta ; 2\delta)=3-\delta-\tau<0,\qquad \widehat{S}(2\delta ; 2\delta)=3-2\delta<0.\]
Note that we do not distinguish yet on the sign of $\widehat{S}(0;2\delta)$, which may be either positive or negative in this case.
We obtain
\[
\beta_0(s_2+t)=[1-\alpha+(3-\tau)\,t]\vee 0,\quad
\beta_1(s_2+t)=[1-(\tau+\delta-3)\,t]\vee [1-2\alpha+(3-\tau)\,t]\vee 0
\]
and
$\
\beta_2(s_2+t)=1-(2\delta-3)\,t\ $
until either $\beta_0(t)=\beta_2(t)>0$, which corresponds to a change of dominant population with exponent smaller than 1, or
$\beta_2(t)=0$, which corresponds to the extinction of the whole population. Note that we cannot have $\beta_1(t)=\beta_2(t)$ before
$\beta_0(t)=\beta_2(t)$ since $\tau+\delta-3>2\delta-3$.
One can easily check that $\beta_2(t)$ hits 0 before crossing the curve $\beta_0(t)$ if and only if
$\frac{2\delta-\tau}{2\delta-3}<\alpha$. Of course, this cannot occur if $\tau<3$, since in this case $\widehat{S}(0;2\delta)>0$.
\medskip

\noindent \textbf{Phase 3, case 2(b)(i): $\frac{2\delta-\tau}{2\delta-3}<\alpha$.} In this case, the whole population gets extinct at time
\[
s_3=s_2+\frac{1}{2\delta-3}=\frac{2\alpha}{\tau-\delta}+\frac{1}{2\delta-3}.
\]
\medskip

\noindent \textbf{Phase 3, case 2(b)(ii): $\frac{2\delta-\tau}{2\delta-3}>\alpha$.} Trait 0 becomes dominant and
replaces trait $2\delta$ at time
\[
s_3=s_2+\frac{\alpha}{2\delta-\tau}=\frac{2\alpha}{\tau-\delta}+\frac{\alpha}{2\delta-\tau}.
\]
We obtain the exponents
$\
\beta_0(s_3)=\beta_2(s_3)=1-\alpha\,\frac{2\delta-3}{2\delta-\tau}\in(0,1)
\ $
and
\[
\beta_1(s_3)=\left[1-\alpha\,\frac{\delta+\tau-3}{2\delta-\tau}\right]\vee\left[
  1-\alpha\,\frac{4\delta-\tau-3}{2\delta-\tau}\right]\vee 0.
\]
\medskip

\noindent\textbf{Phase 4, case 2(b)(ii).}
We obtain the new fitnesses
\[
\widehat{S}(0 ; 0)=3,\qquad \widehat{S}(\delta ; 0)=3+\tau-\delta>3,\qquad S(2\delta ; 0)=3-2\delta+\tau\in(0,3).
\]
To compute $\beta_1(s_3+t)$, one needs to distinguish whether $\beta_1(s_3)>0$ or $\beta_1(s_3)=0$. In the last case, one needs to
wait until $\beta_0(s_3+t)=\alpha$ before $\beta_1$ starts to increase with slope $3+\tau-\delta$. The phase stops either when $\beta_0(s_3+t)=1$ (re-emergence of trait $0$, which becomes resident once again) or
$\beta_0(t)=\beta_1(t)<1$ (change of dominant trait). A delicate case study shows that the first case occurs if $\tau-\delta<3/2$ when $\beta_1(s_3)>0$, or if $\tau-\delta<3\alpha/(1-\alpha)$ when $\beta_1(s_3)=0$,
and the second case if $\tau-\delta>3/2$ when $\beta_1(s_3)>0$, or if $\tau-\delta>3\alpha/(1-\alpha)$ when $\beta_1(s_3)=0$. In the first case, one needs to proceed with similar computations as in the first phases.
In the second case, either trait $\delta$ re-emerges first (i.e.\ becomes resident again), or trait $2\delta$ becomes dominant once
again. Explicit computations of the subsequent dynamics are very lengthy. This case is illustrated by Figure~\ref{fig:simu2}(b).

\subsection{Case 3: $2\delta<\tau$}
\label{sec:case-3}

We proceed similarly as in Case 2.
\medskip

\noindent\textbf{Phase 1: time interval $[0,s_1]$.} This phase is the same as in Case 2. The resident trait is 0 and the fitnesses of
the three traits are given by
\[
S(0;0)=0,\quad S(\delta;0)=\tau-\delta>0,\quad S(2\delta;0)=\tau-2\delta>0.
\]
We obtain, for all $t\in [0,s_1]$, $\beta_0(t)=1$,
\[
\beta_1(t)=1-\alpha+(\tau-\delta)\,t
\quad\text{and}\quad
\beta_2(t)=1-2\alpha+(\tau-\delta)\,t
\]
where $\ s_1=\frac{\alpha}{\tau-\delta}$.
Thus $\ \beta_0(s_1)=\beta_1(s_1)=1$ and $\beta_2(s_1)=1-\alpha$.
\medskip

\noindent\textbf{Phase 2: time interval $[s_1,s_2]$.} This phase is also the same as in Case 2. The resident trait is $\delta$ and the
fitnesses are given by
\[
S(0;\delta)=-(\tau-\delta)<0,\quad S(\delta;\delta)=0,\quad S(2\delta;\delta)=\tau-\delta>0.
\]
We obtain, for all $t\in [0,s_2-s_1]$, $\beta_1(s_1+t)=1$,
\[
\beta_0(s_1+t)=1-(\tau-\delta)\,t
\quad\text{and}\quad
\beta_2(s_1+t)=1-\alpha+(\tau-\delta)\,t,
\]
where $\
s_2=s_1+\frac{\alpha}{\tau-\delta}$.
Thus $\ \beta_1(s_2)=\beta_2(s_2)=1$ and $\beta_0(s_2)=1-\alpha$.
\medskip

Again, we have to separate the cases $2\delta<3$ and $2\delta>3$.

\subsubsection{Case 3(a): $2\delta<3$}
\label{sec:case-3a}

\noindent\textbf{Phase 3: time interval $[s_2,+\infty)$.} The resident trait is $2\delta$ and the
fitnesses are given by
\[
S(0;2\delta)=-(\tau-2\delta)<0,\quad S(\delta;2\delta)=-(\tau-\delta)<0,\quad S(2\delta;2\delta)=0.
\]
We obtain, for all $t\geq 0$, $\beta_2(s_2+t)=1$, $\
\beta_0(s_2+t)=[1-\alpha-(\tau-2\delta)\,t]\vee 0\ $
and
\[
\beta_1(s_2+t)=[1-(\tau-\delta)\,t]\vee [1-2\alpha-(\tau-2\delta)\,t]\vee 0.
\]
Therefore, $2\delta$ remains the resident trait forever.

\subsubsection{Case 3(b): $2\delta>3$}
\label{sec:case-3b}

\noindent\textbf{Time interval $[s_2,+\infty)$.} The resident trait $2\delta$
cannot survive by itself. Hence, the fitnesses are now given by
\[
\widehat{S}(0;2\delta)=3-\tau<0,\quad \widehat{S}(\delta;2\delta)=3-\delta-\tau<0,\quad \widehat{S}(2\delta;2\delta)=3-2\delta<0.
\]
Since $\widehat{S}(0;2\delta)<\widehat{S}(2\delta;2\delta)$ and $\widehat{S}(\delta;2\delta)<\widehat{S}(2\delta;2\delta)$, Phase 3
will end when $\beta_2(t)$ hits 0, i.e.\ when the population gets extinct. Hence, for all $t\geq 0$,
$\beta_0(s_2+t)=[1-\alpha+(3-\tau)\,t]\vee 0$, $\beta_1(s_2+t)=[1-(\tau+\delta-3)\,t]\vee [1-2\alpha+(3-\tau)\,t]\vee 0\ $
and
$\  \beta_2(s_2+t)=[1-(2\delta-3)\,t]\vee 0$,
and the extinction time is
$\
s_3=s_2+\frac{1}{2\delta-3}$.

\section{Proof of Theorem \ref{thm:transfer-main-new}}
\label{sec:proof-main}

\subsection{Main ideas of the proof}\label{sec:mainidea}

Let $T>0$ be fixed during the whole proof. We start from the stochastic birth and death process with mutation, competition and transfer, $(N^K_0(t),\ldots,N^K_L(t))$. Our goal is to study the limit behaviour of the vector $(\beta^K_0(t),\ldots,\beta^K_L(t))$ defined in~\eqref{def:beta}.

Theorem \ref{thm:transfer-main-new} will be obtained by a fine comparison of the size of each subpopulation defined by a given trait
value with carefully chosen branching processes with immigration. The stochastic dynamics consists in a succession of steps, composed
of long phases $[\sigma^K_k\log K,\theta^K_k\log K]$ for $k\geq 1$ (with $\sigma_1^K=0$) followed by short intermediate phases
$[\theta^K_k\log K,\sigma^K_{k+1}\log K]$. In each long phase, there is a single dominant or resident trait. Short intermediate
phases correspond to the replacement of the resident or dominant trait, where two subpopulations are of maximal order. We will prove
that $\theta^K_k$ converges in probability to $s_k$, $k\geq 1$. In the limit, intermediate steps vanish on the time scale $\log K$.
The proof will proceed by induction on $k$ until the occurrence of three particular events at some step $k_0$: the
  exponents of three traits become maximal simultaneously (case (ii)(a) in Theorem~\ref{thm:transfer-main-new}), extinction (case
  (ii)(b)), or the exponent of some trait vanishes at the same time as a change of resident or dominant population (case (ii)(c)). We
  then stop the induction and set $T_0=s_{k_0}$ in cases (a) and (c) or $T_0=+\infty$ in case (b).

Four cases need to be distinguished for the inductive definition of $\sigma_k^K$ and $\theta^K_k$, depending on whether there is a dominant or resident trait at the beginning and the end of each step. When there is a resident (resp. dominant) trait during step $k$, the stopping time $\theta_k$ will be defined as the first time when its size exits a neighborhood of its equilibrium density (resp. its exponent exits a neighborhood of its limit), or when the other subpopulations stop to be negligible with respect to the resident (resp. dominant) subpopulation size. To quantify the latter condition, we introduce a parameter $m>0$ which will be fixed during the proof (see \eqref{teta0}, \eqref{teta0bis}, \eqref{teta0chapeau} and Remark~\ref{rem:m-universel}). When the next trait with higher exponent is resident (resp. dominant), the stopping time $\sigma^K_{k+1}$ converges to $s_{k}$ (resp. to $s_k+s$ for a small parameter $s>0$). The proof will be completed by letting $s$ converge to 0.

To control the exponents $\beta^K_\ell(t)$, we proceed by a double induction, first on the steps, and second, inside each step, on the traits $\ell\delta$, for $\ell=0$ to $\ell=L$.  The exponents are approximately piecewise affine. Changes of slopes may happen when a new trait emerges, when a trait dies or when the dynamics of a trait becomes driven by incoming mutations. We use asymptotic results on branching processes with immigration detailed in Appendix~\ref{sec:BPI} to control the sizes of the non-dominant subpopulations. The main result used for phases inside steps is Theorem~\ref{thm:BPI-general}.
% and Lemma~\ref{lem:sable-popu}(i).\\
 During intermediate phases, we use comparisons with dynamical systems, see Lemmas~\ref{lem:competitionTBDI} and~\ref{lem:competition}.\\

\subsection{Step 1}
\label{sec:phase-1}

Let us begin the induction on the steps and start with Phase 1, corresponding to the time interval $[0,\theta^K_1\log K]$. During this phase, the trait $\ell^*_1 \delta=0$ is resident. We introduce a parameter $\varepsilon_1>0$ and we choose $K$ large enough so that
$N^K_0(0)\in \Big[\big(\frac{3}{C}-\varepsilon_1\big)K,\big(\frac{3}{C}+\varepsilon_1\big)K\Big]$. We define
\begin{multline}
\theta_1^K=\inf\left\{t\geq 0 : N^K_{0}(t \log K)\not\in \Big[\big(\frac{3}{C}-3\varepsilon_{1}\big)K,\big(\frac{3}{C}+3\varepsilon_{1}\big)K\Big]\right.\\
\left.\text{ or }\sum_{\ell\neq 0} N^K_\ell(t \log K)\geq m\varepsilon_{1} K\right\}.\label{teta0}
\end{multline}
For the chosen initial condition \eqref{eq:transfer-init-cond},
$\beta_\ell(0)=(1-\ell\alpha)_+$. We distinguish two cases: either $\tau<\delta$ or $\delta<\tau$.

\subsubsection{Case $\tau<\delta$: a single phase}
\label{sec:phase-1-case-1}

Using formulas \eqref{eq:evol-beta_0} and \eqref{eq:evol-beta_j} recursively, we obtain
\[
\begin{array}{ll}
 \beta_0(t)= 1 & t_{0,1}=0\\
\beta_1(t)=  \big(1-\alpha+(\tau-\delta)t\big)\vee (1-\alpha) \vee 0 = 1-\alpha, & t_{1,1}=0\\
\quad \vdots & \quad \vdots \\
\beta_\ell(t)=  \big(1-\ell\alpha+(\tau-\ell\delta)t\big)\vee (1-\ell\alpha) \vee 0 = (1-\ell\alpha)_+, & t_{\ell,1}=0,
\end{array}
\]where we recall that the times $t_{1,1},\cdots t_{\ell,1}$ have been defined in \eqref{eq:def-t-ell-k}.
From \eqref{eq:rec-sk}, we obtain $s_1=+\infty$, and $\beta(t)$ is constant. This is due to the fact that all non-resident traits have negative fitnesses and their size is kept of constant order due to mutations from the resident trait 0.

\bigskip

In the sequel, we denote by $BP_K(b,d,\beta)$ the
distribution of the branching process, with individual birth rate $b\geq 0$, individual death rate $d\geq 0$ and initial
value $\lfloor K^\beta-1 \rfloor \in\N$. We refer to Appendix~\ref{sec:linear-BDP}
for the classical properties of $BP_K(b,d,\beta)$ which will be used to obtain the following results. \\
We also denote by $BPI_K(b,d,a,c,\beta)$ the distribution of a branching process with immigration, with individual birth rate $b\geq 0$, individual death rate $d\geq 0$, immigration rate $K^{c}e^{as}$  at time $s\geq 0$ with $a,c\in\RR$ and same initial value. We refer to Appendix \ref{sec:BPI}.\\
Similarly, $LBDI_K(b,d,C,\gamma)$ is the distribution of a one-dimensional logistic birth and death process (see
Appendix~\ref{sec:logistic-BDP}) with individual birth rate $b\geq 0$ and individual death rate $d+Cn/K$
when the population size is $n$ and with immigration at predictable rate $\gamma(t)\geq 0$ at time $t$.\\

In order to couple the population with branching processes with immigration, we start with computing the arrival and death rates in the subpopulation of trait $\ell \delta$, for $\ell\geq 0$.
For $K$ large enough, arrivals in this population due to reproduction of trait $\ell\delta$ or transfer occur
at time $t\leq\theta_1^K \wedge T$
at rate $N^K_\ell(t \log K)\left[(4-\ell\delta)(1-K^{-\alpha})+\tau\ \frac{\sum_{\ell'=0}^{\ell-1}N^K_{\ell'}(t \log K)}{\sum_{\ell'=0}^L N^K_{\ell'}(t \log K)}\right]$ satisfying
\begin{multline*}
N^K_\ell(t \log K)\left[4-\ell\delta-\varepsilon_1
+\tau\frac{3-3C\varepsilon_1}{3+C(3+m)\varepsilon_1}\right]\\
\leq N^K_\ell(t\log K)\left[(4-\ell\delta)(1-K^{-\alpha})+\tau\ \frac{\sum_{\ell'=0}^{\ell-1}N^K_{\ell'}(t \log K)}{\sum_{\ell'=0}^L N^K_{\ell'}(t \log K)}\right]
 \leq N^K_\ell(t \log K)\left[4-\ell\delta+\tau\right].
\end{multline*}
Arrivals due to incoming mutations from trait $(\ell-1)\delta$ occur at time $t\leq \theta^K_1\wedge T$ at rate
$N_{\ell-1}^K(t\log K)(4-(\ell-1)\delta)K^{-\alpha}$.\\
Deaths occur at rate $N^K_\ell(t \log K)\left[1+\frac{C}{K}\sum_{\ell'=0}^L N^K_{\ell'}(t \log K)\right]$ satisfying
\[
N^K_\ell(t \log K)\left(4-3C\varepsilon_1\right)\leq N^K_\ell(t \log K)\left[1+\frac{C}{K}\sum_{\ell'=0}^L N^K_{\ell'}(t \log K)\right]\leq N^K_\ell(t \log K)\left(4+C(3+m)\varepsilon_1\right).
\]

\bigskip

\noindent \textbf{Step 1} Let us prove by induction on $\ell\geq 0$ the following bounds on the mutation rates: for all $t\leq \theta^K_1\wedge T$,
with probability converging to 1 as $K\rightarrow+\infty$,
\begin{equation}
\label{eq:control-br}
K^{\beta_{\ell}(t)-\alpha-(\ell+1)\varepsilon_{1}}\leq N^K_{\ell}(t \log K)(4-\ell\delta) K^{-\alpha}  \leq
K^{\beta_{\ell}(t) -\alpha +(\ell+1)\varepsilon_{1}}.
\end{equation}

For $\ell=0$, by definition of $\theta_1^K$,
\begin{equation}
\label{eq:intial-bound}
K^{1-\alpha-\varepsilon_{1}} \leq N^K_{0}(t \log K)\,4 K^{-\alpha} \leq K^{1-\alpha+\varepsilon_{1}}
\end{equation}
is clear. To proceed to $\ell=1$, we use standard coupling arguments to obtain
\begin{equation}
\label{eq:control-dr}
Z^K_{1,1}(t\log K)\leq N^K_1(t \log K)\leq Z^K_{1,2}(t\log K),\quad\forall t\leq\theta^K_1 \wedge T,
\end{equation}
where $Z^K_{1,1}$ is a $BPI_K(4-\delta+\tau-\overline{C}\varepsilon_1,4+\overline{C}\varepsilon_1,0, 1-\alpha-\varepsilon_{1},1-\alpha-\varepsilon_{1})$
and $Z^K_{1,2}$ is a $BPI_K(4-\delta+\tau,4-\overline{C}\varepsilon_1,0, 1-\alpha+\varepsilon_{1},1-\alpha +\varepsilon_{1})$, with
\[
\overline{C}=1+(1\vee \tau)C(6+m).
\]
Note that the addition of $\varepsilon_{1}$ in the coefficient $\beta$ of the upper-bounding branching
process ensures that $c\leq \beta$, so that we can use Theorem~\ref{thm:BPI-general}(i) to deduce that as $K$ tends to infinity, since $\tau-\delta<0$,
\[
\frac{\log(1+Z^K_{1,1}(t\log K))}{\log K}\longrightarrow{}\left[1-\alpha-\varepsilon_{1}+(\tau-\delta-2\overline{C}\varepsilon_1)t\right]\vee
[1-\alpha-\varepsilon_{1}]\geq 1-\alpha-\varepsilon_{1}=\beta_1(t)-\varepsilon_1
\]
and
\[
\frac{\log(1+Z^K_{1,2}(t\log K))}{{\log
    K}}\longrightarrow{}\left[1-\alpha+\varepsilon_{1}+(\tau-\delta+\overline{C}\varepsilon_1)t\right]\vee
[1-\alpha+\varepsilon_1]\leq 1-\alpha+\varepsilon_{1}=\beta_1(t)+\varepsilon_{1},
\]
in $L^\infty([0,T])$ and provided that $\varepsilon_1$ is small enough.\\

Assume now that the induction hypothesis \eqref{eq:control-br} is true for $\ell-1\geq 1$ and let us prove it for $\ell$. Either $\ell\leq 1+\lfloor 1/\alpha\rfloor $ and we have for all $t\leq \theta^K_1 \wedge T$, with probability converging to 1 as $K\rightarrow+\infty$,
\begin{equation*}
K^{1-(\ell-1)\alpha -\alpha-\ell\varepsilon_{1}}\leq N^K_{\ell-1}(t \log K)(4-(\ell-1)\delta) K^{-\alpha}  \leq
K^{1-(\ell-1)\alpha -\alpha +\ell\varepsilon_{1}},
\end{equation*}
or $\ell>1+\lfloor 1/\alpha\rfloor $ and $N^K_{\ell-1}(t \log K)(4-(\ell-1)\delta) K^{-\alpha} =0$.
Hence, with probability converging to 1,
\begin{equation}
\label{eq:control-dr-2}
Z^K_{\ell,1}(t\log K)\leq N^K_\ell(t \log K)\leq Z^K_{\ell,2}(t\log K),\quad\forall t\leq\theta^K_1 \wedge T,
\end{equation}
where $Z^K_{\ell,1}$ is a $BPI_K(4-\ell\delta+\tau-\overline{C}\varepsilon_1,4+\overline{C}\varepsilon_1,0,
(1-(\ell-1)\alpha)_+-\alpha-\ell\varepsilon_{1},(1-\ell\alpha-\ell\varepsilon_{1})_+)$ and $Z^K_{\ell,2}$ is a
$BPI_K(4-\ell\delta+\tau,4-\overline{C}\varepsilon_1,0, (1-(\ell-1)\alpha)_+-\alpha+\ell\varepsilon_{1},(1-\ell\alpha+\ell\varepsilon_{1})_+
)$. Note that when $\ell>1+\lfloor 1/\alpha\rfloor$, the lower bound has nonzero but negligible immigration, hence the
comparison is true only on the event where there is no immigration in $Z^K_{\ell,1}$ on $[0,T\log K]$, which has probability converging to 1 (see Lemma~\ref{lem:non-emergence}).

\bigskip

If $\ell\leq\lfloor 1/\alpha\rfloor$, we use Theorem~\ref{thm:BPI-general}~(i) to deduce that as $K$ tends to infinity, uniformly for $t\in [0,T]$,
\[
\frac{\log(1+Z^K_{\ell,1}(t\log K))}{\log K}\longrightarrow{}\left[1-\ell\alpha-\ell\varepsilon_{1}+(\tau-\ell\delta-2\overline{C}\varepsilon_1)t\right]\vee [1-\ell\alpha-\ell\varepsilon_{1}]\geq \beta_\ell(t)-\ell\varepsilon_{1},
\]
and
\[
\frac{\log(1+Z^K_{\ell,2}(t\log K))}{{\log K}}\longrightarrow{}\left[1-\ell\alpha+\ell\varepsilon_{1}+(\tau-\ell\delta+\overline{C}\varepsilon_1)t\right]\vee [1-\ell\alpha+\ell\varepsilon_{1}]\leq \beta_\ell(t)+\ell\varepsilon_{1}.
\]

If $\lfloor 1/\alpha\rfloor+1\leq \ell\leq L$, we use Theorem~\ref{thm:BPI-general}~(iii) and \eqref{eq:BPI-extinction} (assuming that $\varepsilon_{1}$ is small
enough) to prove that $\log(1+Z^K_{\ell,i}(t\log K))/\log K$ converges to $0$ in $L^\infty([0,T])$, and $N^K_\ell(t\log K)=0$ for all $t\leq \theta^K_1 \wedge T$ with probability converging to 1.

\bigskip

We deduce that, with probability
converging to 1,~\eqref{eq:control-br} is satisfied with $\ell-1$ replaced by $\ell$. This completes the proof
of~\eqref{eq:control-br} by induction. In particular, on the time interval $[0,\theta^K_1\wedge T]$, $\log(1+N^K_{\ell}(t\log K))/\log K$
converges to $\beta_\ell(t)=(1-\ell\alpha)_+$.

\bigskip

\noindent \textbf{Step 2} It remains to prove that $\theta_1^K> T$ with probability converging to 1. Using the previous steps and recalling that $\beta_\ell(t)=\beta_\ell(0)=(1-\ell\alpha)_+$, we have,
with probability converging to 1, that $\sup_{t\in[0,2T]}|\beta^K_\ell(t)-\beta_\ell(t)|<\alpha/2$, and thus, for all $t\leq \theta_1^K\wedge 2T$,
\begin{equation}
  \label{eq:damgan-midi}
  \sum_{\ell=1}^L N^K_\ell(t\log K)\leq  K^{\max_{1\leq\ell\leq L,\,t\in[0,2T]} \beta_\ell(t)+\frac{\alpha}{2}}=K^{1-\frac{\alpha}{2}}.
\end{equation}
Hence, we have for all $t\leq\theta^K_1 \wedge 2T$ that
\begin{equation}
\label{eq:control-dr-3}
Z^K_{0,1}(t\log K)\leq N^K_0(t\log K)\leq Z^K_{0,2}(t\log K),
\end{equation}
where $Z^K_{0,1}$ is a $LBDI_K(4(1-K^{-\alpha}),1+CK^{-\alpha/2},C,0)$ and $Z^K_{0,2}$ is a $LBDI_K(4,1,C,0)$. Applying Lemma~\ref{lem:sable-popu}(i) to the processes $Z^K_{0,i}$, $i=1,2$, we deduce that
$\,\theta_1^K> T\,$ with probability converging to 1 when $K\rightarrow +\infty$.

\subsubsection{Case $\tau>\delta$: Phase 1}
\label{sec:phase-1-case-2}

Let us first give the explicit expression of $\beta_\ell$ on the first phase. We shall use the two equivalent formulations
\eqref{eq:evol-beta_j} and \eqref{eq:evol-beta_j-simple}. The fitnesses involved in these expressions are $\ S(0;0)= 0,\  S(\ell \delta;0)= \tau- \ell \delta$, for all $1\leq  \ell\leq L$.
We use formula \eqref{eq:evol-beta_j-simple} recursively from $\ell=1$ to $L$ to prove that, for $\ell\leq \lfloor \frac{1}{\alpha}\rfloor$
\[
\beta_\ell(t)=  \big(1-\ell\alpha+(\tau-\ell\delta)t\big)\vee \big(1-\ell\alpha+(\tau-\delta)t\big) \vee 0 = 1-\ell\alpha+ (\tau-\delta)t,
\]
and $t_{\ell,1}=0$ when $\ell<\lfloor 1/\alpha\rfloor$. \\

When $\ell=\lfloor \frac{1}{\alpha}\rfloor$, we have $\beta_{\lfloor
  1/\alpha\rfloor}(0)\in (0,\alpha)$ and we deduce that $t_{\lfloor 1/\alpha\rfloor,1}={\frac{\alpha(1+\lfloor 1/\alpha\rfloor)-1}{\tau-\delta}}>0$.\\

For $\ell=\lfloor \frac{1}{\alpha}\rfloor+1$, we use~\eqref{eq:evol-beta_j} to prove
\begin{align*}
\beta_{\lfloor 1/\alpha\rfloor+1}(t)=  & \Big((\tau-(\lfloor \frac{1}{\alpha}\rfloor +1)\delta)(t-t_{\lfloor 1/\alpha\rfloor,1})\Big)\vee \Big(1-(\lfloor \frac{1}{\alpha}\rfloor +1)\alpha+(\tau-\delta)t\Big) \vee 0\\
 = &  \Big(1-(\lfloor \frac{1}{\alpha}\rfloor+1)\alpha+ (\tau-\delta)t\Big)_+.
\end{align*}
Similar computation gives the general formula for all $\ell\in \{1,\dots L\}$:
\begin{align}
  \label{eq:Damgan-fin}
  &\beta_\ell(t)=\big(1-\ell\alpha +(\tau-\delta)t\big)_+\ ;
  & t_{\ell,1}=
  \begin{cases}
    0 & \mbox{ if }\ell<\lfloor 1/\alpha\rfloor\\
    \frac{\alpha(\ell+1)-1}{\tau-\delta} & \mbox{ otherwise}.
  \end{cases}
\end{align}
We see that $\beta_1(t)$ is the first exponent to reach 1, so that the duration of the first phase is
\[s_1=\frac{\alpha}{\tau-\delta}.\]
All the exponents $\beta_\ell$ are affine functions on $[0,s_1]$, except for $\ell=\lfloor 1/\alpha\rfloor+1$ where there is a change
of slope at time $t_{\lfloor 1/\alpha\rfloor,1}$.

\medskip

%%%%%%%%%%

The comparisons of arrival and death rates of $N_\ell^K(t\log K)$ are the
same as in Section \ref{sec:phase-1-case-1}. We proceed by induction over $\ell\geq 0$ as above.
\medskip

For $\ell=0$,~\eqref{eq:intial-bound} remains true by definition of $\theta^K_1$. To proceed to $\ell=1$, we observe that~\eqref{eq:control-dr} remains valid and Theorem~\ref{thm:BPI-general}(i) gives as before that as $K$ tends
to infinity, since $\tau-\delta>0$,
\[
\frac{\log(1+Z^K_{1,1}(t\log K))}{\log K}\longrightarrow{}\left[1-\alpha-\varepsilon_{1}+(\tau-\delta-2\overline{C}\varepsilon_1)t\right]\vee
[1-\alpha-\varepsilon_{1}]\geq \beta_1(t)-(1+2\overline{C} T)\varepsilon_1,
\]
and
\[
\frac{\log(1+Z^K_{1,2}(t\log K))}{{\log
    K}}\longrightarrow{}\left[1-\alpha+\varepsilon_{1}+(\tau-\delta+\overline{C}\varepsilon_1)t\right]\vee
[1-\alpha+\varepsilon_1]\leq \beta_1(t)+\varepsilon_{1}(1+\overline{C}T),
\]
in $L^\infty([0,s_1\wedge T])$ and provided that $\varepsilon_1$ is small enough.\\
\medskip

For $2\leq \ell\leq 1+\lfloor 1/\alpha\rfloor$, the induction relation \eqref{eq:control-br} is modified as follows. For
all $\ell\geq 1$, let
\begin{equation}
  \label{eq:constant-horrible}
  C_\ell=\ell+2\overline{C}T.
\end{equation}
We assume that, for all $t\leq\theta^K_1\wedge s_1\wedge T$,
\begin{multline}
\beta_{\ell-1}(t)-\alpha-C_{\ell-1} \varepsilon_1
\leq \frac{\log(N^K_{\ell-1}(t \log K)(4-(\ell-1)\delta) K^{-\alpha})}{\log K}
 \leq \beta_{\ell-1}(t)-\alpha +C_{\ell-1} \varepsilon_{1}.
\label{eq:control-br-cas2}
\end{multline}
Hence, with probability converging to 1,
\begin{equation}
\label{eq:control-dr-bis}
Z^K_{\ell,1}(t\log K)\leq N^K_\ell(t \log K)\leq Z^K_{\ell,2}(t\log K),\quad\forall t\leq\theta^K_1\wedge s_1\wedge T,
\end{equation}
where $Z^K_{\ell,1}$ is a $BPI_K(4-\ell\delta+\tau-\overline{C}\varepsilon_1,4+\overline{C}\varepsilon_1,\tau-\delta,
1-\ell\alpha- C_{\ell-1}\varepsilon_{1},(1-\ell\alpha-C_{\ell-1}\varepsilon_{1})_+)$ and $Z^K_{\ell,2}$ is a
$BPI_K(4-\ell\delta+\tau,4-\overline{C}\varepsilon_1,\tau-\delta,
1-\ell\alpha+C_{\ell-1}\varepsilon_{1},(1-\ell\alpha +C_{\ell-1}\varepsilon_{1})_+
)$. \\

Note that, in this Phase 1, $\beta_{\ell-1}(t)$ is affine on $[0,s_1]$, so we can apply Theorem~\ref{thm:BPI-general} on the whole
interval $[0,s_1]$. If $\ell\leq\lfloor 1/\alpha\rfloor$, we apply Theorem~\ref{thm:BPI-general}~(i) and if $\ell=1+\lfloor
1/\alpha\rfloor$, $\beta_\ell(0)=0$ so we apply Theorem~\ref{thm:BPI-general}~(ii), using that
$a=\tau-\delta>r=\tau-\ell\delta+\overline{C}\varepsilon_1$ for $\varepsilon_1$ small enough. We deduce that, in both cases, as $K$
tends to infinity, for all $t\leq\theta^K_1\wedge s_1\wedge T$,
\begin{equation}
\label{etape4}
(\beta_\ell(t)-C_{\ell-1}\varepsilon_1)_+
\leq \frac{\log(1+N_{\ell}^K(t\log K))}{\log K}
\leq (1-\ell\alpha-C_{\ell-1}\varepsilon_1+(\tau-\delta)t)_+\leq \beta_{\ell}(t)+C_{\ell-1}\varepsilon_1.
\end{equation}
Hence, we have proved~\eqref{eq:control-br-cas2} for $\ell+1$ and the induction step is complete. We also deduce
from~\eqref{eq:BPI-extinction} that $N^K_{1+\lfloor 1/\alpha\rfloor}(t\log K)=0$ for all $t$ in a closed interval of $[0,s_1]$
included in the complement of the support of $(1-\ell\alpha-C_{\ell-1}\varepsilon_1+(\tau-\delta)t)_+$.
\medskip

For $\ell=2+\lfloor 1/\alpha\rfloor$, because $\beta_{1+\lfloor 1/\alpha\rfloor}$ has a change of slope at time
\[
t_{{\lfloor 1/\alpha\rfloor,1}}=\frac{\alpha-(1-\lfloor 1/\alpha\rfloor\alpha)}{\tau-\delta},
\]
the bounds~\eqref{eq:control-br-cas2} on the immigration rate do not allow to apply Theorem~\ref{thm:BPI-general} on the whole
interval $[0,s_1]$. Instead, we first apply them successively on the intervals $[0,t_{{\lfloor 1/\alpha\rfloor,1}}]$ and $[t_{{\lfloor 1/\alpha\rfloor,1}},s_1]$. On the first interval, we have the
bounds
\begin{equation*}
Z^K_{\ell,1}(t\log K)\leq N^K_\ell(t \log K)\leq Z^K_{\ell,2}(t\log K),\quad\forall t\leq\theta^K_1\wedge t_{{\lfloor 1/\alpha\rfloor,1}}\wedge T,
\end{equation*}
where $Z^K_{\ell,1}$ is a $BPI_K(4-\ell\delta+\tau-\overline{C}\varepsilon_1,4+\overline{C}\varepsilon_1,0,
-\alpha-C_{\ell-1}\varepsilon_1,0)$ and $Z^K_{\ell,2}$ is a $BPI_K(4-\ell\delta+\tau,4-\overline{C}\varepsilon_1,0,
-\alpha+C_{\ell-1}\varepsilon_1,0)$. We apply Theorem~\ref{thm:BPI-general}~(iii) to deduce that, with probability converging to 1,
$N^K_\ell(t\log K)=0$ for all $t\in[0,t_{{\lfloor 1/\alpha\rfloor,1}}\wedge \theta^K_1\wedge T]$. \\
On the second interval, we obtain similar bounds with $Z^K_{\ell,1}$ a
$BPI_K(4-\ell\delta+\tau-\overline{C}\varepsilon_1,4+\overline{C}\varepsilon_1,\tau-\delta, -\alpha-C_{\ell-1}\varepsilon_1,0)$ and
$Z^K_{\ell,2}$ is a $BPI_K(4-\ell\delta+\tau,4-\overline{C}\varepsilon_1,\tau-\delta, -\alpha+C_{\ell-1}\varepsilon_1,0)$. Because
$(\tau-\delta)(s_1-t_{{\lfloor 1/\alpha\rfloor,1}})<\alpha$ , we apply Theorem~\ref{thm:BPI-general}~(ii) to deduce that, with probability converging to
1, $N^K_\ell(t\log K)=0$ for all $t\in[t_{{\lfloor 1/\alpha\rfloor,1}}\wedge \theta^K_1\wedge T,s_1\wedge \theta^K_1\wedge T]$.
\medskip

For $\ell>2+\lfloor 1/\alpha\rfloor$, we proceed similarly to prove by induction that, with probability converging to
1, $N^K_\ell(t\log K)=0$ for all $t\leq s_1\wedge \theta^K_1\wedge T$.
\medskip

Using the comparison argument with logistic birth-death processes~\eqref{eq:control-dr-3} on the interval
$[0,\theta^1_K\wedge(s_1-\eta)\wedge T]$, we can prove as in the previous section that, for all $\eta>0$,
$\theta^K_1>(s_1-\eta)\wedge T$ with probability converging to 1.
\medskip

To conclude, since $\varepsilon_1$ in~\eqref{etape4} is arbitrary, we have proved that, for all $\eta>0$,
$\log(1+N^K_{\ell}(t\log K))/\log K$ converges to $\beta_\ell(t)$ in $L^\infty([0,(s_1-\eta)\wedge T])$.

\subsubsection{Case $\tau>\delta$: Intermediate Phase 1}
\label{sec:intermediate-1}

The goal of the intermediate phase is to prove that, on a time interval $[\theta^K_1\log
K,\theta^K_1\log K+T(\varepsilon_1,m)]$ with $T(\varepsilon_1,m)$ to be defined below, the two traits $\ell^*_1\delta=0$ and
$\ell^*_2\delta=\delta$ are of size-order $K$ and population with trait $0$ is declining below a small threshold while population
with trait $\delta$ reaches a neighborhood of its equilibrium $K\bar{n}(\delta)=\frac{(3-\delta)K}{C}$.
\medskip

Let us first show that $\theta_{1}^K <s_1+\eta$ with probability converging to one, for
any $\eta>0$. For this, we observe that the bounds of~\eqref{eq:control-br-cas2} and~\eqref{eq:control-dr-bis} are actually true for all
$t\leq \theta^K_1\wedge T$ provided we use $(1-\ell\alpha+(\tau-\delta) t)_+$ instead of $\beta_\ell(t)$
inside~\eqref{eq:control-br-cas2}, since $\beta_\ell(t)$ is constructed only on $[0,s_1]$ in Phase 1. This gives with high
probability for all $t\leq \theta^K_1\wedge T$
\begin{equation}
  \label{eq:midi-trente}
  (1-\ell\alpha+(\tau-\delta)t)_+-C_\ell\varepsilon_1\leq\frac{\log(1+N^K_\ell(t\log K))}{\log K}\leq (1-\ell\alpha+(\tau-\delta)t)_+
  +C_\ell\varepsilon_1.
\end{equation}

In particular, if $\theta_{1}^K \geq s_1+\eta$ with positive probability, we would obtain, with high probability on this event
\[
\frac{\log(1+ N^K_1((s_1+\eta)\log K))}{\log K}\geq 1-\alpha-C_1\varepsilon_1+(\tau-\delta)(s_1+\eta),
\]
which is larger than 1 provided $\varepsilon_1$ is small enough. This would contradict the assumption that $\theta_1^K>s_1+\eta$.
Hence, $\lim_{K\rightarrow+\infty} \theta^K_1=s_1$ in probability.
\medskip

In the previous phase, we used bounds on $N^K_0(t\log K)$ until time $s_1-\eta$. We now need to extend them until $\theta^K_1$. In
this case,~\eqref{eq:damgan-midi} is not true anymore. Therefore, assuming $K$ large enough to get $4K^{-\alpha}<Cm\varepsilon_1$, we
couple $N^K_0(t)$ with two logistic processes $Z^K_{0,1}$ and $Z^K_{0,2}$ of respective laws $LBDI_K(4,1,C,0)$ and $LBDI_K(4-
Cm\varepsilon_1,1+\frac{\tau m\varepsilon_1}{3/C-3\varepsilon_1}+Cm\varepsilon_1,C,0)$: $Z^K_{0,2}\leq N^K_0\leq Z^K_{0,1}$. The equilibrium density of $Z^K_{0,1}$ is
$3/C$ but the one of $Z^K_{0,2}$ is
\[
\bar{z}_{0,2}:=\frac{3}{C}-\varepsilon_1\left(2m+\frac{\tau m}{3-3\varepsilon_1 C}\right).
\]
We choose $m$ small enough to have
\begin{equation}
  \label{eq:ineg-m}
0<  2m+\frac{\tau m}{3-3\varepsilon_1 C}<\frac{1}{3}.
\end{equation}
Then, observing that $Z^K_{0,2}(0)\in \Big[\big(\bar{z}_{0,2}-4\varepsilon_1/3\big)K,\big(\bar{z}_{0,2}+4\varepsilon_1/3\big)K\Big]$
we can apply Lemma \ref{lem:sable-popu}(i) to $Z^K_{0,2}$ with $\varepsilon=4\varepsilon_1/3$ to deduce that
\[
\lim_{K\to \infty}\mathbb{P}\Big(\forall t\in [0, s_1+\eta], \frac{Z^K_{0,2}(t\log K)}{K}\leq \frac{3}{C}
+ 3\varepsilon_1\Big) = 1.
\]
Applying similarly Lemma \ref{lem:sable-popu}(i) to $Z^K_{0,1}$, we obtain
\[
\lim_{K\to \infty}\mathbb{P}\Big(\frac{N^K_{0}(\theta^K_1\log K)}{K}\in [\frac{3}{C} - 3\varepsilon_1, \frac{3}{C}
+ 3\varepsilon_1]\Big) = 1.
\]
Hence, by construction of $\theta^K_1$, we have with probability converging to 1
\[
\sum_{\ell=1}^L N^K_\ell(\theta_1^K \log K)\geq m \varepsilon_1 K.
\]

\begin{rem}
  \label{rem:m-universel}
  The constraint~\eqref{eq:ineg-m} on the constant $m>0$ ensures that, with high probability, at time
  $\theta^K_1$, trait $0$ is still close to its equilibrium and the second resident trait $\delta$ has emerged. Similar constraints on $m$ can be defined for any other resident trait $\ell\delta<3$. In all the proof, we assume that such $m>0$ is chosen.
\end{rem}

Using Formula \eqref{eq:midi-trente} and the convergence of $\theta^K_1$ to $s_1$, we show immediately that, with probability
converging to 1,
\begin{equation}
  \label{eq:1}
  \sum_{\ell=2}^L N^K_\ell(\theta_1^K \log K)\leq K^{1-\alpha/2}.
\end{equation}
Hence $N^K_1(\theta_1^K \log K)\geq m\varepsilon_1 K/2$ for $K$ large enough.

\medskip

It is now possible to use the Markov property by conditioning at time $\theta_1^K \log K$. We distinguish whether the emerging trait
$\delta$ can survive on its own ($\delta<3$) or not ($\delta>3$).

\paragraph{Case $\tau>\delta$: Intermediate Phase 1, case $\delta<3$}
\label{sec:intermediate-1-case-1}

For $\delta<3$, we first claim that there exists $s>0$ small enough such that
\[
\sum_{\ell=2}^L N^K_\ell(t \log K)\leq K^{1-\alpha/4},\quad\forall\, t\in[\theta^K_1,\theta^K_1+s].
\]
This can be obtained from the continuity argument of Lemma~\ref{lem:exposant-bouge-pas-trop}. Then, we can apply Lemma \ref{lem:competition}(i) with
\begin{gather*}
b_0^K(t)=4(1-K^{-\alpha}),\quad
b_1^K(t)=(4-\delta)(1-K^{-\alpha}),\\
d_0^K(t)=d_1^K(t)=1+\left(\frac{C}{K}+\frac{\tau}{\sum_{\ell=0}^L N^K_\ell(t)}\right)\sum_{\ell=2}^L N^K_\ell(t),\\
\tau^K(t)=\tau\frac{N^K_0(t)+N^K_1(t)}{\sum_{\ell=0}^L N^K_\ell(t)},\quad
\gamma^K_0(t)=0,\quad
\gamma^K_1(t)=4 K^{-\alpha} N^K_0(t),
\end{gather*}
which converge respectively to $b_0=4$, $b_1=4-\delta$, $d_0=d_1=1$, $\tau$, 0 and 0. Note that Point~(i) of Lemma
\ref{lem:competition} applies here since $r_0=b_0-d_0=3>0$, $r_1=b_1-d_1=3-\delta>0$ and $S=\tau-\delta>0$. We obtain that there
exists a finite time $T(m,\varepsilon_1)$ such that with large probability,
\[
N^K_0\big(\theta^K_1 \log K+T(m,\varepsilon_1)\big)\leq
m\varepsilon_1 K\quad \mbox{ and }\quad N^K_1\big(\theta_1^K \log K+T(m,\varepsilon_1)\big)\in
\Big[\frac{3-\delta}{C}-3\varepsilon_1, \frac{3-\delta}{C}+3\varepsilon_1\Big].
\]

Hence, we can define the time at which the first intermediate phase ends as
\[\sigma^K_2{\log K}=\theta^K_1\log K+T(m,\varepsilon_1)\]
{where $\sigma^K_2\rightarrow s_1$} in probability.
Using~\eqref{eq:midi-trente}, the continuity argument of Lemma~\ref{lem:exposant-bouge-pas-trop} and the fact that
$(\tau-\delta)s_1=\alpha$, the population state at time $\sigma^K_2$ satisfies with probability converging
to 1
\begin{align*}
 & K^{1-\varepsilon_1}\leq N^K_0(\sigma^K_2{\log K})\leq m\varepsilon_1 K,\quad \frac{N^K_1(\sigma^K_2{\log K})}{K}\in \Big[\frac{3-\delta}{C}-\varepsilon_2,
\frac{3-\delta}{C}+\varepsilon_2\Big], \\
 & \frac{\log(1+N^K_\ell(\sigma^K_2{\log K}))}{\log
   K}\in\Big[1-(\ell-1)\alpha-\varepsilon_{2},1-(\ell-1)\alpha+\varepsilon_{2}\Big]=[\beta_\ell(s_1)-\varepsilon_2,\beta_\ell(s_1)+\varepsilon_2],
 \\ & \phantom{N^K_\ell(\sigma^K_2{\log K}))=0,\quad}\forall \,2\leq\ell\leq 1+\lfloor 1/\alpha\rfloor,\\
 & N^K_\ell(\sigma^K_2{\log K})=0,\quad\forall\, 2+\lfloor 1/\alpha\rfloor\leq\ell\leq L,
\end{align*}
where
\begin{equation}\label{eq:varepsilon1}
  \varepsilon_2 = (C_L\vee 3)\varepsilon_1.
\end{equation}
We are then ready to proceed to Phase 2 (see Section~\ref{sec:phase-k}).

\paragraph{Case $\tau>\delta$: Intermediate Phase 1, case $\delta>3$}
\label{sec:intermediate-1-case-2}

When $\delta>3$, we need to apply Lemma \ref{lem:competition}(iii) instead of (i) since $r_1=3-\delta<0$.
Using~\eqref{eq:midi-trente} and the continuity argument of Lemma~\ref{lem:exposant-bouge-pas-trop} as above, we obtain for all $s>0$
small enough, with probability converging to 1,
\begin{align*}
 & \frac{\log(1+N^K_\ell((\theta^K_1+s)\log K))}{\log
   K}\in\Big[1-(\ell-1)\alpha-\varepsilon_{2},1-(\ell-1)\alpha+\varepsilon_{2}\Big],\quad\forall\, 2\leq\ell\leq 1+\lfloor 1/\alpha\rfloor,\\
 & N^K_\ell((\theta^K_1+s)\log K))=0,\quad\forall \, 2+\lfloor 1/\alpha\rfloor\leq\ell\leq L,
\end{align*}
with $\varepsilon_2$ defined in \eqref{eq:varepsilon1}. Without loss of generality, we can take $s>0$ small enough to apply Lemma
\ref{lem:competition}(iii).

In this case, we define the end of the first intermediate phase as
\[{\sigma}^K_2=\theta^K_1+s,\]
where $s>0$ is a fixed small parameter and
\begin{align*}
 & K^{1-\varepsilon_2}\leq N^K_0({\sigma}^K_2{\log K})\leq K^{-s\rho} N^K_1({\sigma}^K_2 \log K),\quad N^K_1({\sigma}^K_2{\log K})\leq
 K^{1-s\rho}.
\end{align*}

\subsection{Step $k$}
\label{sec:phase-k}

Our goal is to describe the dynamics on the interval $[\sigma^K_k\log K,\theta^K_k\log K]$, that converges in the $\log K$ scale to $[s_{k-1},s_k]$.
Recall that when $\tau<\delta$ there is only one phase. So now, we only consider $\tau>\delta$.\\

Let $k\geq 2$. Assume that Step $k-1$ is completed {and that $T_0>s_{k-1}$ (otherwise, we stop
  the induction at the end of step $[s_{k-2},s_{k-1}]$)}. Two cases may occur: either the emerging trait becomes resident during Step
$k$, or it becomes only a dominant trait. The latter case can occur when the trait is fit ($\ell^*_{k}\delta<3$) but its
population size is small, or when it is unfit ($\ell^*_{k}\delta>3$).

We proceed as in Step 1 by induction over $\ell\in\{0,\ldots,L\}$. For each $\ell$, we couple $N_\ell^K$ with branching
processes with immigration given by the size of the population $\ell-1$ on each time interval included in $[s_{k-1},s_{k}]$ where $\beta_{\ell-1}$ is affine.

In Step 1, we took care to give bounds involving explicit constants $C_\ell$ for sake of precision. From now on, we use a constant
$C_*$ which may change from line to line.

\subsubsection{Step $k$, case 1}
\label{sec:phase-k+1-case-1}

We assume that, $\max_{0\leq\ell\leq L}\beta_\ell(s_{k-1}+s)=1$ with $\ell^*_{k}\delta<3$, and that this maximum is
attained only for $\ell^*_{k-1}$ and $\ell^*_{k}$, since $s_{k-1}<T_0$.

The induction assumption in this case is as follows: suppose that, for all $\varepsilon_{k}>0$ small enough, we have constructed
{$\sigma_{k}^K$ such that $\sigma_{k}^K\log K$ is} a stopping time
and
\begin{itemize}
\item $\sigma^K_{k}$ converges in probability to $s_{k-1}$;
\item
    $\frac{N^K_{\ell^*_{k}}(\sigma^K_{k}{\log K})}{K}\in\Big[\frac{3-\delta\ell^*_{k}}{C}-\varepsilon_{k},\frac{3-\delta\ell^*_{k}}{C}+\varepsilon_{k}\Big]$;
  \item $K^{1-\varepsilon_{k}}\leq N^K_{\ell^*_{k-1}}(\sigma^K_k{\log K})\leq \frac{m}{2}\varepsilon_{k} K$;
  \item for all $\ell\neq\ell^*_{k-1},\ell^*_{k}$, either $N^K_\ell(\sigma^K_k{\log K})=0$ if $\beta_\ell(s_{k-1})=0$, or otherwise
    \[
    \beta_\ell(s_{k-1})-\varepsilon_k\leq\frac{\log(1+N^K_\ell(\sigma^K_k{\log K}))}{\log K}\leq \beta_\ell(s_{k-1})+\varepsilon_k<1.
    \]
\end{itemize}

We now give the proof of the induction step. Let us define
\begin{multline}
\theta_k^K=\inf\left\{t\geq \sigma_k^K : N^K_{\ell^*_k}(t \log K)\not\in \Big[\big(\frac{3-\ell^*_k\delta}{C}-3\varepsilon_{k}\big)K,\big(\frac{3-\ell^*_k \delta}{C}+3\varepsilon_{k}\big)K\Big]\right.\\
\left.\text{ or }\sum_{\ell\neq \ell^*_k} N^K_\ell(t \log K)\geq m\varepsilon_{k} K\right\}.\label{teta0bis}
\end{multline}

Observe first that, by definition of $\theta^K_k$,
\[
1-\varepsilon_k\leq\frac{\log(1+N^K_{\ell^*_k}(t\log K))}{\log K}\leq 1+\varepsilon_k,\quad\forall t\in[\sigma^K_k,\theta^K_k].
\]
Our goal is to obtain bounds on $N^K_\ell$ by induction on $\ell\neq\ell^*_k$.
\medskip

Induction initialization: If $\ell^*_k=0$, we start the induction at $\ell=1$; otherwise, we start it at $\ell=0$.

In the first case, using the Markov property at time $\sigma^K_k \log K$, where $\sigma^K_k$ converges to $s_{k-1}$, we can proceed exactly as in Phase 1 (see Section \ref{sec:phase-1-case-2}) to prove
that, with high probability for all $t\in[\sigma^K_k,\theta^K_k\wedge T]$,
\begin{equation*}
\beta_{1}(s_{k-1})+(\tau-\delta)(t-s_{k-1})-C_*\varepsilon_k
\leq \frac{\log(N^K_{1}(t \log K))}{\log K}
\leq \beta_{1}(s_{k-1})+(\tau-\delta)(t-s_{k-1}) +C_* \varepsilon_k.
\end{equation*}
Note that, for $s_{k-1}\leq t\leq s_{k}$, $\beta_{1}(s_{k-1})+(\tau-\delta)(t-s_{k-1})=\beta_1(t)$, so we recover bounds of the
form~\eqref{eq:control-br-cas2}.

In the second case, since there is no incoming mutation for trait $0$, we can bound the process $N^K_0(t\log K)$ for $t\in [\sigma^K_k,\theta^K_k\wedge T]$ with branching
processes with constant parameters. If $\beta_0(s_{k-1})>0$, we deduce from Lemma~\ref{lem:BP} that
\begin{equation*}
\beta_0(s_{k-1})+(\ell^*_k\delta-\tau)(t-s_{k-1})-C_*\varepsilon_k
\leq \frac{\log(N^K_0(t \log K))}{\log K}
\leq \beta_0(s_{k-1})+(\ell^*_k\delta-\tau)(t-s_{k_1}) +C_* \varepsilon_k.
\end{equation*}
If $\beta_0(s_{k-1})=0$, we deduce that $N^K_0(t\log K)=0$ for all $t\geq\sigma^K_k$.

\medskip

Induction step: Assume that we have proved that, with high probability for all $t\in[\sigma^K_k,\theta^K_k\wedge s_{k}\wedge T]$,
\begin{equation*}
\beta_{\ell-1}(t)-C_* \varepsilon_k
\leq \frac{\log(N^K_{\ell-1}(t \log K))}{\log K}
\leq \beta_{\ell-1}(t) +C_* \varepsilon_{k}.
\end{equation*}
{and that $N^K_{\ell-1}(t\log K)=0$ for all $t\in [\sigma^K_k,\theta^K_k\wedge s_{k}\wedge T]$ such that
  $\beta_{\ell-1}\equiv 0$ on a small neighborhood of $t$, $s\in [(t-\varepsilon_k)\vee s_{k-1},t+\varepsilon_k]$.}
Our goal is to prove that this holds true with $\ell$ replaced by $\ell+1$. We split the interval $[s_{k-1},s_{k}]$ into
subintervals where $\beta_{\ell-1}$ is affine. We proceed inductively on each of these subintervals by coupling with branching
processes with immigration as in Step 1. Let us detail the computation for the first subinterval, say $[s_{k-1},t_1]$. On this interval,
we introduce $a$ and $c$ such that
\[
\beta_{\ell-1}(t)=c+\alpha+a\ (t-s_{k-1}),\quad\forall t\in[s_{k-1},t_1],
\]
to be coherent with the notation of Appendix~\ref{sec:BPI}. {We can then construct as in Step 1 branching processes with immigration bounding from above and below
  $N^K_\ell(t\log K)$ for all $t\in[\sigma^K_k,t_1\wedge\theta^K_k\wedge T]$ with distributions
  $BPI(4-\ell\delta+\tau\mathbbm{1}_{\ell>\ell^*_k}\pm C_*\varepsilon_k,4-\ell^*_k\delta+\tau\mathbbm{1}_{\ell<\ell^*_k}\mp
  C_*\varepsilon_k,a,c\pm C_*\varepsilon_k,\beta_\ell(s_{k-1})\pm C_*\varepsilon_k)$ if $\beta_\ell(s_{k-1})>0$, or $BPI(4-\ell\delta+\tau\mathbbm{1}_{\ell>\ell^*_k}\pm C_*\varepsilon_k,4-\ell^*_k\delta+\tau\mathbbm{1}_{\ell<\ell^*_k}\mp
  C_*\varepsilon_k,a,c\pm C_*\varepsilon_k,0)$ otherwise.} \\

\noindent We shall consider several cases.

\paragraph{Case (a)} Assume that $\beta_\ell(s_{k-1})>0$. Then, applying Theorem~\ref{thm:BPI-general}~(i) to the bounding processes, we
deduce that, with high probability on the time interval $[\sigma^K_k,\theta^K_k\wedge t_1\wedge T]$,
\begin{multline*}
  \left(\beta_\ell(s_{k-1})+S(\ell\delta;\ell^*_k\delta)\,(t-s_{k-1})-C_*\varepsilon_k\right)\vee(\beta_{\ell-1}(t)-\alpha-C_*\varepsilon_k)\vee
  0 \\ \leq \frac{\log(N^K_{\ell}(t \log K))}{\log K} \leq
  \left(\beta_\ell(s_{k-1})+S(\ell\delta;\ell^*_k\delta)\,(t-s_{k-1})+C_*\varepsilon_k\right)\vee(\beta_{\ell-1}(t)-\alpha+C_*\varepsilon_k)\vee
  0.
\end{multline*}

\paragraph{Case (b)} Assume that $\beta_\ell(s_{k-1})=0$ and the time $t_{\ell-1,k}$ defined in~\eqref{eq:def-t-ell-k} satisfies
$t_{\ell-1,k}<t_1$. This last inequality implies that $c<0$ and then $t_{\ell-1,k}=s_{k-1}-c/a$. Applying
Theorem~\ref{thm:BPI-general}~(ii) to the bounding processes, we deduce that, with high probability on the time interval
$[\sigma^K_k,\theta^K_k\wedge t_1\wedge T]$,
\begin{multline*}
  [(S(\ell\delta;\ell^*_k\delta)-C_*\varepsilon_k)\vee a]\vee\left(t-s_{k-1}+\frac{c-C_*\varepsilon_k}{a}\right)\vee 0
  \\ \leq \frac{\log(N^K_{\ell}(t \log K))}{\log K} \leq
  [(S(\ell\delta;\ell^*_k\delta)+C_*\varepsilon_k)\vee a]\vee\left(t-s_{k-1}+\frac{c+C_*\varepsilon_k}{a}\right)\vee 0.
\end{multline*}
The bound can be rewritten as
\[  [S(\ell\delta;\ell^*_k\delta)\vee a](t-t_{\ell-1,k})_+-C_*\varepsilon_k
   \leq \frac{\log(N^K_{\ell}(t \log K))}{\log K} \leq
  [S(\ell\delta;\ell^*_k\delta)\vee a](t-t_{\ell-1,k})_++C_*\varepsilon_k.
\]

\paragraph{Case (c)} Assume that $\beta_\ell(s_{k-1})=0$ and $t_{\ell-1,k}>t_1$. Then, applying Theorem~\ref{thm:BPI-general}~(ii) or~(iii) to the
bounding processes, we deduce that, with high probability on the time interval $[\sigma^K_k,\theta^K_k\wedge t_1\wedge T]$,
$N^K_\ell(t\log K)=0$.
\medskip

Summing up all these cases, it appears that we can extend $\beta_\ell(t)$ on the interval $[s_{k-1},t_1]$ as in~\eqref{eq:evol-beta_j} to
obtain, in all cases,
\begin{equation}
\label{eq:damgan-youpi}
\beta_{\ell}(t)-C_* \varepsilon_k
\leq \frac{\log(N^K_{\ell}(t \log K))}{\log K}
\leq \beta_{\ell}(t) +C_* \varepsilon_{k}.
\end{equation}

We proceed similarly for the other subintervals and deduce that~\eqref{eq:damgan-youpi} holds true with high probability on the time
interval $[\sigma^K_k,\theta^K_k\wedge s_{k}\wedge T]$. This finishes the induction on $\ell\neq\ell^*_k$.
\medskip

As in Phase 1, we can control $N^K_{\ell^*_k}(t)$ by logistic branching processes to show that $\theta^K_k\geq (s_{k}-\eta)\wedge T$
with probability converging to 1 for all $\eta>0$.

It also follows from~\eqref{eq:BPI-extinction} that $N^K_\ell(t\ \log K)=0$ with high probability on close subintervals of
$\text{int}\{\beta_\ell=0\}{\cap[s_{k-1},s_{k})}$, where $\text{int}$ denotes the interior.

\subsubsection{Intermediate Phase $k$, case 1}\label{sec:intermediate-k}

As in Intermediate Phase 1, we first extend the inequalities \eqref{eq:damgan-youpi} to $[\sigma^K_k,\theta^K_k\wedge T]$.
These inequalities involve the $\beta_\ell(t)$ that are constructed in Phase $k$ only until $s_{k}$.
Explicit expressions of $\beta_\ell(t)$, that we do not develop here, can be obtained using the recursion in \eqref{eq:evol-beta_j}. Using these formulas instead of $\beta_\ell(t)$ in \eqref{eq:damgan-youpi}, we obtain inequalities valid on $[\sigma^K_k,\theta^K_k\wedge T]$, with probability converging to 1.\\

At time $s_{k}$, {there exists at least one $\ell^*_{k+1}\neq\ell^*_{k}$ such that} $\beta_{\ell^*_{k+1}}(s_{k})=1$ and $S(\ell^*_{k+1}\delta, \ell^*_{k}\delta)>0$. Proceeding by contradiction as in Intermediate Phase 1, this allows us to prove that
\[\lim_{K\rightarrow +\infty}\theta^K_k=s_{k}.\]
{This ends the proof of Theorem~\ref{thm:transfer-main-new} if $T_0=s_{k}$ (i.e. in cases (ii)(a) or (ii)(c) in Theorem~\ref{thm:transfer-main-new}). If $T_0>s_{k}$,}
we deduce as in Intermediate Phase 1 that, with high probability, for some $\kappa>0$,
\[\sum_{\ell \notin \{\ell^*_{k},\ell^*_{k+1}\}} N^K_\ell(\theta^K_k\log K)\leq K^{1-\kappa},\]
and $N^K_{\ell^*_{k+1}}(\theta^K_k\log K)\geq m\varepsilon_k K/2$.\\

Distinguishing as in Section \ref{sec:intermediate-1-case-1} whether the emerging trait $\ell^*_{k+1}\delta$ is above 3 or not, we
can define a time $\sigma^K_{k+1}{\log K}$ satisfying
the recursion properties stated at the beginning of Section~\ref{sec:phase-k+1-case-1}. In particular, $\sigma^K_{k+1}$ converges to $s_{k}$ if $\max\beta_\ell(s_{k+1}+s)=1$ and $\ell^*_{k+1}\delta<3$, or $\sigma^K_{k+1}$ converges to $s_k+s$ otherwise. {The fact that the
  populations such that $\beta_\ell(s_{k})=0$ are actually extinct at time $\sigma^K_{k}\log K$ follows
  from~\eqref{eq:BPI-extinction} and from the fact that, by definition of $T_0$, we also have $\beta_\ell(s_{k}-\varepsilon)=0$ for
  $\varepsilon>0$ small enough.}

This ends the Step $k$ in case 1.

\subsubsection{Step $k$, case 2}
\label{sec:phase-k+1-case-2}

We assume here that {$T_0>s_{k-1}$ and} $\max_{0\leq \ell \leq L}\beta_\ell(s_{k-1})<1$ or $\ell^*_{k}\delta>3$ and the maximum is attained only for $\ell^*_{k-1}$ and $\ell^*_{k}$.

Here, the induction assumption is as follows: assume that, for all $\varepsilon_k>0$ small enough, for all $s>0$ small enough, we
have constructed a stopping time ${\sigma}_k^K{\log K}$ such that
\begin{itemize}
\item ${\sigma}^K_k$ converges in probability to $s_{k-1}+s$;
\item for all $\ell\not=\ell^*_{k}$, either $N^K_\ell(\sigma^K_k{\log K})=0$ if $\beta_\ell(s_{k-1}+s)=0$, or otherwise
    \[
    \beta_\ell(s_{k-1}+s)-\varepsilon_k\leq\frac{\log(1+N^K_\ell(\sigma^K_{k}{\log K}))}{\log K}\leq \beta_\ell(s_{k-1}+s)+\varepsilon_k<\beta_{\ell^*_{k}}(s_{k-1}+s)<1.
    \]
\end{itemize}

Let us define ${\theta}^K_k$ as
\begin{multline}
\label{teta0chapeau}
{\theta}_k^K=\inf\left\{t\geq \sigma_k^K: \frac{\log(1+N^K_{\ell^*_k}(t \log K))}{\log K}\not\in \Big[\beta_{\ell^*_k}(t)-\varepsilon_k, \beta_{\ell^*_k}(t)+\varepsilon_k, \Big]\right.\\
\left.\text{ or } \sum_{\ell\neq \ell^*_k} N^K_\ell(t \log K)\geq m\varepsilon_{k} N^K_{\ell^*_k}(t\log K)\right\}.
\end{multline}
As in Step $k$, case 1, our goal is to obtain bounds on
$N^K_\ell$, by an induction on $\ell$. Either the trait $\ell^*_k\delta$ remains dominant during the whole phase ($\sup_{t\in [s_{k-1},s_{k}]}\beta_{\ell^*_k}(t)<1$, see case (a) below) or it becomes resident during the phase (case (b)).

\paragraph{Case (a)} In  this case, there is no density dependence since the whole population is of size order $o(K)$. We can proceed exactly as in Step $k$, case 1, with the use of the fitness $\widehat{S}$ (defined in \eqref{eq:fitness-chap}) instead of $S$.\\

With high probability for all $t\in[\sigma^K_k,\theta^K_k\wedge s_{k}\wedge T]$, for all $\ell\in \{0,\dots L\}$,
\begin{equation}\label{eq:temp-fin}
\beta_{\ell}(t)-C_* \varepsilon_k
\leq \frac{\log(N^K_{\ell}(t \log K))}{\log K}
\leq \beta_{\ell}(t) +C_* \varepsilon_{k}
\end{equation}
{and that $N^K_{\ell}(t\log K)=0$ for all $t\in [\sigma^K_k,{\theta}^K_k\wedge s_{k}\wedge T]$ such
  that $\beta_{\ell-1}(s)=0$ on $s\in[(t-\varepsilon_k)\vee s_{k-1},t+\varepsilon_k]$.}

\paragraph{Case (b)} Let us denote by $t_1$ the first time at which $\beta_{\ell^*_k}(t)=1$. We introduce:
\[\widehat{\theta}^K_k=\inf\big\{ t\geq \sigma_k^K \ : \ N_{\ell^*_k}^K(t\log K)\geq m\varepsilon_k K\big\}.\]

On the time interval $[\sigma^K_k,  \widehat{\theta}^K_k\wedge {\theta}^K_k\wedge t_1 \wedge T]$, we proceed as in the case (a) to deduce that \eqref{eq:temp-fin} holds with high probability on this time interval. As in case (a), the exponents $\beta_\ell$ are defined with $\widehat{S}$ for all $s\in [s_{k-1},t_1]$. Extending $\beta_\ell$ on $[s_{k-1},t_1]$ like this is consistent with \eqref{eq:evol-beta_j} where $\widetilde{S}_{s,k}=\widehat{S}$ (see \eqref{eq:def-tilde-S}).\\

Proceeding as in Section \ref{sec:intermediate-k}, we can prove that, with high probability, $\widehat{\theta}^K_k<{\theta}^K_k$ and
\[\lim_{K\rightarrow +\infty}\widehat{\theta}^K_k=t_1.\]
Using Lemma \ref{lem:sable-popu}(ii), there exists $T(\varepsilon_k)$ such that with high probability
\[\frac{N^K_{\ell^*_k}(\widehat{\theta}^K_k \log K+T(\varepsilon_k))}{K}\in \big[\frac{3-\ell^*_k\delta}{C}-3\varepsilon_k,\frac{3-\ell^*_k\delta}{C}+3\varepsilon_k\big].\]
By Lemma \ref{lem:exposant-bouge-pas-trop}, at the time $\widehat{\theta}^K_k \log K+T(\varepsilon_k)$, we have for all $\ell\not=\ell^*_k$,
\[\frac{\log(1+N^K_\ell(\widehat{\theta}^K_k \log K+T(\varepsilon_k)))}{\log K} \in \big[\beta_\ell(t_1)-\varepsilon_k, \beta_\ell(t_1)+\varepsilon_k \big].\]
Applying the Markov property at this time $\widehat{\theta}^K_k \log K+T(\varepsilon_k)$ and proceeding as in Step $k$, case 1, we obtain \eqref{eq:temp-fin} where $\beta_\ell$ now evolves with the fitness $S$ instead of $\widehat{S}$. This
is consistent with \eqref{eq:evol-beta_j} where $\widetilde{S}_{s,k}=S$ on $[t_1,s_{k}]$. This enlightens the introduction of $\widetilde{S}$ in \eqref{eq:def-tilde-S}. This case is the only one where $\widetilde{S}_{t,k}$ is not constant on the phase $[s_{k-1},s_{k}]$.

\subsubsection{Intermediate Phase $k$, case 2}
\label{sec:interm-phase-k+1}

In case (b), a new trait emerges in a resident population. This is treated in Section \ref{sec:intermediate-k}.\\

\noindent In case (a), either there is extinction of the population at time $s_{k}<+\infty$ (note that extinction can occur only in this case), or
there is a change of dominant population at time $s_{k}<+\infty$.

\paragraph{Extinction} In this case, {either $T_0=s_{k}$, and we can conclude the induction as in
  Section~\ref{sec:intermediate-k}, or $T_0>s_{k}$. This means that,} for all $\ell\not=\ell^*_k$, $\beta_{\ell}(t_1-\eta)=0$ for
some $\eta>0$, so after time $t_1-\eta$, $N^K_\ell(t)=0$ with high probability. The population {is then composed
  only of individuals with trait $\ell^*_k\delta$ and} can be dominated by a subcritical branching process. Lemma \ref{lem:BP} proves
the extinction of the population.

\paragraph{Emergence of a new dominant population}
The emergence occurs in a population of size $o(K)$. Recall that $\theta^K_k$ has been defined in \eqref{teta0chapeau}. As in Section \ref{sec:intermediate-k}, we show that $\lim_{K\rightarrow +\infty}{\theta}^K_k=s_{k}$ in probability and that at time ${\theta}^K_k$, with high probability, for some $\kappa>0$,
\[\sum_{\ell \notin \{\ell^*_{k},\ell^*_{k+1}\}} N^K_\ell({\theta}^K_k\log K)\leq K^{-\kappa}N^K_{\ell^*_{k}}({\theta}^K_k\log K),\]
and $N^K_{\ell^*_{k+1}}(\theta^K_k\log K)\geq m\varepsilon_k N^K_{\ell^*_{k}}({\theta}^K_k\log K)/2$.\\

 Depending on the sign of $S=\widehat{S}(\ell^*_{k}\delta,\ell^*_{k+1}\delta)-\widehat{S}(\ell^*_{k+1}\delta,\ell^*_{k}\delta)$, we can use Lemma \ref{lem:competitionTBDI}(i) or (ii). We can then define a time ${\sigma}^K_{k+1}$ satisfying the recursion properties stated in Section \ref{sec:phase-k}.\\

This ends Step $k$ and finishes the proof.

\section{Proof of Theorem~\ref{thm:criterion-evol-suicide}}
\label{sec:general}

Recall that we assume $\tau>\delta$. We first give a lemma describing the dynamics of the exponents before the first re-emergence of
trait 0 or the first time when a trait in $(3,4)$ becomes dominant. Recall that $\widetilde{k}$, $\bar{k}$ and $m_0$ have been
defined in~\eqref{eq:def-k} and~\eqref{eq:def-m0}.
\begin{lem}
  \label{lem:criterion-evol-suicide}
  Under the assumptions of Theorem~\ref{thm:criterion-evol-suicide} and with the definition \eqref{eq:def-tau-bar} of $\bar{\tau}$,
  \begin{description}
  \item[\textmd{(a)}] if $\,m_{0}>0$, we have
  \[
  s_k=\frac{k\alpha}{\tau-\delta}\quad \text{for all} \quad k\delta\leq\bar{k}\delta\wedge 3.
  \]
  For all $\,
    s\leq s_{\bar{k}\wedge \lceil 3/\delta\rceil}\wedge  \bar{\tau}$,
     let $k\delta<3$ be such that $s\in[s_k,s_{k+1}]$. If $k\leq \bar{k}-1$, we have
    \begin{equation}
      \label{eq:lem-a-general}
      \beta_\ell(s)=
      \begin{cases}
        1 & \text{if }\ell=k, \\
        \left[1-(\ell-k)\alpha+(\tau-\delta)(s-s_k)\right]\vee 0 & \text{if }k<\ell\leq L, \\
        1-\frac{\alpha(k-\ell-1)}{\tau-\delta}\left(\tau-\frac{k-\ell}{2}\delta\right)-(\tau-(k-\ell)\delta)(s-s_k) &
        \text{if }0\leq \ell<k,
      \end{cases}
    \end{equation}
    and if $k=\bar{k}$, we have
    \begin{equation}
      \label{eq:lem-a-general-end}
      \beta_\ell(s)=
      \begin{cases}
        1 & \text{if }\ell=\bar{k}, \\
        \left[1-(\ell-\bar{k})\alpha+(\tau-\delta)(s-s_{\bar{k}})\right]\vee 0 & \text{if }\bar{k}<\ell\leq L, \\
        \left[1-\frac{\alpha(\bar{k}-\ell-1)}{\tau-\delta}\left(\tau-\frac{\bar{k}-\ell}{2}\delta\right)-(\tau-(\bar{k}-\ell)\delta)(s-s_{\bar{k}})\right]
        & \\
        \quad\vee\left[1-\ell\alpha-\frac{\alpha(\bar{k}-1)}{\tau-\delta}\left(\tau-\frac{\bar{k}}{2}\delta\right)-(\tau-\bar{k}\delta)(s-s_{\bar{k}})\right]
        & \text{if }0\leq\ell<\bar{k}; \\
      \end{cases}
    \end{equation}
  \item[\textmd{(b)}] if
  $\ m_{0}<0$, we have
  \[
  s_k=\frac{k\alpha}{\tau-\delta}\quad\text{for all}\quad k\delta<3.
  \]
    For all $s\leq s_{\lceil 3/\delta\rceil}$, let $k\delta<3$ be such that $s\in[s_k,s_{k+1}]$, then
    \begin{equation}
      \label{eq:lem-b-general}
      \beta_\ell(s)=
      \begin{cases}
        1 & \text{if }\ell=k, \\
        \left[1-(\ell-k)\alpha+(\tau-\delta)(s-s_k)\right]\vee 0 & \text{if }k<\ell\leq L, \\
        \left[1-\frac{\alpha(k-\ell-1)}{\tau-\delta}\left(\tau-\frac{k-\ell}{2}\delta\right)-(\tau-(k-\ell)\delta)(s-s_k)\right]\vee
        0 & \text{if } 0\vee(k-\widetilde{k}+1)\leq \ell<k, \\
        0 & \text{if } 0\leq \ell\leq k-\widetilde{k}.
      \end{cases}
    \end{equation}
  \end{description}
\end{lem}

\begin{proof}
  We apply Theorem~\ref{thm:transfer-main-new} and proceed by induction on $k\geq 0$. We already checked
  that~\eqref{eq:lem-a-general} and~\eqref{eq:lem-b-general} hold true for $k=0$ in the beginning of
  Section~\ref{sec:phase-1-case-2}.

\paragraph{Proof of~(a)}
Assume that $\,m_{0}>0$ and that we proved~\eqref{eq:lem-a-general} until time $s_k$ for some $k\in\{0,1,\ldots,(\lceil
3/\delta\rceil-1)\wedge \bar{k}\}$, and let us prove that~\eqref{eq:lem-a-general} is valid until time $s_{k+1}\wedge \bar{\tau}$ if
$k<\bar{k}$, or that~\eqref{eq:lem-a-general-end} is valid until time $\bar{\tau}<s_{\bar{k}+1}$ if $k=\bar{k}$. At time $s_k$, the
new resident trait $k\delta<3$ replaces the former resident trait $(k-1)\delta$ and hence, the values of the fitnesses are given by
\begin{equation}
  \label{eq:fitnesses-lem}
  S(\ell\delta;k\delta)=\tau-(\ell-k)\delta \text{ if }\ell>k,\ S(k\delta;k\delta)=0,\
  S(\ell\delta;k\delta)=(k-\ell)\delta-\tau\text{ if }\ell< k.
\end{equation}
Since $S(\ell\delta;k\delta)<\tau-\delta=S((k+1)\delta;k\delta)$ for all $\ell>k+1$, applying Theorem~\ref{thm:transfer-main-new}, we
obtain for all $s\geq s_k$,
\[
\beta_\ell(s)=[1-(\ell-k)\alpha+(\tau-\delta)(s-s_k)]\vee 0,\quad\forall \ell>k,
\]
until the next change of resident population. Since obviously $\beta_k(s)=1$ until this time, we have proved the first two lines
in~\eqref{eq:lem-a-general} (or of~\eqref{eq:lem-a-general-end} if $k=\bar{k}$). For the last lines, we obtain for all $s\geq s_k$
and all $\ell<k$
\begin{multline}
  \beta_\ell(s)=\left[\beta_\ell(s_k)-[\tau-(k-\ell)\delta](s-s_k)\right]\vee\left[\beta_{\ell-1}(s_k)-\alpha-[\tau-(k-\ell+1)\delta](s-s_k)\right]\vee\ldots
  \\ \vee \left[\beta_0(s_k)-\ell\alpha-[\tau-k\delta](s-s_k)\right]\vee 0, \label{eq:pf-lem-b}
\end{multline}
until the next change of resident trait. Now, all the terms in the right-hand side except the last one lie between two consecutive
terms of the sequence $\left(1-\frac{\alpha (n-1)}{\tau-\delta}\left(\tau-\frac{n}{2}\delta\right)\right)_{n\geq 0}$, so they are all
positive {since we assumed $m_0>0$}. Hence
\begin{multline}
  \beta_\ell(s)=\left[\beta_\ell(s_k)-[\tau-(k-\ell)\delta](s-s_k)\right]\vee\left[\beta_{\ell-1}(s_k)-\alpha-[\tau-(k-\ell+1)\delta](s-s_k)\right]\vee\ldots
  \\ \vee \left[\beta_0(s_k)-\ell\alpha-[\tau-k\delta](s-s_k)\right]. \label{eq:pf-lem-a}
\end{multline}
In addition, for all $1\leq\ell<k$, the function
\begin{equation}
  \label{eq:function-lem-a}
  s\in[s_k,s_{k+1}\wedge \bar{\tau}]\mapsto\beta_\ell(s_k)-[\tau-(k-\ell)\delta](s-s_k)-\beta_{\ell-1}(s_k)+[\tau-(k-\ell+1)\delta](s-s_k)+\alpha
\end{equation}
is affine and hence takes values between its values at times $s_k$ and $s_{k+1}$, i.e. between
\[
\frac{\alpha}{\tau-\delta}\left(2\tau-(k-\ell+1)\delta\right)\quad\text{and}\quad
\frac{\alpha}{\tau-\delta}\left(2\tau-(k-\ell+2)\delta\right).
\]

In the case where $k<\bar{k}$, we have $1\leq k-\ell\leq k-1 \leq\bar{k}-2=\lfloor 2\frac{\tau}{\delta}\rfloor-2$, so both terms
above are nonnegative, and the function~\eqref{eq:function-lem-a} is positive for all $s-s_k\in[0,\frac{\alpha}{\tau-\delta}]$. This
means that the maximum of two consecutive terms in the right-hand side of~\eqref{eq:pf-lem-a} is always the first term, therefore
\[
\beta_\ell(s)=\beta_\ell(s_k)-[\tau-(k-\ell)\delta](s-s_k),\quad \forall 0\leq\ell<k,
\]
until the next change of resident population. Since $k-\ell+1\leq\bar{k}$ for all $0\leq \ell <k$, we have
\[
\beta_\ell(s_k)-[\tau-(k-\ell)\delta]\frac{\alpha}{\tau-\delta}=1-\frac{\alpha
  (k-\ell)}{\tau-\delta}\left(\tau-\frac{k-\ell+1}{2}\delta\right)<1,
\]
so the first exponent $\beta_\ell$ for $\ell\neq k$ reaching 1 after time $s_k$ is $\ell=k+1$ and the next change of resident
population occurs at time $s_{k+1}$. Therefore, we have proved~\eqref{eq:lem-a-general}.

Let us now consider the case where $k=\bar{k}$. The previous argument shows that the maximum of two consecutive terms in the
right-hand side of~\eqref{eq:pf-lem-a} is always the first term, except possibly for the last two terms. Therefore, in all cases,
\[
\beta_\ell(s)=\left[\beta_\ell(s_{\bar{k}})-[\tau-(\bar{k}-\ell)\delta](s-s_{\bar{k}})\right]\vee
\left[\beta_0(s_{\bar{k}})-\ell\alpha-[\tau-\bar{k}\delta](s-s_{\bar{k}})\right] ,\quad \forall 0\leq\ell<\bar{k},
\]
until the next change of resident population. In this case, the first exponent to reach 1 is $\beta_0$, at time $\bar{\tau}$, since
$\bar{\tau}-s_{\bar{k}}<\frac{\alpha}{\tau-\delta}$. This concludes the proof of~\eqref{eq:lem-a-general-end}.

\paragraph{Proof of~(b)}
Assume now that $m_{0}<0$. This means that the sequence $\left(1-\frac{\alpha
    (n-1)}{\tau-\delta}\left(\tau-\frac{n}{2}\delta\right)\right)_{n\geq 0}$ becomes negative at some index $k^{**}\leq \widetilde{k}$ and
decreases until index $n=\widetilde{k}$.

Assume that we proved~\eqref{eq:lem-b-general} until time $s_k$ for some $k\in\{0,1,\ldots,\lceil 3/\delta\rceil-1\}$, and let us
prove that~\eqref{eq:lem-b-general} is valid until time $s_{k+1}$. The fitnesses are the same as in~\eqref{eq:fitnesses-lem}, and the
first computations of case~(a) apply similarly to prove the first two lines of~\eqref{eq:lem-b-general} until the next change of
resident population.

In order to prove the last two lines of~\eqref{eq:lem-b-general}, we need to modify~\eqref{eq:pf-lem-b} accordingly. To this aim, let
us observe that, when the exponent $\beta_\ell(s_k)$ of trait $\ell\delta$ is 0, it cannot increase (even if its fitness is positive)
unless some mutant individuals get born from trait $(\ell-1)\delta$, which only occurs at times $s$ such that $\beta_{\ell-1}(s)\geq
\alpha$. Therefore,~\eqref{eq:pf-lem-b} needs to be modified as $\beta_\ell(s)=0$ if $0\leq\ell\leq k-k^{**}$, for all $s\geq s_k$
until the next change of resident trait, or \eqref{eq:pf-lem-b} remains true for $0\vee(k-k^{**}+1)\leq \ell\leq k-2$. As in case
(a), we can prove that the maximum between two successive terms in the previous expression (except the last one, 0) is always reached
by the first one, hence
\[
\beta_\ell(s)=\left[\beta_\ell(s_k)-[\tau-(k-\ell)\delta](s-s_k)\right]\vee 0. % \label{eq:pf-lem-b-2}
\]
In view of the expression for $\beta_\ell(s)$, the next change of resident population occurs when $\beta_{k+1}(s)$ hits 1, at time
$s_{k+1}=s_k+\frac{\alpha}{\tau-\delta}$. This ends the proof of~\eqref{eq:lem-b-general}.
\end{proof}

\begin{proof}[Proof of Theorem~\ref{thm:criterion-evol-suicide}]
  Let us first prove (a). We assume  $\,m_{0}>0$  and $\bar{k}\delta<3$. Then, it follows from
  Lemma~\ref{lem:criterion-evol-suicide}~(a) that there is re-emergence of trait 0 at time
 $\bar{\tau}$.

   \medskip

   Let us now consider the case (b). Assume $\,m_{0}<0$. We introduce the first index such that the trait becomes larger than 3 by
   $\widehat{k}=\lceil \frac{3}{\delta}\rceil$.

   Note that the integer $k^{**}$ defined in the last proof satisfies $k^{**}\leq \widetilde{k}\leq\widehat{k}$. Hence,
   Lemma~\ref{lem:criterion-evol-suicide}~(b) implies that traits 0, $\delta$, $\cdots$, $(\widehat{k}-k^{**}) \delta$, get lost
   before their re-emergence and before a trait larger than 3 becomes dominant. Note that they remain extinct forever since mutations
   only produce individuals with larger traits. Our goal is to prove that, after time $s_{\widehat{k}}$, the other traits get
   progressively lost until the global extinction of the population.

   At time $s_{\widehat{k}}$, the dominant trait becomes $\widehat{k}\delta>3$ and the fitnesses are given by
   \begin{equation}
    \label{eq:fitnesses-cas-b}
    \widehat{S}(\ell\delta;\widehat{k}\delta)=3-\ell\delta-\tau\text{ if }\ell<\widehat{k};\quad
    \widehat{S}(\widehat{k}\delta;\widehat{k}\delta)=3-\widehat{k}\delta;\quad \widehat{S}(\ell\delta;\widehat{k}\delta)=3-\ell\delta+\tau\text{ if }\ell>\widehat{k}.
  \end{equation} Now, for
  $\widehat{k}-k^{**}<\ell<\widehat{k}$, $\widehat{S}(\ell\delta;\widehat{k}\delta)<S(\ell\delta;\widehat{k}\delta)<0$ (recall that $k^{**}\leq \widetilde{k}$). Therefore, all positive exponents $\beta_\ell(s)$ for
  $\ell<\widehat{k}$ keep on decreasing until they hit 0 (which means that trait $\ell\delta$ is lost) or until the next change of
  dominant population. For all $\ell>\widehat{k}$, $\beta_\ell(s_{\widehat{k}})=[1-(\ell-\widehat{k})\delta]_+$ and $\beta_\ell$
  has slope $\widehat{S}((\widehat{k}+1)\delta;\widehat{k}\delta)=\widehat{S}(\widehat{k}\delta;\widehat{k}\delta) +\tau-\delta$, therefore the
  next change of dominant population occurs when $\beta_{\widehat{k}+1}(s)=\beta_{\widehat{k}}(s)$ at time $s_{\widehat{k}+1}$. At this time, by
  definition of $k^{**}$,
  \begin{align*}
    \beta_{\widehat{k}-k^{**}+1}(s_{\widehat{k}+1})
    % & =\left[\beta_{\widehat{k}-k^{**}+1}(s_{\widehat{k}})+(3-(\widehat{k}-k^{**}+1)\delta-\tau)\frac{\alpha}{\tau-\delta}\right]\vee 0 \\
    & \leq \left[\beta_{\widehat{k}-k^{**}+1}(s_{\widehat{k}})+((k^{**}-1)\delta-\tau)\frac{\alpha}{\tau-\delta}\right]\vee 0=0,
  \end{align*}
  so trait $(\widehat{k}-k^{**}+1)\delta$ is lost. For all $\widehat{k}-\ell^*+1<\ell<\widehat{k}+1$,
  \begin{align*}
    \beta_\ell(s_{\widehat{k}+1}) & \leq \left[\beta_\ell(s_{\widehat{k}})+(3-\ell\delta-\tau)\frac{\alpha}{\tau-\delta}\right]\vee \ldots \\
    & \quad\vee
    \left[\beta_{\widehat{k}-k^{**}+2}(s_{\widehat{k}})-(\ell+k^{**}-\widehat{k}-2)\alpha+(3-(\widehat{k}-k^{**}+2)\delta-\tau)\frac{\alpha}{\tau-\delta}\right]\vee
    0\\
    %\phantom{\beta_\ell(s_{\widehat{k}+1})}
    & \leq \left[\beta_\ell(s_{\widehat{k}})-(\tau-(\widehat{k}-\ell)\delta)\frac{\alpha}{\tau-\delta}\right]\vee \ldots \\
    & \quad\vee
    \left[\beta_{\widehat{k}-k^{**}+2}(s_{\widehat{k}})-(\ell+k^{**}-\widehat{k}-2)\alpha-(\tau-(k^{**}-2)\delta)\frac{\alpha}{\tau-\delta}\right]\vee
    0 \\ & =\left[\beta_\ell(s_{\widehat{k}})-(\tau-(\widehat{k}-\ell)\delta)\frac{\alpha}{\tau-\delta}\right]\vee 0,
  \end{align*}
  where the last equality was proved in the proof of Lemma~\ref{lem:criterion-evol-suicide}~(b). Hence, we can proceed inductively to prove
  that, for all $1\leq k\leq L-\widehat{k}$, at time $s_{\widehat{k}+k}=(\widehat{k}+k)\frac{\alpha}{\tau-\delta}$, $\max_{0\leq\ell\leq
    L}\beta_\ell(s_{\widehat{k}+k})=\beta_{\widehat{k}+k}(s_{\widehat{k}+k})>0$, $\beta_\ell(s_{\widehat{k}+k})=0$ for all $\ell\leq
 \widehat{k}+k-k^{**}$, all exponents $\beta_\ell(s)$ with $\widehat{k}+k-k^{**}<\ell\leq \widehat{k}+k$ decrease. In addition,
  \begin{itemize}
  \item either $\beta_{\widehat{k}+k}(s_{\widehat{k}+k})>\alpha$ and $\widehat{k}+k<L$, and then
    $\beta_{\widehat{k}+k+\ell}(s_{\widehat{k}+k})=[\beta_{\widehat{k}+k}(s_{\widehat{k}+k})-\ell\alpha]\vee 0$ for all $1\leq\ell\leq L-\widehat{k}-k$,
    the next time of change of dominant population is $s_{\widehat{k}+k+1}=(\widehat{k}+k+1)\frac{\alpha}{\tau-\delta}$;
  \item or $\beta_{\widehat{k}+k}(s_{\widehat{k}+k})<\alpha$ or $\widehat{k}+k=L$, and then $\beta_{\widehat{k}+k+\ell}(s_{\widehat{k}+k})=0$ for all
    $\ell\geq 1$, so every exponent keeps on decreasing until $\beta_{\widehat{k}+k}(s)=0$ and there is extinction of the population.
  \end{itemize}
  The second case will necessarily occur after a finite number of steps, so extinction of the population occurs before the re-emergence
  of any trait. This ends the proof of~(b).
  \medskip

  To prove (c), we assume $m_{0}>0$ and $\bar{k}\geq \widehat{k}$. By Lemma~\ref{lem:criterion-evol-suicide}, in both cases,
  $\beta_0(s_{\widehat{k}})\in(0,1)$ and trait $\widehat{k}\delta$ becomes dominant at time $s_{\widehat{k}}$. Since
  $\widehat{k}\delta>3$, the dynamics of the exponents $\beta_\ell$ after time $s_{\widehat{k}}$ is governed by the fitnesses given
  in \eqref{eq:fitnesses-cas-b}.

  Since $\widehat{S}(\widehat{k}\delta;\widehat{k}\delta)<0$, $\max_{0\leq\ell\leq L}\beta_\ell(s)<1$ for all $s>s_{\widehat{k}}$ until the next change of dominant trait. Observe that the fitness of trait 0, $\widehat{S}(0;\widehat{k}\delta)=3-\tau$ is positive as long as
  $\widehat{k}\delta$ is the dominant trait, and for any other dominant trait $\ell\delta$, we have
  \[
  \widehat{S}(0;\ell\delta)=
  \begin{cases}
    3-\tau & \text{if }\ell\geq 1, \\
    3 & \text{if }\ell=0.
  \end{cases}
  \]
  So, as long as $\max_{0\leq\ell\leq L}\beta_\ell(s)<1$, $\beta_0(s)$ is increasing. This implies that there is necessarily
  re-emergence of some trait before the extinction of the population. Since in addition
  $\widehat{S}(\ell\delta;\ell\delta)=3-\ell\delta<0$ for all $\ell\geq\widehat{k}$, the first re-emerging trait cannot be
  any of the traits $\ell\delta$ for $\ell\geq \widehat{k}$. This ends the proof of~(c).
\end{proof}

\appendix

\section{Branching process in continuous time}
\label{sec:linear-BDP}

The goal of this section and the next two ones is to give general results on specific branching or birth-and-death processes in one and two dimensions.

\medskip \noindent Here, we consider a single population $(Z^K_t,t\geq 0)$, following a linear
birth and death process, i.e.\ a branching process, with individual birth rate $b\geq 0$, individual death rate $d\geq 0$ and initial
value $Z^K_0=\lfloor K^\beta-1\rfloor\in\N$. The definition of $Z^K_0$ means that the population is extinct initially when $\beta=0$, and that
$Z^{K}_0\sim K^\beta$ when $K\rightarrow+\infty$ if $\beta>0$. We denote the law of $Z^K$ by $BP_K(b,d,\beta)$.
In the sequel, we set $r=b-d$.
\begin{lem}
  \label{lem:BP}
  Let $Z^K$ follow the law $BP_K(b,d,\beta)$, where $b,d\geq 0$ and $\beta>0$. Then the process $\bigg(\frac{\log( 1+ Z^K_{s\log
      K})}{\log K}, s\geq 0\bigg)$ converges in probability in $L^\infty([0,T])$ for all $T>0$ to $((\beta+ r s)\vee 0, s\geq 0)$
  when $K$ tends to infinity.

  In addition, if $b<d$, for all $t>\beta/r$,
  \begin{equation}
    \label{eq:true-extinction-LBP}
    \lim_{K\rightarrow +\infty}\P\left(Z^K_{t\log K}=0\right)=1.
  \end{equation}
\end{lem}

\begin{proof}
We divide the proof in several steps. We borrow ideas from~\cite{durrettmayberry}.
\medskip

\noindent \textbf{Step 1. Construction of a martingale}.
The process writes
\ben
Z^K_t = K^\beta + M^K_{t} + \int_{0}^t r\, Z^K_{s} ds,\een
where
$M^K$ is a square integrable martingale with quadratic variation
$ \langle M^K\rangle_{t}= \int_{0}^t (b+d)Z^K_{s} ds$.
Taking the expectation, we obtain $\ \E(Z^K_t)=  K^\beta + \int_0^t r\, \E(Z^K_s)ds$, which gives that
\begin{equation}
\E(Z^K_t)=  K^\beta e^{rt}.\label{NMsimple:esp}
\end{equation}

\noindent Using Itô's formula,
$\
1+e^{-r t}Z^K_t=1+K^\beta +\widehat{M}^K_t$,
where $\widehat{M}^K_t=\int_0^t e^{-rs}dM^K_s$ is a square integrable martingale with quadratic variation process
$$\langle \widehat{M}^K\rangle_t=\int_0^t e^{-2rs}(b+d)Z^K_s ds.$$

Using Doob inequality, for any $0<\eta<\beta$,
\begin{align}
\P\big(\sup_{t\leq T \log K} \big|e^{-r t}Z^K_t-K^\beta\big|\geq K^{\eta}\big) = & \ \P\big(\sup_{t\leq T\log K}
|\widehat{M}^K_t|\geq K^{\eta}\big) \notag \\
  \leq & \ 4 K^{-2\eta} \E\big(\langle M^K\rangle_{T\log K}\big) \notag \\
= &\  4 K^{-2\eta} \int_{0}^{T\log K} (b+d) e^{-2rs} K^\beta e^{rs} ds \notag \\
= &\  4 K^{\beta-2\eta} (b+d)
\begin{cases}
\frac{1}{r} \Big(1-K^{-rT}\Big) & \text{if }b\neq d, \\
T\log K & \text{if }b=d.
\end{cases}
\label{etape1}
 \end{align}
\medskip

\noindent \textbf{Step 2. Case $r> 0$}. Fix $T>0$ and $\eta=2\beta/3$. On the set
$$\Omega^K_1=\big\{\sup_{t\leq T \log K} \big|e^{-r t}Z^K_t-K^\beta\big|\leq K^{\frac{2\beta}{3}}\big\},$$
whose probability tends to 1 by~\eqref{etape1}, we have:
\begin{multline*}
\sup_{t\leq T}\left|\frac{\log(1+Z^K_{t\log K})}{\log K}- \big(\beta +rt\big)\right|\\
= \ \sup_{t\leq T}\left|\frac{\log(1+Z^K_{t\log K})}{\log K} - \frac{\log(1+K^{\beta+rt})}{\log K} + \frac{\log(1+K^{\beta+r
      t})}{\log K}-\frac{\log(K^{\beta+rt})}{\log K}\right|
\end{multline*}

\begin{align*}
% = &\ \sup_{t\leq T}\left|\frac{\log(1+Z^K_{t\log K})}{\log K} - \frac{\log(1+K^{\beta+rt})}{\log K} + \frac{\log(1+K^{\beta+r t})}{\log K}-\frac{\log(K^{\beta+rt})}{\log K}\right|\\
\leq &\  \sup_{t\leq T} \left[\frac{1}{\log K}\left|\log\Big(\frac{1+Z^K_{t\log K}}{1+K^{\beta+rt}}\Big)\right| +\frac{\log\big(1+K^{-\beta-rt}\big)}{\log K}\right]\\
\leq &\   \sup_{t\leq T} \left[\frac{1}{\log K}\left|\log\left(\frac{1+Z^K_{t\log K}\vee K^{\beta+rt}}{1+Z^K_{t\log K}\wedge K^{\beta+rt}}\right)\right| +\frac{K^{-\beta-rt}}{\log K}\right] \\
= &\   \sup_{t\leq T} \left[\frac{1}{\log K}\log\left(1+\frac{\left|Z^K_{t\log K}- K^{\beta+rt}\right|}{1+Z^K_{t\log K}\wedge
      K^{\beta+rt}}\right) +\frac{K^{-\beta-rt}}{\log K}\right], \\
% \end{aligned}
% \end{multline*}
% so that
% \begin{align*}
% \sup_{t\leq T}\left|\frac{\log(1+Z^K_{t\log K})}{\log K}- \big(\beta +rt\big)\right|
\leq &\  \sup_{t\leq T}\left[\frac{1}{\log K} \frac{|Z^K_{t\log K}-K^{\beta+rt}|}{1+K^{\beta+rt}-K^{\frac{2\beta}{3}+rt}}+\frac{K^{-\beta-rt}}{\log K}\right]\\
\leq &\  \frac{K^{\frac{2\beta}{3}}}{\log K}\,\sup_{t\leq T}\frac{K^{rt}}{K^{\beta+rt}-K^{\frac{2\beta}{3}+rt}} + \frac{K^{-\beta}}{\log K}\\
\leq &\  \frac{2K^{-\frac{\beta}{3}}+K^{-\beta}}{\log K} \ \hbox{ (for $K$ large enough) },
% \end{align*}
\end{align*}
% \end{multline*}
which converges to zero when $K\rightarrow +\infty$.

\medskip \noindent A simple adaptation of the previous argument gives the same conclusion for $b=d$.
\medskip

\noindent \textbf{Step 3. Case $r<0$}. Since the function $t\mapsto\beta+rt$ vanishes at time $\beta/|r|$, we need to
consider three phases. We fix $\varepsilon\in(0,\beta)$ and set $T_\varepsilon=\frac{\beta-\varepsilon}{|r|}$ and
$\eta=\beta-\varepsilon/3$. First we prove that, before time $T_\varepsilon\log K$ the population size remains large enough to use the same
argument as in Step 2. Second, the population gets extinct with high probability between time $T_\varepsilon\log K$ and
$(T_\varepsilon+2\varepsilon/|r|)\log K$ and third, in order to obtain the convergence for the $L^\infty$ norm, we prove that the
supremum of the process on the this time interval remains of the order of $\varepsilon$. Since $\varepsilon$ can be chosen
arbitrarily small, the result follows.
\medskip

\noindent \textbf{Step 3(i).} On the set
$$\Omega^K_2=\big\{\sup_{t\leq T_\varepsilon\log K} \big|e^{-rt}Z^K_t-K^\beta\big|\leq K^{\beta-\frac{\varepsilon}{3}}\big\},$$
whose probability tend to 1, we have:
\begin{align*}
\sup_{t\leq T_\varepsilon}\left|\frac{\log(1+Z^K_{t\log K})}{\log K}- \big(\beta +rt\big)\right|
\leq & \ \sup_{t\leq T_\varepsilon}\left[\frac{1}{\log K} \frac{|Z^K_{t\log K}-K^{\beta+rt}|}{1+\big(K^{\beta+rt}-K^{\frac{2\beta}{3}+rt}\big)_+}+\frac{K^{-\beta-rt}}{\log K}\right]\\
\leq &\  \frac{K^{\beta-\frac{\varepsilon}{3}}}{\log K}\,\sup_{t\leq T_\varepsilon}\frac{K^{rt}}{1+\big(K^{\beta+rt}-K^{\frac{2\beta}{3}+rt}\big)_+} + \frac{K^{-\varepsilon}}{\log K}\\
\leq &\  \frac{2 K^{-\frac{\varepsilon}{3}}+K^{-\varepsilon}}{\log K} \ \hbox{ (for $K$ large enough) },
\end{align*}
which also converges to zero when $K\rightarrow +\infty$.

\medskip\noindent\textbf{Step 3(ii).} It
follows from the last step that, with probability converging to 1, $Z^K_{T_\varepsilon\log K}\leq 2K^\varepsilon$. Remind
from~\cite[Section 5.4.5, p. 180]{Meleardbook} that, if $T_{\text{ext}}$ denotes the extinction time of a $BP(b,d,1)$,
  \begin{align}
    \label{eq:Sylvie}
    \PP(T_{\text{ext}}>t)=\frac{r e^{rt}}{b e^{rt}-d}.
  \end{align}
  Hence, for a $BP(b,d,2K^\varepsilon)$ branching process,
  \begin{align*}
    \PP(T_{\text{ext}}>t)=1-\left(1-\frac{r e^{rt}}{b e^{rt}-d}\right)^{2K^\varepsilon}.
  \end{align*}
  Thus, for $t=\frac{2\varepsilon}{|r|}\log K$,
  \begin{align*}
    \PP\left(T_{\text{ext}}>\frac{2\varepsilon}{|r|}\log K\right)\sim 2\,\frac{|r|}{d}\,K^{-\varepsilon}
  \end{align*}
  as $K\rightarrow+\infty$. Since this goes to 0 when $K\rightarrow+\infty$, we have completed the proof
  of~\eqref{eq:true-extinction-LBP}.

\medskip\noindent\textbf{Step 3(iii).}
Since the last two steps were true for any value of $\varepsilon>0$, in order to complete the proof, it is enough to check that
\[
\sup_{s\in[T_\varepsilon,T_\varepsilon+2\varepsilon/|r|]} \frac{\log(1+Z^K_{s\log K})}{\log K}\leq \frac{2d}{|r|}\,\varepsilon.
\]
For this, we observe that
the maximal size on the
time interval $\log K\times[T_\varepsilon,T_\varepsilon+2\varepsilon/|r|]$ of the families stemming from each individual alive at time $T_\varepsilon\log K$  is bounded by the size at time
$\frac{2\varepsilon\log K}{|r|}$ of a Yule process with birth rate $b$, i.e.\ a geometric random variable $G_i$ with expectation
$K^{2b\varepsilon/|r|}$, independently for each immigrant families. Hence, with probability converging to 1,
  \begin{align*}
    \sup_{t\in [T_\varepsilon,T_\varepsilon+2\varepsilon/|r|]} Z^K_{t\log K}\leq\sum_{i=1}^{2K^\varepsilon} G_i\leq K^{\frac{2d}{|r|}\varepsilon}.
  \end{align*}
The proof is completed.
\end{proof}

\section{Branching process with immigration}
\label{sec:BPI}

Our goal in this section is to extend the arguments of the previous section to include immigration. %The proof of our main result
%(Theorem~\ref{thm:BPI-general} below) will combine similar ideas as for Lemma~\ref{lem:BP}, detailed below in
%Theorem~\ref{lem:large-popu} and a series of lemmas.

We consider a linear birth and death process with immigration $(Z^K_{t}, t\geq 0)$ with law $BPI_K(b,d,a,c,\beta)$, where
$\
Z^K_{0} = \lfloor K^\beta-1\rfloor\
$
with $\beta\geq 0$, $b\geq 0$ is the individual birth rate, $d\geq 0$ the individual death rate and $K^{c}e^{as}$ is
the immigration rate at time $s\geq 0$ with $a,c\in\RR$.

With these notations, the generator of $Z^K$ at time $t\geq 0$ is given for bounded measurable functions $f$ by
\[
L_t f(n) = (bn +
K^{c}e^{at})\,(f(n+1) - f(n)) + dn\,(f(n-1) - f(n)).
\]
Recall that $r=b-d$. {We start with a first result in the case where the population does not go extinct (Theorem \ref{lem:large-popu}) and a series of lemmas. Our
  general result on branching processes with immigration is Theorem~\ref{thm:BPI-general}, whose proof combines similar ideas as for Lemma~\ref{lem:BP}.}

\begin{thm}[Large population]\label{lem:large-popu}
  Assume that $c\leq \beta$, $\beta>0$. Then, for all $T>0$ such that
  \begin{equation}\label{cond:BPI-large-popu}
    \inf_{t\in[0,T]}(\beta+rt)\vee(c+at)>0,
  \end{equation}
  the process $\bigg(\frac{\log (1+ Z^K_{s\log K})}{\log K}, 0\leq s\leq T\bigg)$ converges when $K$ tends to infinity in
  probability in $L^\infty([0,T])$ to $\Big((\beta+rs)\vee(c+as),s\in[0,T]\Big)$.
\end{thm}

In the proofs of the main results of~\cite{durrettmayberry} and~\cite{boviercoquillesmadi}, similar asymptotic results on branching
processes with immigration were proved. Our framework is more general since we consider cases where immigration is time-dependent
($a\neq 0$, contrary to~\cite{boviercoquillesmadi}) and where growth can be driven by immigration ($a>r$, contrary
to~\cite{durrettmayberry}). The results of these references are based on Doob's inequality. In our general case, this is not
sufficient: we also need to make use of the maximal inequality for supermartingales~\cite[Ch. VI, p.\ 72]{dellacherie-meyer} in the
case $a>r$.

To motivate this result, let us first compute the expectation and variance of a $BPI_K(b,d,a,c,\beta)$ process. Note that this result
is valid also if $c>\beta$.

\begin{lem}
  \label{lem:esp-var}
  Assume that $(Z^K_t,t\geq 0)$ is a $BPI_K(b,d,a,c,\beta)$ process. Then, for all $t\geq 0$,
  \begin{equation}
    \label{eq:esperance}
    x^K_t=\E(Z^K_t)=
    \begin{cases}
      \left(K^\beta-1 +\frac{K^c}{r-a}\right)\,e^{rt} -\frac{K^{c}e^{at}}{r-a} & \text{ if }r\neq a, \\
      e^{rt}(K^\beta-1+K^c t) &\text{ if }r=a
    \end{cases}
  \end{equation}
  and
  \begin{equation}
    \label{eq:variance}
    % v^K_t=
    \textnormal{Var}(Z^K_{t}) =
    \begin{cases}
      (b+d)\left(K^\beta-1+\frac{K^c}{r-a}\right)\frac{e^{2rt}
        -e^{rt}}{r}+K^c\left(1-\frac{b+d}{r-a}\right)\frac{e^{at}-e^{2rt}}{a-2r} & \text{ if }r\neq a, \\
      (b+d)(K^\beta-1)\frac{e^{2rt}
        -e^{rt}}{r}+K^c\frac{e^{2rt}-e^{rt}}{r}+(b+d)K^c\frac{e^{2rt}-e^{rt}-rte^{rt}}{r^2}      &\text{ if }r=a\neq 0, \\
      (K^c+2b (K^\beta-1)) t+b K^c t^2 &\text{ if }r=a=0 \quad (b=d),
    \end{cases}
  \end{equation}
  where, in the first line, by convention $\frac{e^{2rt} -e^{rt}}{r}=t$ if $r=0$ and $\frac{e^{at}-e^{2rt}}{a-2r}=t$ if $a=2r$.
\end{lem}

\begin{proof}
  The semimartingale decomposition of $Z^K$ is given by
  \begin{equation}
    \label{eq:BPI-martpb}
    Z^K_{t}= Z^K_{0}+ \int_{0}^t \big(rZ^K_{s} +K^{c}e^{as}\big) ds + M^K_{t},
  \end{equation}
  where $M^K$ is a square integrable martingale with quadratic variation
  \begin{equation}
    \label{eq:BPI-crochet-martpb}
    \langle M^K\rangle_{t}=\int_{0}^t\big((b+d)Z^K_{s} +K^{c}e^{as} \big) ds.
  \end{equation}
  By a usual argument, the expectation $x(t) =\E( Z^K_{t})$ solves the linear equation
  \[
  \dot{x}(t) = r x(t) + K^{c}e^{at} \ ;\ x(0)=K^\beta-1
  \]
  whose solution is given by~\eqref{eq:esperance}.

  It\^o's formula applied to $(Z^K_{t})^2$ yields that $u(t) =\E( (Z^K_{t})^2)$ is solution to
  \ben
  \dot{u}(t)= 2r u(t) + (2 K^{c}e^{at} + b+d)\,x(t) + K^{c}e^{at} \ ;\ u(0) = (K^{\beta}-1)^2.\een
  Straightforward computation in each case gives~\eqref{eq:variance}.
\end{proof}

We deduce from the last lemma that
\[
x^K_{t\log K}=
\begin{cases}
  K^{rt}(K^\beta-1) +\frac{K^{c+rt}-K^{c+at}}{r-a} & \text{ if }r\neq a, \\
  K^{rt}(K^\beta-1)+K^{c+rt}t \log K  &\text{ if }r=a,
\end{cases}
\]
so that,
\[
\frac{\log(1+x^K_{t\log K})}{\log K}\sim
\begin{cases}
  (\beta+rt)\vee(c+rt)\vee(c+at) & \text{if }\beta>0 \\
  (c+rt)\vee(c+at) & \text{if }\beta=0
\end{cases}
\]
as $K\rightarrow+\infty$. Note that the expression~\eqref{eq:variance} in the case $r\neq a$ can be written as
\begin{align}
  \label{eq:variance-positive}
  \text{Var}(Z^K_{t\log K}) = (b+d) (K^\beta-1)\varphi_t(r)+K^c\varphi_t(a)+(b+d)K^c\frac{\varphi_t(r)-\varphi_t(a)}{r-a}
\end{align}
where
\begin{align*}
  \varphi_t(x)=\frac{K^{xt}-K^{2r t}}{x-2r}.
\end{align*}
Since $\varphi_t$ is nonnegative and nondecreasing for all $t\geq 0$, the three terms in the right-hand side
of~\eqref{eq:variance-positive} are positive, so there is no compensation between these terms and hence
\begin{align*}
  \text{Var}(Z^K_{t\log K})\sim
  \begin{cases}
    C K^{\beta+(2r\vee r)t}\vee K^{c+(2r\vee r\vee a)t} & \text{if }\beta>0, \\
    C K^{c+(2r\vee r\vee a)t} & \text{if }\beta=0.
  \end{cases}
\end{align*}
for a constant $C$ that may depend on $t$ but remains uniformly bounded and bounded away from 0 for bounded values of
$t\in\mathbb{R}_+$.

In several cases, we can compare the process with its expectation because the standard deviation $\text{Var}(Z^K_{t\log K})^{1/2}$ is negligible
compared with $x^K_{t\log K}$. One can easily check that this is always the case when $t>0$ and $\beta>0$ and $r\geq 0$ or when
 $\beta\geq 0$, $r\leq 0$ and $(\beta+rt)\vee(c+(a\vee r)t)>0$ or when  $\beta=0$, $r>0$ and $c>0$.

The proof of Theorem~\ref{lem:large-popu} is based on martingales inequalities.
We start by defining a martingale and computing its quadratic variation.

\begin{lem}
  \label{lem:crochet}
  The process
  \begin{equation}
    \widetilde{M}^K_t:= e^{-rt}\big(Z^K_{t}-x^K_t\big)
  \end{equation}
  is a martingale whose predictable quadratic variation process satisfies
  \begin{align*}
  \E\big(\langle \widetilde M^K\rangle_{T\log K}\big) & =K^{c}\frac{K^{(a-2r)T}-1}{a-2r}+ (b+d)(K^{\beta}-1) \frac{1-K^{-rT}}{r}\\
 & +(b+d)\times \begin{cases}\frac{K^c}{a-r}\left(\frac{K^{(a-2r)T}-1}{a-2r}-\frac{1-K^{-rT}}{r}\right) & \mbox{if }\ r\not=a\\
 K^c \frac{1-K^{-rT}}{r^2} - K^c \frac{T\log( K) K^{-rT}}{r}& \mbox{if }\ r=a.
\end{cases}
  \end{align*}
\end{lem}

\begin{proof}
  Recall the semi-martingale decomposition of $(Z^K_t)_{t\in \R_+}$ in \eqref{eq:BPI-martpb} where the bracket of the martingale part
  is given in \eqref{eq:BPI-crochet-martpb}. It follows that $\widetilde{M}^K_t$ is a martingale with predictable quadratic
  variation process
  \begin{equation}
    \langle \widetilde{M}^K\rangle_t=\int_0^t e^{-2rs}\big((b+d)Z^K_s+K^{c}e^{as} \big) ds.
  \end{equation}
  Using \eqref{eq:esperance}, we have that
  \begin{align*}
    \int_0^{T\log K} e^{-2rs}x^K_s \ ds= &
    \begin{cases}
      (K^\beta-1) \frac{1-K^{-rT}}{r}+\frac{K^c}{a-r}\left(\frac{K^{(a-2r)T}-1}{a-2r}-\frac{1-K^{-rT}}{r}\right) & \mbox{if }\ r\not=a\\
      (K^\beta-1) \frac{1-K^{-rT}}{r}+ K^c \frac{1-K^{-rT}}{r^2} - K^c \frac{T\log( K) K^{-rT}}{r}& \mbox{if }\ r=a,
    \end{cases}
  \end{align*}
  and Lemma~\ref{lem:crochet} follows.
\end{proof}

\noindent {\bf Proof of Theorem~\ref{lem:large-popu} \phantom{9}}

Let us denote $\bar{\beta}_t = (\beta+rt)\vee(c+at)$. The proof extends ideas used to prove Lemma \ref{lem:BP}. We give the proof in
the case where $r\neq a$, $r\neq 0$ and $a\neq 2r$. The extension to the
other cases can be easily deduced using comparison arguments.

We shall use two different martingale inequalities applied to $\widetilde{M}^K_t$. The first one is Doob's inequality: for any
$0<\eta<\beta$,
\begin{align}
\P\big(\sup_{t\leq T \log K} \big|e^{-r t}(Z^K_t-x^K_t)\big|\geq K^{\eta}\big)
= & \P\big(\sup_{t\leq T\log K}
|\widetilde{M}^K_t|\geq K^{\eta}\big) \notag \\
  \leq  & 4 K^{-2\eta} \E\big(\langle M^K\rangle_{T\log K}\big) \notag \\
  \leq & C K^{-2\eta}\times
  \begin{cases}
    K^{\beta\vee (\beta-rT)\vee(c+(a-2r)T)} & \mbox{if }\ r\not=a\\
    K^{\beta\vee (\beta-rT)}\log K & \mbox{if }\ r=a,
  \end{cases}
\label{etape2}
\end{align}
where we used the fact that $c\leq \beta$ in the last inequality. % This inequality will be used in cases where $a\leq r$.
When $a> r$, we also apply the maximal inequality of~\cite[Ch. VI, p.\ 72]{dellacherie-meyer} to the supermartingale
$e^{-(a-r)t}\widetilde{M}^K_t$:
\begin{align*}
  K^\eta\,\PP\big(\sup_{t\leq T \log K} e^{-at}|Z^K_t-x^K_t|\geq K^{\eta}\big) & =K^\eta\,\PP\big(\sup_{t\leq T \log K}
  e^{-(a-r)t}|\widetilde{M}^K_t|\geq K^{\eta}\big) \\ & \leq 3 \sup_{t\leq T\log K} e^{-(a-r)t}\left(\E\langle\widetilde{M}^K_{t}\rangle\right)^{1/2}
\end{align*}
Using Lemma~\ref{lem:crochet}, we deduce
\begin{equation}
  \PP\big(\sup_{t\leq T \log K} e^{-at}|Z^K_t-x^K_t|\geq K^{\eta}\big) \leq 3 C K^{-\eta}\sup_{t\leq T}
    K^{-(a-r)t+\frac{\beta\vee (\beta-rt)\vee(c+(a-2r)t)}{2}} \label{eq:ineg-max}
\end{equation}

\medskip

\noindent{\bf Step 1: Doob's inequality.}
Fix $T$ satisfying~\eqref{cond:BPI-large-popu}. We consider first the case where we can find $\eta$ such that
\begin{equation}
\label{eq:condition-eta}
\frac{\beta\vee (\beta-rT)\vee(c+(a-2r)T)}{2}<\eta<\beta.
\end{equation}
Then, it follows from~\eqref{etape2} that the probability of the event $\Omega^K_1$ converges to 1, where
$$\Omega^K_1=\big\{\sup_{t\leq T \log K} \big|e^{-rt}(Z^K_t-x^K_t)\big|\leq K^{\eta}\big\}.$$
On this event, a similar computation as in Step 2 of Lemma~\ref{lem:BP} entails
\[
\sup_{t\leq T}\left|\frac{\log(1+Z^K_{t\log K})}{\log K}- \bar{\beta}_t\right|
\leq  \sup_{t\leq T}\left[\frac{1}{\log K} \frac{K^{-rt}|Z^K_{t\log K}-x^K_{t\log K}|}{K^{-rt}\left(x^K_{t\log K}\wedge Z^K_{t\log K}\right)}+\frac{\log\big(K^{-\bar{\beta}_t}(1+x^K_{t\log K})\big)}{\log K}\right].
\]
Because of Lemma~\ref{lem:esp-var}, we observe that $\bar{\beta}_t$ has been chosen such that, for all $t\leq T$,
\begin{equation}
  \label{eq:encadre-x}
  C^{-1}K^{\bar{\beta}_t}\leq x^K_{t\log K}\leq C K^{\bar{\beta}_t}
\end{equation}
for some constant $C>0$. Hence,
\[
\log\big(K^{-\bar{\beta}_t}(1+x^K_{t\log K})\big)\leq \log(C+K^{-\bar{\beta}_t})\leq C',
\]
where the last inequality follows from the assumption that $\inf_{t\leq T}\bar{\beta}_t>0$.

Hence, on the event $\Omega^K_1$,
\begin{align*}
\sup_{t\leq T}\left|\frac{\log(1+Z^K_{t\log K})}{\log K}- \bar{\beta}_t\right| &
\leq  \left[\frac{K^\eta}{\log K} \sup_{t\leq T}\frac{K^{rt}}{\left(x^K_{t\log K}-K^{\eta+rt}\right)_+}+\frac{C'}{\log K}\right] \\ &
\leq  \left[\frac{K^\eta}{\log K} \sup_{t\leq T}\frac{2K^{rt}}{x^K_{t\log K}}+\frac{C'}{\log K}\right], % \notag \\
\end{align*}
where we used~\eqref{eq:encadre-x} and the fact that $K^{\eta+rt}=o(x^K_{t\log K})$, since $\eta<\beta$. Hence,
\begin{align}
  \sup_{t\leq T}\left|\frac{\log(1+Z^K_{t\log K})}{\log K}- \bar{\beta}_t\right| & \leq C\left[\frac{K^\eta}{\log K}
    K^{-[\beta\vee\inf_{t\leq T}(c+(a-r)t)]}+\frac{1}{\log K}\right] \notag \\ &
    \leq
  C\left[\frac{K^{\eta-\beta}}{\log K}+\frac{1}{\log K}\right]. \label{etape3}
\end{align}
\medskip

\noindent \textbf{Case 1(a): $r\geq 0$ and $a\leq 2r$.} In this case, the constraint~\eqref{eq:condition-eta} becomes
$\beta/2<\eta<\beta$, hence we can choose $\eta=3\beta/4$ for any value of $T$ (note that~\eqref{cond:BPI-large-popu} is always
satisfied here). Hence~\eqref{etape3} implies that $\sup_{t\leq T}\left|\frac{\log(1+Z^K_{t\log K})}{\log K}- \bar{\beta}_t\right|$
converges to 0 on the event $\Omega^K_1$.
\medskip

\noindent \textbf{Case 1(b): $r< 0$ and $a\leq r$.}  In this case, Assumption~\eqref{cond:BPI-large-popu} is equivalent to $\inf_{t\leq
  T}(\beta+rt)>0$, i.e.\ $T<\beta/|r|$. Since $c+(a-2r)T\leq \beta-rT+(a-r)T\leq \beta-rT<2\beta$, it is possible to find $\eta$ such
that
\[
\frac{\beta\vee (\beta-rT)\vee(c-(r\wedge (2r-a))T)}{2}=\frac{\beta-rT}{2}<\eta<\beta
\]
and~\eqref{etape3} allows again to conclude.
\medskip

\noindent \textbf{Case 1(c): $r\geq 0$ and $a> 2r$.} In this case,~\eqref{eq:condition-eta} is satisfied provided $T$ is such that
$c+(a-2r)T<2\beta$, i.e.\ $T<T^*:=\frac{2\beta-c}{a-2r}$. Let us observe that the two lines $\beta+rt$ and $c+at$ intersect at time
$t^*=\frac{\beta-c}{a-r}$, and in our case, $T^*>t^*$ since $a\beta>rc$. Therefore, we can apply the computation~\eqref{etape3} to
$T\in(t^*,T^*)$ to obtain the convergence of $\sup_{t\leq T}\left|\frac{\log(1+Z^K_{t\log K})}{\log K}- \bar{\beta}_t\right|$ to 0.
We explain below (in step 3) how to conclude for any value of $T$ satisfying~\eqref{cond:BPI-large-popu}.
\medskip

\noindent \textbf{Case 1(d): $r<0$, $a>r$ and $c+a\beta/|r|\leq 0$.} In this case, $\beta+rt\geq c+at$ for all $t\leq \beta/|r|$ and
hence Assumption~\eqref{cond:BPI-large-popu} is satisfied if and only if $T<\beta/|r|$. For such $T$,~\eqref{eq:condition-eta} is
satisfied since $c+(a-2r)T<c+(a-2r)\frac{\beta}{|r|}=2\beta+c+a\frac{\beta}{|r|}\leq 2\beta$.
\medskip

\noindent \textbf{Case 1(e): $r<0$, $a>r$ and $c+a\beta/|r|>0$.} In this case, Condition~\eqref{eq:condition-eta} is satisfied
provided $T<T^*=\frac{\beta}{|r|}\wedge\frac{2\beta-c}{a-2r}$. Since $a\beta>rc$, we actually have $T^*=\frac{2\beta-c}{a-2r}$. In
addition, if we define $t^*=\frac{\beta-c}{a-r}$ the first time where the line $\beta+rt$ intersects the line $c+at$, we can see
exactly as in Case~1(c) that $T^*>t^*$. Hence, we can apply the computation~\eqref{etape3} to $T\in(t^*,T^*)$ to obtain the
convergence of $\sup_{t\leq T}\left|\frac{\log(1+Z^K_{t\log K})}{\log K}- \bar{\beta}_t\right|$ to 0. We explain below (in step 3)
how to conclude for any value of $T$ satisfying~\eqref{cond:BPI-large-popu}.

\medskip

\noindent \textbf{Step 2: maximal inequality.} We restrict here to the case $\beta=c$ and $a>r$. In this case,
\begin{align*}
  \sup_{t\leq T}\frac{(\beta-2(a-r)t)\vee (\beta-(2a-r)t)\vee(c-at)}{2} & =\sup_{t\leq T}\frac{(\beta-2(a-r)t)\vee(\beta-at)}{2} \\ &
  =\sup_{t\leq T}\frac{\beta-at}{2}=\frac{\beta\vee(\beta-aT)}{2},
\end{align*}
where the second equality comes from the fact that the maximum of $\beta-2(a-r)t$ is attained for $t=0$, and the function $\beta-at$
takes the same value at time $t=0$. Assuming
\begin{equation}
% & >\sup_{t\leq T}\frac{(\beta-2(a-r)t)\vee (\beta-(2a-r)t)\vee(\beta-at)}{2} \notag \\
% & =\sup_{t\leq T}\frac{(\beta-2(a-r)t)\vee(\beta-at)}{2} \notag \\
% & =
\frac{\beta\vee(\beta-aT)}{2}<\eta<\beta,
\label{eq:condition-eta-2}
\end{equation}
it follows from~\eqref{eq:ineg-max} that the probability of the event $\Omega^K_2$ converges to 1, where
$$\Omega^K_2=\big\{\sup_{t\leq T \log K} \big|e^{-at}(Z^K_t-x^K_t)\big|\leq K^{\eta}\big\}.$$
On this event, a similar computation as in Case 1 entails
\begin{align*}
  \sup_{t\leq T}\left|\frac{\log(1+Z^K_{t\log K})}{\log K}- \bar{\beta}_t\right| & \leq \left[\frac{K^\eta}{\log K} \sup_{t\leq
      T}\frac{K^{at}}{\left(x^K_{t\log K}-K^{\eta+at}\right)_+}+\frac{C'}{\log K}\right] \\ &
  \leq  \left[\frac{K^\eta}{\log K} \sup_{t\leq T}\frac{2K^{at}}{x^K_{t\log K}}+\frac{C'}{\log K}\right] % \\ &
  \leq C\left[\frac{K^{\eta-\beta}}{\log K} +\frac{1}{\log K}\right],
% \notag \\
% & \leq C\left[\frac{K^\eta}{\log K} K^{-[\beta\vee\inf_{t\leq T}(c+(a-r)t)]}+\frac{1}{\log K}\right] \notag \\ & \leq
% C\left[\frac{K^{\eta-\beta}}{\log K}+\frac{1}{\log K}\right]. \label{etape3}
\end{align*}
where we used the fact that $K^{\eta+at}=o(K^{\beta+at})$, since $\eta<\beta$ and that $x^K_{t\log K}\geq CK^{\beta+(a\vee r)t}=CK^{\beta+at}$
since $c=\beta$.
\medskip

\noindent \textbf{Case 2(a): $c=\beta$, $a>r$ and $a\geq 0$.} In this case, we can choose any $\eta\in(\beta/2,\beta)$ and deduce the
convergence of $\sup_{t\leq T}\left|\frac{\log(1+Z^K_{t\log K})}{\log K}- \bar{\beta}_t\right|$ to 0.

\medskip

\noindent \textbf{Case 2(b): $c=\beta$, $a>r$ and $a<0$.} In this case, Assumption~\eqref{cond:BPI-large-popu} is equivalent to
$\inf_{t\leq T}(\beta+(a\vee r)t)=\inf_{t\leq T}(\beta+at)>0$, which is satisfied if and only if $T<\beta/|a|$. Now, for such
$T$, one can find $\eta$ satisfying~\eqref{eq:condition-eta-2}, so we can again conclude.

\medskip

\noindent \textbf{Step 3: gluing parts together.} The remaining cases not covered by the previous steps are
\begin{itemize}
\item $r\geq 0$, $a>2r$, $c<\beta$ and $T\geq T^*=\frac{2\beta-c}{a-2r}$;
\item $r< 0$, $a>r$, $c<\beta$, $c+\frac{a\beta}{|r|}$ and $T\geq T^*=\frac{2\beta-c}{a-2r}$.
\end{itemize}
In both cases, we recall that $T^*>t^*$, where $t^*$ is the first time where $\beta+rt$ crosses $c+at$. So we can fix
$T_1\in(t^*,T^*)$ and apply Case 1(c) or Case 1(e) of Step 1 to obtain the convergence in probability in $L^\infty([0,T_1])$ of
$\log(1+Z^K_{t\log K})/\log K$ to $\bar{\beta}_t=(\beta+rt)\vee(c+at)$. In particular, for all $\varepsilon>0$, on an event
$\Omega^K_3$ with probability converging to 1, $K^{c+aT_1-\varepsilon}\leq Z^K_{T_1\log K}\leq K^{c+aT_1+\varepsilon}$ for $K$ large
enough.

Now, on $\Omega^K_3$, standard coupling arguments show that, for all $t\geq 0$, $\hat{Z}^K_t\leq Z^K_{T_1\log K+t}\leq \bar{Z}^K_t$,
where $\hat{Z}^K$ is a $BPI_K(b,d,a,c+aT_1-\varepsilon,c+aT_1-\varepsilon)$ and $\bar{Z}^K$ is a
$BPI_K(b,d,a,c+aT_1+\varepsilon,c+aT_1+\varepsilon)$. Hence, we can apply Case 2(a) or 2(b) of Step 2 to $\hat{Z}^K$ and $\bar{Z}^K$
to obtain the convergence of $\log(1+\hat{Z}^K_{t\log K})/\log K$ to $c-\varepsilon+a(T_1+t)$ and of $\log(1+\bar{Z}^K_{t\log
  K})/\log K$ to $c+\varepsilon+a(T_1+t)$ on $[0,T-T_1]$. Note that, in the case where $a<0$, Assumption~\eqref{cond:BPI-large-popu}
requires that $T<c/|a|$. We can choose $\varepsilon>0$ small enough to have $T<(c-\varepsilon)/|a|$. In this case, we have
$T-T_1<\frac{c+aT_1-\varepsilon}{|a|}$, so we can indeed apply Case 2(b) to $\hat{Z}^K$ on the time interval $[0,T-T_1]$.

Since $\varepsilon>0$ is arbitrary, the Markov property allows to conclude.
\hfill$\Box$\bigskip

Our next goal is to extend Theorem~\ref{lem:large-popu} to the case where $\beta=0$ or $\beta>0$ without
assuming~\eqref{cond:BPI-large-popu} nor $c\leq \beta$.

We first consider the case where $c>\beta$ in Lemma~\ref{lem:initial-growth}. It shows that the population
instantaneously (on the time scale $\log K$) reaches a population size of order $K^c$. We can use the Markov property at time
$\varepsilon\log K$ and comparison techniques in a similar way as in Step 3 of Theorem~\ref{lem:large-popu} to reduce the problem to
the case $c\leq \beta$.

\begin{lem}[Initial growth for strong immigration]\label{lem:initial-growth}
  Assume that $\beta<c$. Then, for all $\varepsilon>0$ and all $\bar{a}>|r|\vee|a|$,
  \[
  \lim_{K\rightarrow+\infty}\P\left(Z^K_{\varepsilon\log K}\in[ K^{c-\bar{a}\varepsilon},K^{c+\bar{a}\varepsilon}]\right)=1.
  \]
\end{lem}

\begin{proof}
  We give the proof in the case where $r\neq a$, $r\neq 0$ and $a\neq 2r$. The extension to the other cases is straightforward. Using
  Lemma~\ref{lem:esp-var}, there exists a positive constant $C$ such that
  \[
  C^{-1} K^{c-\varepsilon (|r|\vee|a|)}\leq \E(Z^K_{\varepsilon\log K})\leq C K^{c+\varepsilon (|r|\vee|a|)}
  \]
  and $\text{Var}(Z^K_{\varepsilon\log K})\leq C K^{c+(2|r|\vee|a|)\varepsilon}$. The result follows from Chebyshev's inequality.
\end{proof}

Let us now state our general result with $c\leq\beta$.

\begin{thm}\label{thm:BPI-general}
  Let $(Z^K_{t}, t\geq 0)$ be a $BPI_K(b,d,a,c,\beta)$ process with $c\leq\beta$ and assume either $\beta>0$ or $c\neq 0$. The process $\bigg(\frac{\log (1+ Z^K_{t\log K})}{\log K}, t> 0\bigg)$ converges when $K$ tends to
  infinity in probability in $L^\infty([0,T])$ for all $T>0$ to the continuous deterministic function $\bar{\beta}$ given by
  \begin{description}
  \item[\textmd{(i)}] if $\beta>0$,
    $\ \bar{\beta}\,:\,t\mapsto(\beta+rt)\vee(c+at)\vee 0$;
  \item[\textmd{(ii)}] if $\beta=0$, $c<0$ and $a>0$, $\ \bar{\beta}\,:\,t\mapsto((r\vee a)(t-|c|/a))\vee 0$;
  \item[\textmd{(iii)}] if $\beta=0$, $c<0$ and $a\leq 0$, $\ \bar{\beta}\,:\,t\mapsto 0$.
  \end{description}
  Note that, in case~(i), when $r\geq 0$, one can remove $\vee 0$ from the definition of $\bar{\beta}_t$.

  In addition, in the case where $c\neq 0$ or $a\neq 0$, for all compact interval $I\subset\RR_+$ which does not intersect the
  support of $\bar{\beta}$,
  \begin{equation}
    \label{eq:BPI-extinction}
    \lim_{K\rightarrow+\infty}\P\left(Z^K_{t\log K}=0,\ \forall t\in I\right)=1.
  \end{equation}
\end{thm}

The case $\beta=c=0$ could be deduced from the previous result using comparison argument, but is not useful here.\\

\begin{figure}[!ht]
\begin{center}
\begin{tabular}{ccc}
  \unitlength=0.25cm
   \begin{picture}(15,10)
      \put(1,2){\vector(1,0){13}} \put(13,1){time}
      \put(1,0.5){\vector(0,1){9.5}} \put(0,10.3){$\bar{\beta}$}
      \put(0.3,1.5){0}
      \put(0.1,8.5){1}
      \multiput(1,9)(0.4,0){30}{\line(1,0){0.2}}
      \put(1,1){\line(2,1){12}} \put(13,6){$c+at$}
      \put(2,1){\line(1,1){9}} \put(11.2,9.8){$r\big(t-\frac{|c|}{a}\big)$}
      \thicklines
      \put(1,2){\line(1,0){2}}
      \put(3,2){\line(1,1){7}}
      \put(10,9){\line(1,0){4}}
    \end{picture}
 &
\hspace{0.5cm}
  \unitlength=0.25cm
   \begin{picture}(15,10)
      \put(1,2){\vector(1,0){13}} \put(13,1){time}
      \put(1,0.5){\vector(0,1){9.5}} \put(0,10.3){$\bar{\beta}$}
      \put(0.3,1.5){0}
      \put(0.1,8.5){1}
      \multiput(1,9)(0.4,0){30}{\line(1,0){0.2}}
      \put(1,6){\line(2,-1){12}} \put(1.1,3){$c+at$}
      \put(1,8.5){\line(1,-1){8.5}} \put(4.1,6.3){$\beta+rt$}
      \thicklines
      \put(1,8.5){\line(1,-1){5}}
      \put(6,3.5){\line(2,-1){3}}
      \put(9,2){\line(1,0){3}}
    \end{picture}
 &
\hspace{0.5cm}
  \unitlength=0.25cm
   \begin{picture}(15,10)
      \put(1,2){\vector(1,0){13}} \put(13,1){time}
      \put(1,0.5){\vector(0,1){9.5}} \put(0,10.3){$\bar{\beta}$}
      \put(0.3,1.5){0}
      \put(0.3,8.5){1}
      \multiput(1,9)(0.4,0){30}{\line(1,0){0.2}}
      \put(7,1){\line(2,1){6}} \put(13,3){$c+at$}
      \put(1,8){\line(1,-1){7}} \put(3.5,6.5){$\beta+rt$}
      \thicklines
      \put(1,8){\line(1,-1){6}}
      \put(7,2){\line(1,0){2}}
      \put(9,2){\line(2,1){4}}
    \end{picture}
  \\
    (a): $c<\beta=0$, $0<a<r$ & (b): $0<c<\beta$, $r<a<0$ & (c): $c<0<\beta$, $r<0<a$
 \end{tabular}
\caption{{\small \textit{Illustration of Theorem \ref{thm:BPI-general}. (a): Initially $\bar{\beta}=0$, but thanks to immigration, the population is revived. Once this happens, the growth rate $r$ being larger than $a$, immigration have a negligible effect after time $|c|/a$. (b): After time $(\beta-c)/(a-r)$, the dynamics is driven by mutation before getting extinct when $\bar{\beta}_t=0$. (c): We observe a local extinction before the population is revived thanks to incoming mutations.}}}\label{fig:thm-main}
\end{center}
\end{figure}
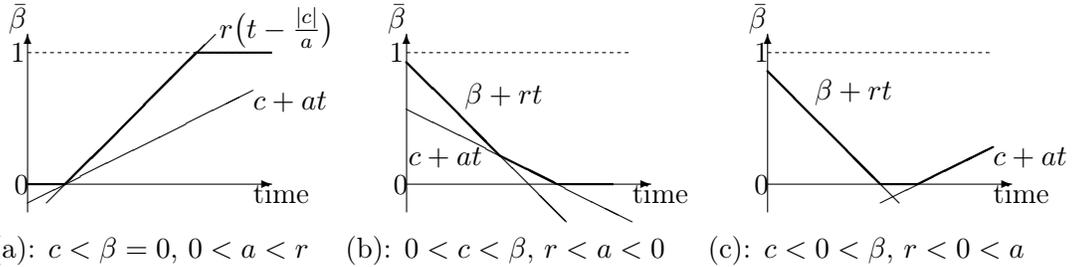

\begin{rem}
\label{remarqueeclairante}
Note that  point (i) in Theorem \ref{thm:BPI-general}  means that, at least for small $t>0$, either   $\, \beta=c$ and then $\overline \beta(t) = \beta+(r\vee a)t\,$
or
$\,\beta>c$ and then $\overline \beta(t) = \beta+rt$.  This explains  Equation~\eqref{rec-sigma} in Corollary~\ref{thm:transfer-main}.

\end{rem}

\begin{proof}
  The proof combines Theorem~\ref{lem:large-popu} with a series of lemmas and extensive use of the Markov property. The proofs of the
  lemmas are given at the end of the section.

  \medskip

  \noindent{\bf Proof of~(iii).}
  Theorem~\ref{thm:BPI-general}~(iii) follows directly from the next lemma. Note that it also proves that~\eqref{eq:BPI-extinction}
  holds true in case~(iii) for all $I\subset [0,T]$.

  \begin{lem}[Non emergence of any new population]\label{lem:non-emergence}
    Assume that $\beta=0$ and $c<0$. Let us consider $T>0$ such that $-|c|+aT<0$. Then:
    \begin{equation}
      \lim_{K\rightarrow +\infty}\P\Big( \ Z^K_t=0, \forall t\leq T\log K\Big)=1.
    \end{equation}
  \end{lem}

  \medskip

  \noindent{\bf Proof of~(ii).} We need to combine Theorem~\ref{lem:large-popu} and Lemma~\ref{lem:non-emergence} with
  the next two lemmas.

  \begin{lem}[Emergence of a new population]\label{lem:emergence}
    Assume that $\beta=0$, that $c=-\varepsilon$ with $\varepsilon>0$ and that $a>0$ (so that the immigration rate starts being positive at time
    $t_0=\frac{\varepsilon}{a} \log K$). Then, for all $\eta>(1\vee\frac{2r}{a})\varepsilon$,
    \begin{equation}
      \label{eq:lem-emergence}
      \lim_{K\rightarrow +\infty}\P\Big(K^{\varepsilon/2}-1\leq Z^K_{\frac{2\varepsilon}{a} \log K}\leq K^{\eta}-1\Big)=1.
    \end{equation}
  \end{lem}

  The second lemma is valid for all values of $c\leq \beta$, $r$ and $a$. It gives uniform estimates on the modulus of continuity of
  $\log(1+Z^K_{t\log K})/\log K$.

  \begin{lem}[continuity of the exponent]
    \label{lem:exposant-bouge-pas-trop}
    Let $(Z^K_t,t\geq 0)$ be a $BPI_K(b,d,a,c,\beta)$ with $c\leq\beta$. Then, there exists a constant $\bar{c}=\bar{c}(b,d,a)$ such
    that, for all $\varepsilon>0$,
    \begin{align*}
      \lim_{K\rightarrow +\infty}\P\Big(\forall t\in [0,\varepsilon\log K],\ K^{\beta-\bar{c}\varepsilon}-1\leq Z^K_{t}\leq
      K^{\beta+\bar{c}\varepsilon}-1\Big)=1.
    \end{align*}
  \end{lem}

  We proceed as follows. Fix $\varepsilon>0$ and apply
  Lemma~\ref{lem:non-emergence} on the time interval $[0,t_1]$ with $t_1=|c|/a-\varepsilon$. We deduce that
  \begin{equation}
    \label{eq:pf-technique-1}
    \frac{\log (1+
      Z^K_{t\log K})}{\log K}=0,\quad \forall t\in[0,t_1]
  \end{equation}
  on an event $\Omega_1^K$ with probability converging to 1 and, since $\bar{\beta}_t=0$ if and only if $t\leq
  |c|/a$,~\eqref{eq:BPI-extinction} is proved in case~(ii). Applying the Markov property at time $t_1\log K$ on $\Omega_1^K$. Since
  $Z^K_{t_1\log K}=0$ and $c+at_1<0$, we can apply Lemma~\ref{lem:emergence} to deduce that
  \[
  \frac{\log (1+ Z^K_{(t_1+\delta\varepsilon)\log K})}{\log K}\in(\underline{c}\varepsilon,\bar{c}\varepsilon)
  \]
  with probability converging to 1 for constants $\delta,\underline{c}>0$ and $\bar{c}<\infty$ independent of $\varepsilon$ and $K$.
  In addition, Lemma~\ref{lem:exposant-bouge-pas-trop} implies that, with probability converging to 1,
  \begin{equation}
    \label{eq:pf-contin-exponent}
    \sup_{t\in[t_1,t_1+\delta\varepsilon]}\frac{\log (1+ Z^K_{t\log K})}{\log K}\leq \bar{c}'\varepsilon
  \end{equation}
  for a constant $\bar{c}'>\bar{c}$ independent of $\varepsilon$ and $K$.

  Using the comparison trick of Step 3 of the proof of Theorem~\ref{lem:large-popu} after time $(t_1+\delta)\log K$, we can then
  apply Theorem~\ref{lem:large-popu} to prove that, with probability converging to 1, for all $t\in[t_1+\delta\varepsilon,T]$,
  \begin{align}
    \underline{c}\varepsilon+(r\vee a)(t-t_1-\delta\varepsilon) & \leq \liminf_{K\rightarrow+\infty}\frac{\log (1+ Z^K_{t\log
        K})}{\log K} \notag \\ & \leq \limsup_{K\rightarrow+\infty}\frac{\log (1+ Z^K_{t\log K})}{\log K}\leq
    \bar{c}\varepsilon+(r\vee a)(t-t_1-\delta\varepsilon).
    \label{eq:pf-technique-3}
  \end{align}
  We conclude combining~\eqref{eq:pf-technique-1},~\eqref{eq:pf-contin-exponent} and~\eqref{eq:pf-technique-3} and letting
  $\varepsilon\rightarrow 0$.

  \medskip

  \noindent{\bf Proof of~(i).}
  Note that, when $r\geq 0$, Point~(i) has been already proved in Theorem~\ref{lem:large-popu}. Similarly, if $r<0$, $a\geq 0$ and
  $c+a\beta/|r|>0$, Point~(i) also follows directly from Theorem~\ref{lem:large-popu}. In these two cases,~\eqref{eq:BPI-extinction}
  is also trivial. We divide our study of the remaining cases in four.

  \medskip

  \noindent{\bf Case (a): $r<0$, $a<0$ or $r<0$, $a=0$ and $c<0$.} In this case, we combine the previous lemmas with the next one, similarly to the proof of ~(ii). Note that we need to use Lemma~\ref{lem:extinction}~(i) if $c+a\beta/|r|<0$, Lemma~\ref{lem:extinction}~(ii) if
  $c+a\beta/|r|>0$. If $c+a\beta/|r|=0$, we use a comparison between $Z^K$ and a $BPI_K(b,d,a,c-\varepsilon,\beta)$ and a
  $BPI_K(b,d,a,c+\varepsilon,\beta\vee(c+\varepsilon))$, for which the previous cases apply, and let $\varepsilon\rightarrow 0$.

  \begin{lem}[Extinction]
    \label{lem:extinction}
    Assume that $r<0$.
    \begin{description}
    \item[\textmd{(i)}] Asume also that $c<0$ and $c+a\beta/|r|<0$. Then for all $\eta>0$ small enough,
      \begin{equation}
        \lim_{K\rightarrow +\infty}\P\Big(\forall t\in \Big[\big(\frac{\beta}{|r|}+\eta\big)\log
          K,\big(\frac{\beta}{|r|}+2\eta\big)\log K\Big], \ Z^K_t=0\Big)=1.
      \end{equation}
    \item[\textmd{(ii)}] Assume also that $c+a\beta/|r|>0$ and $a<0$. Then for all $\eta>0$ and $T>\eta$,
      \begin{equation}
        \lim_{K\rightarrow +\infty}\P\Big(\forall t\in \Big[\big(\frac{c}{|a|}+\eta\big)\log
          K,\big(\frac{c}{|a|}+T\big)\log K\Big], \ Z^K_t=0\Big)=1.
      \end{equation}
    \end{description}
  \end{lem}

  \noindent{\bf Case (b): $r<0$, $a>0$, $c<0$ and $\frac{\beta}{|r|}<\frac{|c|}{a}$.} This corresponds to Figure~\ref{fig:thm-main}~(c). In this case, we combine as above Theorem~\ref{lem:large-popu}, Lemma~\ref{lem:extinction}, Lemma~\ref{lem:non-emergence} and Lemma~\ref{lem:emergence}.
  \medskip

  \noindent{\bf Case (c): $r<0$, $a>0$, $c<0$ and $\frac{\beta}{|r|}=\frac{|c|}{a}$.} This can be treated using comparisons
  between $Z^K_t$ and $BPI_K(b,d,a,c-\varepsilon,\beta)$ and $BPI_K(b,d,a,c+\varepsilon,\beta\vee(c+\varepsilon))$ and letting
  $\varepsilon\rightarrow 0$.
  \medskip

  \noindent{\bf Case (d): $r<0$, $c=a=0$.} This can be treated using comparisons between $Z^K_t$ and
  $BPI_K(b,d,0,-\varepsilon,\beta)$ and $BPI_K(b,d,0,\varepsilon,\beta\vee\varepsilon)$ and letting $\varepsilon\rightarrow 0$.\end{proof}

\noindent{\bf Proof of Lemma~\ref{lem:non-emergence}.}
Since the rate of immigration is upper bounded by $K^{c\vee(c+aT)}$ on $[0,T\log K]$, the probability that a migrant arrives during
this time interval is upper bounded by $T K^{c\vee(c+aT)}\log K$ which tends to 0 when $K\rightarrow+\infty$. The lemma is proved.
\hfill$\Box$\bigskip

\noindent{\bf Proof of Lemma~\ref{lem:emergence}.}
The number of immigrant families which arrived during the time interval $[0,\frac{2\varepsilon}{a}\log K]$
  and which survived up to time $\frac{2\varepsilon}{a}\log K$ is, by thinning, a Poisson random variable with parameter
  \begin{align*}
    \lambda=\int_0^{\frac{2\varepsilon}{a}\log K}K^{-\varepsilon}e^{at}\frac{r K^{\frac{2r\varepsilon}{a}}e^{-rt}}{b K^{\frac{2r\varepsilon}{a}}e^{-rt}-d}dt.
  \end{align*}
  This formula is obtained by \eqref{eq:Sylvie}, where the probability of keeping a family immigrated at time $t$ is the fraction in the above expression.\\
  In the case where $r\geq 0$,
  \begin{align*}
    \lambda & \geq \frac{r}{b}\int_0^{\frac{2\varepsilon}{a} \log K}K^{-\varepsilon}e^{at}dt\geq\frac{rK^{\varepsilon}}{2ab}.
  \end{align*}
  In the case where $r<0$,
  \begin{align*}
    \lambda & \geq \frac{d-b}{d}K^{\frac{2r\varepsilon}{a}-\varepsilon}\int_0^{\frac{2\varepsilon}{a}\log K} e^{(a+|r|)t}dt\geq
    \frac{|r|K^\varepsilon}{2d(a+|r|)}.
  \end{align*}
  Therefore,
  \[
  \lim_{K\rightarrow +\infty}\P\left(Z^K_{\frac{2\varepsilon}{a} \log K} \geq K^{\varepsilon/2}\right)=1.
  \]
  For the upper bound, it follows from~\eqref{eq:esperance} that
  \begin{align*}
    \E\left(Z^K_{\frac{2\varepsilon}{a}\log K}\right)\leq
    \begin{cases}
      \frac{K^\varepsilon}{a-r} & \text{if }r<a, \\
      \frac{4\varepsilon}{a}K^\varepsilon\log K & \text{if }r=a,\\
      \frac{K^{\frac{2r}{a}\varepsilon}}{r-a} & \text{if }r>a.
    \end{cases}
  \end{align*}
  Therefore,~\eqref{eq:lem-emergence} follows from Markov's inequality and the choice $\eta>(1\vee\frac{2r}{a})\varepsilon$.
\hfill$\Box$\bigskip

\noindent{\bf Proof of Lemma~\ref{lem:exposant-bouge-pas-trop}.}
  The number of immigrant families which arrive during the time interval $[0,\varepsilon \log K]$
  and which survive up to time $\varepsilon\log K$ is a Poisson random variable with parameter
  \begin{align*}
    \lambda\leq \int_0^{\varepsilon\log K} K^c e^{at}dt\leq K^\beta \frac{K^{(|a|+1/2)\varepsilon}}{|a|+1/2}.
  \end{align*}
  The maximal size of each families (immigrant or present at time 0) on the time interval $[0,\varepsilon\log K]$ is bounded by the
  size at time $\varepsilon\log K$ of a Yule process with birth rate $b$, i.e.\ a geometric random variable $G_i$ with parameter
  $K^{b\varepsilon}$, independently for each immigrant families. Hence, with probability converging to 1,
  \begin{align}
    \label{eq:domination-Yule}
    \sup_{t\in [0,\varepsilon\log K]} Z^K_t\leq\sum_{i=1}^{K^{\beta+(|a|+3/4)\varepsilon}} G_i\leq K^{\beta+(|a|+1+b)\varepsilon}
  \end{align}

  For the lower bound, we observe that $Z^K_t$ is bigger than a linear pure death process $BP(0,d,\beta)$. For each of the $\lfloor
  K^\beta-1\rfloor$ initial individuals, the probability of survival up to time $\varepsilon\log K$ is $K^{-d\varepsilon}$, hence,
  with probability converging to 1, $\ \inf_{t\in [0,\varepsilon\log K]} Z^K_t\geq K^{\beta-2d\varepsilon}$. Hence the lemma is
  proved with $\bar{c}=(2d)\vee (|a|+b+1)$. \hfill$\Box$\bigskip

\noindent{\bf Proof of Lemma~\ref{lem:extinction}.} We first prove (i). Using the same argument as in the proof of
Lemma~\ref{lem:non-emergence}, we can prove that the probability of the event $\Gamma$ that a migrant arrives during the time interval
$[0,T\log K]$ converges to 0. Therefore, on the complementary event $\Gamma^c$ and on this time interval, the process $Z^K$ is a
$BP(b,d,K^\beta)$ and, using~\eqref{eq:Sylvie}, its extinction time $T_\text{ext}$ satisfies, for all $t\leq T\log K$,
\begin{align*}
  \PP(T_{\text{ext}}>t;\Gamma^c)\leq 1-\left(1-\frac{r e^{rt}}{b e^{rt}-d}\right)^{K^\beta}.
\end{align*}
Thus, for $t=\eta\log K$ with $\eta>\beta/|r|$, there exists a constant $C$ such that
\begin{align*}
  \PP(T_{\text{ext}}>\eta\log K)\leq\PP(\Gamma)+C\,K^{-|r|\eta+\beta}\xrightarrow[K\rightarrow+\infty]{} 0.
\end{align*}

Now, let us prove~(ii). Using the Markov property and Theorem~\ref{lem:large-popu}, we can assume without loss of generality that
$c\leq\beta<\eta|r|/4$. Note that assuming $c+a\beta/|r|>0$ implies that $a>r$. In addition, we can prove as above that, with a
probability converging to 1, there is no new immigrant arriving in the population between times $(c/|a|+\eta/2)\log K$ and
$(c/|a|+T)\log K$. Hence we only need to check that all the families that descend from each of the individuals initially present in
the population and of all immigrants which arrive before time $(c/|a|+\eta/2)\log K$ are extinct before time $(c/|a|+\eta)\log K$.
Using~\eqref{eq:Sylvie}, each of these families has a probability to survive longer than a time $(\eta/2)\log K$ which is smaller
than $2 K^{-\eta|r|/2}$ (for $K$ large enough). Since the number of these families is equal to $K^\beta$ plus a Poisson random
variable of parameter $K^c\int_0^{(c/|a|+\eta/2)\log K}e^{as}ds\leq K^c/a$, we deduce that it is smaller that $K^{\eta r/3}$ with
probability converging to 1. Hence, on this event, the probability that at least one of these families survives up to time
$(c/|a|+\eta)\log K$ is smaller than
\[
\left(1-2 K^{-\eta|r|/2}\right)^{K^{\eta r/3}}\xrightarrow[]{K\rightarrow 0} 0.
\]
This concludes the proof of Lemma~\ref{lem:extinction}.
\hfill$\Box$\bigskip

\section{Logistic birth and death process with  immigration}
\label{sec:logistic-BDP}

\subsection{One-dimensional case}
\label{sec:logistic-1d}

We consider here a one-dimensional logistic birth and death process with individual birth rate $b$ and individual death rate $d+Ck/K$
when the population size is $k$ and immigration at predictable rate $\gamma(t)\geq 0$ at time $t$. We denote by
$LBDI_K(b,d,C,\gamma)$ the law of this process. We consider a specific initial condition in the next result. Recall that $r=b-d$.

\begin{lem}\label{lem:sable-popu}
  Assume that $(Z^K_t,t\geq 0)$ follows the law $LBDI_K(b,d,C,\gamma)$ with $b>d$. Let $T>0$ and assume
  that $\gamma(t)\leq K^{1-\alpha}$ for all $t\in[0,T\log K]$ for some $\alpha>0$.
  \begin{description}
  \item[\textmd{(i)}] If  $\frac{Z^K_0}{K}\in\left[\frac{r}{C}-\varepsilon,\frac{r}{C}+\varepsilon\right]$ for some $\varepsilon>0$, then
  \[
  \lim_{K\rightarrow+\infty}\P\left(\forall t\in[0,T\log K],\ \frac{Z^K_t}{K}\in\left[\frac{r}{C}-2\varepsilon,\frac{r}{C}+2\varepsilon\right]\right)=1.
  \]
  \item[\textmd{(ii)}] For all $\varepsilon,\varepsilon'>0$, there exists $T(\varepsilon,\varepsilon')<+\infty$ such that for all initial condition
  $\frac{Z^K_0}{K}\geq \varepsilon $ we have that
  \[
  \lim_{K\rightarrow+\infty}\P\left( \frac{Z^K_{T(\varepsilon,\varepsilon')}}{K}\in\left[\frac{r}{C}-2\varepsilon',\frac{r}{C}+2\varepsilon'\right]\right)=1.
  \]
\end{description}
\end{lem}

\begin{proof}
  This result is related to the problem of exit from a domain of~[Freidlin-Wentzell] and can be proved with standard arguments as
  in~\cite{champagnat06,champagnatmeleard2011}. The only difficulty comes from the additional immigration rate (smaller than
  $K^{1-a}$ for all $t\in[0,T\log K]$), which is negligible with respect to the reproduction rate (of order $K$), but this can be
  handled for example adapting the proof of~\cite[Prop. 4.2]{champagnatjabinmeleard}.
\end{proof}

\subsection{Two-dimensional case}
\label{sec:logistic-2d}

\subsubsection{Transfer birth-death process with immigration}

We consider a two-dimensional transfer process with immigration $(Y^K_t,Z^K_t)_{t\geq 0}$, with transition
rates from $(n,m)\in\N^2$ to
\begin{align*}
  (n+1,m) & \text{ with rate }n b^K_1(\omega,t)+\gamma^K_1(\omega,t), \\
  (n-1,m) & \text{ with rate }n d^K_1(\omega,t), \\
  (n,m+1) & \text{ with rate }m b^K_2(\omega,t)+\gamma^K_2(\omega,t), \\
  (n,m-1) & \text{ with rate }m d^K_2(\omega,t), \\
  (n-1,m+1) & \text{ with rate }\tau^K(\omega,t)\frac{nm}{n+m},
\end{align*}
with $C>0$ and predictable rates $\tau^K,b^K_i,d^K_i,\gamma^K_i:\Omega\times\R_+\rightarrow\R_+$. Remark that transfer only occurs from
population $Y^K$ to population $Z^K$.

\noindent We denote by $TBDI_K(b^K_1,d^K_1,b^K_2,d^K_2,\tau^K,\gamma^K_1,\gamma^K_2)$ the law of $(Y^K,Z^K)$.

\begin{lem}\label{lem:competitionTBDI}
  Assume that $(Y^K_t,Z^K_t)_{t\geq 0}$ follows the law $TBDI_K(b^K_1,d^K_1,b^K_2,d^K_2,\tau^K,\gamma^K_1,\gamma^K_2)$ and that there
  exist constants $b_1,d_1,b_2,d_2,\tau>0$ such that for some $s>0$,
  \begin{multline}
    \label{eq:condition-CIRM}
    \sup_{t\in[0,s\log
      K]}\|b_1^K(t)-b_1\|+\|b_2^K(t)-b_2\|+\|d_1^K(t)-d_1\|+\|d_2^K(t)-d_2\| \\
    +\|\tau^K(t)-\tau\|+\|\gamma_1^K(t)\|+\|\gamma_2^K(t)\|
    \xrightarrow[K\rightarrow+\infty]{} 0
  \end{multline}
  in probability. Let $S:=r_2-r_1+\tau$.
  \begin{description}
  \item[\textmd{(i)}] Assume $S>0$, that $Y^K_0\geq K^\beta$ for some $\beta>0$ and that $\eta Y^K_0<Z^K_0< Y^K_0$ for some $\eta>0$. Then, there exists
    $T=T(\eta)<\infty$ and $\rho>0$ such that for $s>0$ small enough,
\begin{equation}
    \label{eq:lemme-TBDI1}
    \lim_{K\rightarrow+\infty}\P\left(Y^K_{s\log K}\leq  K^{-s\rho}Z^K_{s\log
        K} \right)=1.
  \end{equation}
  \item[\textmd{(ii)}] Assume $S<0$, that $Z^K_0\geq K^\beta$ for some $\beta>0$ and that $\eta Z^K_0<Y^K_0<Z^K_0$ for some $\eta>0$. Then, there exists
    $T=T(\eta)<\infty$ and $\rho>0$ such that for $s>0$ small enough,
\begin{equation}
    \label{eq:lemme-TBDI2}
    \lim_{K\rightarrow+\infty}\P\left(Z^K_{s\log K}\leq  K^{-s\rho}Y^K_{s\log
        K} \right)=1.
  \end{equation}
  \end{description}
\end{lem}

Lemma~\ref{lem:competitionTBDI}~(i) says that, if the population $Y^K$ is initially dominant and if the population $Z^K$ has a positive fitness and is initially not negligible, after a short time on the time-scale $\log K$, $Y^K$ becomes negligible with respect to $Z^K$. For the proof, we refer to Lemma~\ref{lem:competition}~(iii).

\subsubsection{Logistic transfer birth-death process with immigration}

We consider a two-dimensional logistic transfer process with immigration $(Y^K_t,Z^K_t)_{t\geq 0}$, with transition
rates from $(n,m)\in\N^2$ to
\begin{align*}
  (n+1,m) & \text{ with rate }nb^K_1(\omega,t)+\gamma^K_1(\omega,t), \\
  (n-1,m) & \text{ with rate }n \left[d^K_1(\omega,t)+\frac{C}{K}(n+m)\right], \\
  (n,m+1) & \text{ with rate }m b^K_2(\omega,t)+\gamma^K_2(\omega,t), \\
  (n,m-1) & \text{ with rate }m \left[d^K_2(\omega,t)+\frac{C}{K}(n+m)\right], \\
  (n-1,m+1) & \text{ with rate }\tau^K(\omega,t)\frac{nm}{n+m},
\end{align*}
with $C>0$ and predictable rates $\tau^K,b^K_i,d^K_i,\gamma^K_i:\Omega\times\R_+\rightarrow\R_+$. Remark that transfer only occurs from
population $Y^K$ to population $Z^K$.

\noindent We denote by $LTBDI_K(b^K_1,d^K_1,b^K_2,d^K_2,C,\tau^K,\gamma^K_1,\gamma^K_2)$ the law of $(Y^K,Z^K)$.

\begin{lem}[Competition]\label{lem:competition}
  Assume that $(Y^K_t,Z^K_t)_{t\geq 0}$ follows the law\\ $LTBDI_K(b^K_1,d^K_1,b^K_2,d^K_2,C,\tau^K,\gamma^K_1,\gamma^K_2)$ and that there
  exist constants $b_1,d_1,b_2,d_2,\tau>0$ such that for some $s>0$, the convergence
 \eqref{eq:condition-CIRM}
   is satisfied.
  \begin{description}
  \item[\textmd{(i)}] Assume that $r_1:=b_1-d_1>0$, $r_2:=b_2-d_2>0$, $S:=r_2-r_1+\tau>0$ and
    $\frac{Y^K_0}{K}\in\left[\frac{r_1}{C}-\varepsilon,\frac{r_1}{C}+\varepsilon\right]$ for some $\varepsilon>0$ and
    $\frac{Z^K_0}{K}\geq m\varepsilon$ for some $m>0$. Then, for all $\varepsilon'>0$, there exists
    $T=T(m,\varepsilon,\varepsilon')<\infty$ such that
    \[
    \lim_{K\rightarrow+\infty}\P\left(Y^K_{T}\leq \varepsilon' K,\
      \frac{Z^K_{T}}{K}\in\left[\frac{r_2}{C}-\varepsilon',\frac{r_2}{C}+\varepsilon'\right]\right)=1.
    \]
  \item[\textmd{(ii)}] Assume that $r_1>0$, $r_2>0$, $S<0$ and $\frac{Y^K_0}{K}\geq m\varepsilon$ for some $\varepsilon,m>0$ and
    $\frac{Z^K_0}{K}\in\left[\frac{r_2}{C}-\varepsilon,\frac{r_2}{C}+\varepsilon\right]$. Then, for all $\varepsilon'>0$ there exists
    $T=T(m,\varepsilon,\varepsilon')<\infty$ such that
    \[
    \lim_{K\rightarrow+\infty}\P\left(
      \frac{Y^K_{T}}{K}\in\left[\frac{r_1}{C}-\varepsilon',\frac{r_1}{C}+\varepsilon'\right],\
      Z^K_{T}\leq \varepsilon' K,\right)=1.
    \]
  \item[\textmd{(iii)}] Assume that $r_1>0$, $r_2<0$, $S>0$ and
    $\frac{Y^K_0}{K}\in\left[\frac{r_1}{C}-\varepsilon,\frac{r_1}{C}+\varepsilon\right]$ for some $\varepsilon>0$ and
    $\frac{Z^K_0}{K}\geq m\varepsilon$ for some $m>0$. Then there exists a constant $\rho>0$ such that,
    for all $s>0$ small enough,
  \begin{equation}
    \label{eq:lemme-extinction}
    \lim_{K\rightarrow+\infty}\P\left(Y^K_{s\log K}\leq K^{-s\rho}Z^K_{s\log
        K},\ Z^K_{s\log
        K}\leq K^{1-s\rho},\ \right)=1.
  \end{equation}
  \end{description}
\end{lem}

In Lemma~\ref{lem:competition}~(iii), the initial population $Y^K$ is resident and after invasion, the population $Z^K$ becomes dominant but not resident, as can be seen from the exponent $1-s\rho$ in~\eqref{eq:lemme-extinction}.

\medskip

\begin{proof}
  Let us first prove (i). By Condition~\eqref{eq:condition-CIRM}, the proofs of~\cite{billiardcolletferrieremeleardtran} or~\cite[Ch.
  11]{ethierkurtz} can be easily adapted to prove that, when $\frac{1}{K}(Y^K_0,Z^K_0)$ converges in probability to
  $(y_0,z_0)\in(0,\infty)^2$, the process $(Y^K,Z^K)$ converges in probability in $L^\infty_\text{loc}(\RR_+)$ to the solution
  $(y(t),z(t))$ of the dynamical system
  \begin{equation}
    \label{eq:syst-logistic+transfert}
    \begin{cases}
      \dot{y} & = y\left(r_1-C(y+z)\right)-\tau\frac{yz}{y+z} \\
      \dot{z} & = z\left(r_2-C(y+z)\right)+\tau\frac{yz}{y+z}.
    \end{cases}
  \end{equation}
  In our frequency-dependent case with constant competition, invasion implies fixation~\cite[Section
  3.3.1]{billiardcolletferrieremeleardtran}, so that, for all $(y_0,z_0)\in (0,\infty)^2$, $(y(t),z(t))$ converges when
  $t\rightarrow+\infty$ to $(0,r_2/C)$ (since $S>0$). The lemma follows from this result as in~\cite[Thm. 3(b)]{champagnat06}.

  The proof of (ii) follows the same arguments.

  Let us turn to the proof of (iii). First, we remark that, for $y(t)$, $z(t)$ solution to~\eqref{eq:syst-logistic+transfert},
  $p(t)=y(t)/z(t)$ solves $\dot{p}=-Sp$ when $(y,z)$ solves~\eqref{eq:syst-logistic+transfert}. Hence, if
  $y(0)\in\left[\frac{r_1}{C}-\varepsilon,\frac{r_1}{C}+\varepsilon\right]$ and $z(0)\geq m\varepsilon$,
  \[
  p(t)\leq\frac{r_1/C+\varepsilon}{m\varepsilon}e^{-St}.
  \]
  In particular, there exists $t_0$ such that, for all $t\geq t_0$, $y(t)\leq \frac{|r_2|}{2}z(t)$. Therefore, $\dot{z}/z\leq
  r_2+|r_2|/2=r_2/2<0$ for all $t\geq t_0$. Therefore, $(y(t),z(t))$ converges to $(0,0)$ for $t\rightarrow+\infty$. In addition, for
  any $\eta>0$, there exists $t_{\eta}$ large enough such that the solutions at time $t_{\eta}$ of the dynamical system issued from
  any initial condition in $\left[\frac{r_1}{C}-\varepsilon,\frac{r_1}{C}+\varepsilon\right]\times[m\varepsilon,\infty)$,
  belongs to a compact subset of ${\cal C}_{\eta}= \{(y,z)\in B(0,\eta);0<y/z<\eta\}$.

  We deduce that there exists a compact subset of ${\cal C}_{1.5 \eta}$ containing $(Y^K_{t_{\eta}}/K, Z^K_{t_{\eta}}/K)$ with a
  probability close to $1$ for $K$ large enough. In particular, there exists a constant $\kappa>0$ such that
  \begin{equation}
    \label{eq:borne-inf-Y-Z}
    \PP\left(
      Y^K_{t_{\eta}}>\kappa
      K,\  Z^K_{t_{\eta}}>\kappa K
    \right)\xrightarrow[K\rightarrow+\infty]{}1
  \end{equation}
 After time $t_{\eta}$, we can construct a coupling between the stochastic
  population processes as follows: for all $t$ smaller than $T^\text{exit}_{2\eta}$ the first exit time of $\text{Adh}(K {\cal C}_{2\eta})$
  by $(Y^K, Z^K)$,
  \[
  A^0_t\leq Y^K_{t+t_\eta}\leq A^1_t+I^1_t,\quad A^2_t\leq Z^K_{t+t_\eta}\leq A^3_t+I^3_t,
  \]
  where $A^0_0=A^1_0=Y^K_{t_\eta}$, $A^2_0=A^3_0=Z^K_{t_\eta}$, $A^0$ has law $BP(b_1,d_1+\tau+4C\eta)$, $A^1$ has law
  $BP(b_1,d_1+\tau/(1+2\eta))$, $A^2$ has law $BP(b_2,d_2+4C\eta)$, $A^3$ has law $BP(b_2+2\tau\eta/(1+2\eta),d_2)$, and the random
  variables $I^1_t$ and $I^3_t$ count, in $Y^K$ and $Z^K$ respectively, the number of individuals alive at time $t+t_\eta$ born from
  immigrant individuals which arrived in the population after time $t_\eta$. Note that all theses processes are not necessarily
  independent.

  Using domination of immigrant population by Yule processes as in the proof of Lemma~\ref{lem:exposant-bouge-pas-trop}, we can prove
  as in~\eqref{eq:domination-Yule} (with $\beta=0$) that, for $s>0$ small enough,
  \begin{equation}
    \label{eq:controle-immigrants}
  \lim_{K\rightarrow+\infty}\PP\left(\forall t\leq s\log K,\ I^1_t+I^3_t\leq K^{1-\alpha/2}\right)=1.
  \end{equation}
  Let us first prove that $T^\text{exit}_{2\eta}>s\log K$ with probability converging to 1. We first need to check that
  $(Y^K_t,Z^K_t)\in B(0,2\eta)$ for all $t\leq s\log K$ with probability converging to 1. Because of~\eqref{eq:controle-immigrants},
  it is enough to prove that
  \[
  \P_{\lfloor 1.5 \eta K\rfloor}(T_{ext}(A^1) < T_{2 \eta K}(A^1))\xrightarrow[K\rightarrow+\infty]{} 1
  \]
  and similarly for $A^3$, where $T_\text{ext}(A^i)$ is the extinction time of $A^i$ and $T_{M}(A^i)$ is the first time $t$
  such that $A^i_t\geq M$. It is classical to prove that, for $A\sim BP(b,d)$ with $b<d$,
  \[
  \PP_k(T_\text{ext}(A)< T_n(A))=\frac{\left(\frac{d}{b}\right)^n-\left(\frac{d}{b}\right)^k}{\left(\frac{d}{b}\right)^n-1},
  \]
  so the result follows.

  The second steps consists in proving that $Y^K_{t}\leq 2\eta Z^K_t$ for all $t\in[t_\eta,t_\eta+s\log K]$. Recall that $\kappa
  K\leq Y^K_{t_\eta}\leq 1.5\eta Z^K_{t_\eta}$ with high probability. Using~\eqref{eq:borne-inf-Y-Z}, we can apply
  Lemma~\ref{lem:BP} to obtain that, for $s>0$ small enough and all $\varepsilon'>0$, for all $t\leq s\log K$,
  \begin{equation}
    \label{eq:dominations}
    \begin{aligned}
      A^1_{t} & \leq (1+\varepsilon') A_{0}^1 e^{(r_{1}-\tau/(1+2\eta))t} \\
      A^2_{t} & \geq (1-\varepsilon') A_{0}^2 e^{(r_2-4C\eta)t}.
    \end{aligned}
  \end{equation}
  We choose $\eta>0$ and $\varepsilon'>0$ small enough so that $r_1-\tau/(1+2\eta)<r_2-4C\eta$ and
  $1.5(1+\varepsilon')<2(1-\varepsilon')$ . Hence, with probability converging to 1, for all $t\leq s\log K$,
  \[
  Y^K_{t+t_\eta}\leq A^1_t\leq 2(1-\varepsilon')\eta A^2_0 e^{(r_{1}-\tau/(1+2\eta))t}\leq 2 \eta Z^K_{t+t_\eta}.
  \]
  So we have proved that $T^\text{exit}_{2\eta}>s\log K$ with probability converging to 1.

  In addition, it follows from~\eqref{eq:dominations} that
  \[
  Y^K_{t+t_\eta}\leq 2\eta \exp\left((r_{1}-r_2-\tau/(1+2\eta)+4C\eta)t\right) Z^K_{t+t_\eta}.
  \]
  Introducing $0<\rho< -(r_{1}-r_2-\tau/(1+2\eta)+4C\eta)$ and choosing $t=s\log K-t_\eta$ and $\eta<\varepsilon/2$, we
  deduce~\eqref{eq:lemme-extinction}.
\end{proof}

\section{Algorithmic construction of the slopes of $\beta(t)$}
\label{sec:algo}

We defined in Theorem \ref{thm:transfer-main-new} the times $s_{k}$ of change of resident or dominant populations. For algorithmic purpose , it is useful to  characterize all the times $t_{k}$ where  the functions $\beta_{\ell}(t)$ change their slopes.
 The successive slopes are given in the next result which  stems from Corollary \ref{thm:transfer-main}. The next formula
 \eqref{eq:affreuse-2} is explained after the statement of the theorem. We define
\begin{equation}
    \label{eq:probleme-limite-annexe}
      \Sigma^0_\ell(t) :=  \Sigma_\ell(t)\ind_{\beta_\ell(t)>0\text{ or }(\beta_\ell(t)=0\text{ and }\beta_{\ell-1}(t)=\alpha)}
\end{equation}

\begin{thm}
  \label{cor:transfer-main}
  Under the same assumptions as in Theorem~\ref{thm:transfer-main-new}, the limit $\beta(t)$ is continuous and piecewise affine and can
  be constructed recursively as follows: assume that, for some $k\geq 0$, we have constructed times $0=t_0<t_1<\ldots< t_k<T_0$ and
  integers $\ell'_0,\ldots,\ell'_{k-1}$ such that $(\beta(t),t\in[t_i,t_{i+1}])$ is affine and $\ell^*(t)=\ell'_i$ for all
  $t\in[t_i,t_{i+1})$ for all $0\leq i\leq k-1$. {Recall that, since $t_k<T_0$, we have} $\text{Card}(\text{Argmax}_{0\leq \ell\leq L}
  \beta_\ell(t_i))\in\{1,2\}$ for all $0\leq i\leq k$.

  Then
  \begin{align}
    & t_{k+1} =t_k+
    \left(\inf\left\{\frac{\beta_{\ell'_k}(t_k)-\beta_\ell(t_k)}{{\Sigma}^0_\ell(t_k+)-\widetilde{S}_{t_{k}+}(\ell'_k\delta;\ell'_k\delta)};\ell\neq\ell'_k
        \text{ s.t.\
        }{\Sigma}^0_\ell(t_k+)>\widetilde{S}_{t_{k}+}(\ell'_k\delta;\ell'_k\delta)\right\}\right. \notag \\
    & \ \wedge\inf \Big\{\frac{\beta_\ell(t_k)}{-{\Sigma}^0_\ell(t_k+)};\ell\text{ s.t.\ }\beta_\ell(t_k)>0 \text{ and
      }{\Sigma}^0_\ell(t_k+)<0 \Big\} \notag\\
    & \ \wedge \inf\Big\{\frac{\beta_\ell(t_k)-\beta_{\ell-1}(t_k)+\alpha}{{\Sigma}^0_{\ell-1}(t_k+)- \widetilde{S}_{t_{k}+}(\ell\delta,\ell'_k\delta)\ind_{\beta_\ell(t_k)>0}};
        \ell\neq\ell'_k\text{ s.t.\ }\beta_\ell(t_k)>\beta_{\ell-1}(t_k)-\alpha
       \notag \\
    &\hskip 5.5cm \text{ and }{\Sigma}^0_{\ell-1}(t_k+)- \widetilde{S}_{t_{k}+}(\ell\delta,\ell'_k\delta) \ind_{\beta_\ell(t_k)>0}>0
      \Big\}\notag \\
      & \ \wedge \left.\inf \Big\{\frac{1-\beta_{\ell'_k}(t_k)}{{\Sigma}^0_{\ell'_k}(t_k+)} ; \ {\Sigma}^0_{\ell'_k}(t_k+)>0 \hbox{ and }  \beta_{\ell'_k}(t_k)<1\Big\}  \right) \label{eq:affreuse-2}
  \end{align}
  and, for all $t\in[0,t_{k+1}-t_k]$,
  \begin{equation}
    \label{eq:inductive-construction1}
    \beta_\ell(t+t_k)=\beta_\ell(t_k)+{\Sigma}^0_\ell(t_k+)\, t,\quad\forall 0\leq\ell\leq L.
  \end{equation}
\end{thm}

It follows from Corollary \ref{thm:transfer-main} that
$$\widetilde{S}_{t_{k}+} = \mathbbm{1}_{\{\beta_{\ell'_{k}}(t_{k})=1 ; \, \ell'_{k} \delta<3\}}\,S(y;x)+\mathbbm{1}_{\{\beta_{\ell'_{k}}(t_{k})<1 \hbox{ or } \ell'_{k} \delta > 3\}}\,\widehat{S}(y;x).
$$
and ${\Sigma}^0_\ell(t_k+)$ is defined from $\Sigma_\ell(t_k+)$ as in~\eqref{eq:probleme-limite-annexe}, and $\Sigma_\ell(t_k+)$ is defined
using \eqref{rec-sigma} replacing $\widetilde S_{t_{k}}$ by $\widetilde{S}_{t_{k}+}$.

\bigskip

Note that the slope in~\eqref{eq:inductive-construction1} only depends on the vector $\beta(t_k)$.
Hence these formula  make possible exact numerical simulation of $\beta(t)$.
In the
definition~\eqref{eq:affreuse-2} of $t_{k+1}$,
\begin{enumerate}
\item the first infimum corresponds to the first time where another component of $\beta(t)$
intersects $\beta_{\ell'_{k}}(t)$ (change of resident or dominant population),
\item the second one to the first time where a component  of $\beta(t)$ hits 0 (extinction of a subpopulation),
\item the third one
to the first time where immigration due to mutations becomes dominant in the dynamics of one coordinate of $\beta(t)$,
\item the fourth one  to the first time where $\beta_{\ell'_{k}}(t)$ hits $1$ when $\beta_{\ell'_{k}}(t_{k})<1$ (transition from dominant to resident for trait $\ell'_{k} \delta$).
\end{enumerate}
The smallest of the four infima in~\eqref{eq:affreuse-2} gives the nature of the next event.

The third infimum involves only traits $\ell \delta$ whose initial dynamics (after time $t_k$) is not driven by mutations, i.e.\ such
that $\beta_{\ell-1}(t_k)<\beta_\ell(t_k)+\alpha$. In this case,
\[
\Sigma^0_\ell(t_k+)=\widetilde S_{t_{k}+}(\ell\delta,\ell^*_k\delta) \ind_{\beta_\ell(t_k)>0}.
\]
When $\Sigma^0_{\ell-1}(t_k+)>\widetilde S_{t_{k}+}(\ell\delta,\ell^*_k\delta)\ind_{\beta_\ell(t_k)>0}$, the time where the curve
$\beta_\ell(t_k)+t\,\widetilde S_{t_{k}+}(\ell\delta,\ell^*_k\delta) \ind_{\beta_\ell(t_k)>0}$ crosses $\beta_{\ell-1}(t_k)+t\,
\Sigma^0_{\ell-1}(t_k+)-\alpha$ is then given by
$\frac{\beta_\ell(t_k)-\beta_{\ell-1}(t_k)+\alpha}{\Sigma^0_{\ell-1}(t_k+)-\widetilde S_{t_{k}+}(\ell\delta,\ell^*_k\delta)\ind_{\beta_\ell(t_k)>0}}$.

\section*{Funding}
N.C., S.M. and V.C.T. have been supported by the Chair ``Mod\'elisation Math\'ematique et
Biodiversit\'e'' of Veolia Environnement-Ecole Polytechnique-Museum National d'Histoire Naturelle-Fondation X. V.C.T. also acknowledges support from ANR CADENCE (ANR-16-CE32-0007).

{\footnotesize
\providecommand{\noopsort}[1]{}\providecommand{\noopsort}[1]{}\providecommand{\noopsort}[1]{}\providecommand{\noopsort}[1]{}\providecommand{\noopsort}[1]{}\providecommand{\noopsort}[1]{}\providecommand{\noopsort}[1]{}\providecommand{\noopsort}[1]{}\providecommand{\noopsort}[1]{}\providecommand{\noopsort}[1]{}\providecommand{\noopsort}[1]{}\providecommand{\noopsort}[1]{}

}


\begin{thebibliography}{10}

\bibitem{baltrus}
D.~Baltrus.
\newblock Exploring the costs of horizontal gene transfer.
\newblock {\em Trends in Ecology and Evolution}, 28:489--495, 2013.

\bibitem{billiardcolletferrieremeleardtran}
S.~Billiard, P.~Collet, R.~Ferri\`ere, S.~M\'el\'eard, and V.C. Tran.
\newblock The effect of competition and horizontal trait inheritance on
  invasion, fixation and polymorphism.
\newblock {\em Journal of Theoretical Biology}, 411:48--58, 2016.

\bibitem{billiardcolletferrieremeleardtran2}
S.~Billiard, P.~Collet, R.~Ferri\`ere, S.~M\'el\'eard, and V.C. Tran.
\newblock Stochastic dynamics for adaptation and evolution of microorganisms.
\newblock In V.~Mehrmann and M.~Skutella, editors, {\em Proceedings of 7th
  European Congress of Mathematics}, pages 527--552. European Mathematical
  Society, 2018.

\bibitem{boviercoquillesmadi}
A.~Bovier, L.~Coquille, and C.~Smadi.
\newblock Crossing a fitness valley as a metastable transition in a stochastic
  population model.
\newblock {\em Annals of Applied Probability}, 2019.

\bibitem{champagnat06}
N.~Champagnat.
\newblock A microscopic interpretation for adaptative dynamics trait
  substitution sequence models.
\newblock {\em Stochastic Processes and their Applications}, 116:1127--1160,
  2006.

\bibitem{champagnatjabinmeleard}
N.~Champagnat, P.-E. Jabin, and S.~M\'el\'eard.
\newblock Adaptive dynamics in a stochastic multi-resources chemostat model.
\newblock {\em Journal de Math\'ematiques Pures et Appliqu\'ees},
  101(6):755--788, 2014.

\bibitem{champagnatmeleard2011}
N.~Champagnat and S.~M\'{e}l\'{e}ard.
\newblock Polymorphic evolution sequence and evolutionary branching.
\newblock {\em Probability Theory and Related Fields}, 151(1-2):45--94, 2011.

\bibitem{dellacherie-meyer}
C.~Dellacherie and P.-A. Meyer.
\newblock {\em Probabilit\'es et potentiel}.
\newblock Collection Enseignement des Sciences. Hermann, 1975.

\bibitem{dercoleferrieregragnanirinaldi}
F.~Dercole, R.~Ferri\`ere, A.~Gragnani, and S.~Rinaldi.
\newblock Coevolution of slow-fast populations : Evolutionary sliding,
  evolutionary pseudo-equilibria, and complex red queen dynamics.
\newblock {\em Proceedings of the Royal Society of London B}, 273:983--990,
  2006.

\bibitem{durrettmayberry}
R.~Durrett and J.~Mayberry.
\newblock Travelling waves of selective sweeps.
\newblock {\em Annals of Applied Probability}, 21(2):699--744, 2011.

\bibitem{ethierkurtz}
S.N. Ethier and T.G. Kurtz.
\newblock {\em Markov Processus, Characterization and Convergence}.
\newblock John Wiley \& Sons, New York, 1986.

\bibitem{raul}
R.~Fernandez Lopez.
\newblock private communication.
\newblock 2014.

\bibitem{Getino2015}
M. Getino, D. Sanabria-Rios, R. Fernandez-Lopez, J. Sanchez-Lopez, A. Fernandez, N. Carballeira, F. de la Cruza. Synthetic fatty acids prevent plasmid-mediated horizontal gene transfer. mBio 6,e01032--15, 2015.

\bibitem{Ginty2013}
S.M. Ginty, L. Lehmann, S. Brown, D. Rankin.The interplay between relatedness and horizontal gene transfer drives the evolution of plasmid-carried public goods. Proceedings R. Soc.B 208 0400, 2013.


\bibitem{gyllenbergparvinen}
M.~Gyllenberg and K.~Parvinen.
\newblock Necessary and sufficient conditions for evolutionary suicide.
\newblock {\em Bulletin of Mathematical Biology}, 63(5):981--993, 2001.

\bibitem{Keeling2008}
P. Keeling, J. Palmer. Horizontal gene transfer in eukaryotic evolutio. Nat. Rev. Genet. 9, 605--618, 2008.

\bibitem{Meleardbook}
S.~M\'el\'eard.
\newblock {\em Mod\`eles al\'eatoires en {E}cologie et {E}volution}.
\newblock Springer, 2016.

\bibitem{Ochman2000}
H. Ochman, J. Lawrence, E. Groisman. Lateral gene transfer and the nature of bacterial innovatio. Nature 405, 299--304, 2000.

 \bibitem{Stewart1979}
F. Stewart, B. Levin. The population biology of bacterial plasmids: a priori conditions for the existence of conjugationally transmitted factors. Genetics 87, 209--228, 1977.

\end{thebibliography}
\end{document}